\documentclass[a4paper,11pt]{article}
\usepackage{graphicx}
\usepackage{amssymb}
\usepackage[french, english]{babel}
\usepackage[utf8]{inputenc}
\usepackage{amsmath,amsfonts,amssymb,amsthm}
\usepackage[T1]{fontenc}
\usepackage{fullpage}
\usepackage{hyperref}
\usepackage{mathrsfs}
\usepackage{relsize}
\usepackage{tikz}
\usepackage{bbm}
\usepackage{appendix}
\usepackage{pifont}
\usetikzlibrary{patterns}

\definecolor{darkgreen}{rgb}{0.0, 0.5, 0.0}

\newcommand{\eps}{\varepsilon}

\newcommand{\loc}{\mathrm{loc}}

\newcommand{\B}{\mathcal{B}}
\newcommand{\A}{\mathcal{A}}

\newcommand{\N}{\mathbb N}
\newcommand{\MM}{\mathbb M}
\newcommand{\M}{M}
\newcommand{\Z}{\mathbb Z}

\newcommand{\R}{\mathbb R}

\newcommand{\E}{\mathbb E}

\renewcommand{\P}{\mathbb P}

\newcommand{\expl}{\mathcal{E}}

\newcommand{\ff}{\mathbf{f}}

\newcommand{\vv}{\mathbf{v}}
\newcommand{\ww}{\mathbf{w}}
\newcommand{\q}{\mathbf{q}}
\newcommand{\kk}{\mathbf{k}}
\newcommand{\QQ}{\mathbf{Q}}
\newcommand{\cQ}{\mathcal{Q}}
\newcommand{\dq}{\mathrm{d} \q}
\newcommand{\VV}{\mathcal{V}}
\newcommand{\nnu}{\widetilde{\nu}}
\newcommand{\FF}{\mathbf{F}}
\newcommand{\Ver}{v}

\theoremstyle{definition}

\newtheorem{thm}{Theorem}

\newtheorem{defn}{Definition}%[section]
\newtheorem{rem}[defn]{Remark}

\newtheorem{prop}[defn]{Proposition}
\newtheorem{corr}[defn]{Corollary}

\newtheorem{lem}[defn]{Lemma}
\newtheorem{conj}[defn]{Conjecture}
\newtheorem{fact}[defn]{Fact}

\tikzstyle{every node}=[circle, draw, fill=black!50, inner sep=0pt, minimum width=4pt]
\tikzstyle{white}=[circle, draw, fill=black!0, inner sep=0pt, minimum width=4pt]
\tikzstyle{bigwhite}=[circle, draw, fill=black!0, inner sep=0pt, minimum width=10pt]
\tikzstyle{dual}=[circle, draw=blue, fill=black!0, inner sep=0pt, minimum width=4pt]
\tikzstyle{fat}=[circle, draw, fill=red!50, inner sep=0pt, minimum width=8pt]
\tikzstyle{fat_bis}=[circle, draw, fill=blue!50, inner sep=0pt, minimum width=8pt]
\tikzstyle{fat_ter}=[circle, draw, fill=green!50, inner sep=0pt, minimum width=8pt]
\tikzstyle{rouge}=[circle, draw, fill=red, inner sep=0pt, minimum width=7pt]
\tikzstyle{bleu}=[circle, draw, fill=blue, inner sep=0pt, minimum width=7pt]
\tikzstyle{petitrouge}=[circle, draw, fill=red, inner sep=0pt, minimum width=4pt]
\tikzstyle{petitbleu}=[circle, draw, fill=blue, inner sep=0pt, minimum width=4pt]
\tikzstyle{texte}=[draw=none, fill=none]

\title{\bf{Local limits of bipartite maps with prescribed face degrees in high genus}}
\author{Thomas \bsc{Budzinski}\footnote{University of British Columbia, \url{budzinski@math.ubc.ca}} \, and Baptiste \bsc{Louf}\footnote{
Uppsala universitet, \url{baptiste.louf@math.uu.se}}}

\begin{document}

\maketitle

\paragraph{Abstract.}
We study the local limits of uniform high genus bipartite maps with prescribed face degrees. We prove the convergence towards a family of infinite maps of the plane, the $\q$-IBPMs, which exhibit both a spatial Markov property and a hyperbolic behaviour. Therefore, we observe a similar local behaviour for a wide class of models of random high genus maps, which can be seen as a result of universality. Our results cover all the regimes where the expected degree of the root face remains finite in the limit. This follows a work by the same authors on high genus triangulations \cite{BL19}.

\section{Introduction}

\paragraph{Planar maps.}

Maps, i.e. gluings of polygons forming an orientable surface, have been the object of extensive research in the last decades, both from the combinatorial and probabilistic viewpoints. The most popular category of maps are planar maps, i.e. maps homeomorphic to the sphere. Their combinatorial study goes back to Tutte in the 60s, e.g.~\cite{Tutte63}, who gave explicit formulas for the enumeration of various classes of planar maps using a generating functions approach. More recently, bijective approaches have been developped such as the Cori--Vauquelin--Schaeffer bijection for quadrangulations~\cite{Sch98} and its generalization, the Bouttier--di Francesco--Guitter bijection~\cite{BDG04}.
 
On the other hand, much attention has been given in the last 20 years to asymptotic properties of large random planar maps picked uniformly in certain classes. These asymptotic properties are usually understood by proving the convergence of random maps in some sense when the size goes to infinity. Two different notions of limits are commonly used: scaling and local limits. Scaling limits, which we will not study in this work, consist in renormalizing the distances in order to build continuous objects. In particular, many discrete models are known to have the Brownian map as a scaling limit~\cite{LG11, Mie11, Mar16}. The theory of scaling limits of planar maps shares deep links with other random geometry models such as Liouville Quantum Gravity~\cite{MS16b}. On the other hand, local limits, on which the present work focuses, study the neighbourhood of a typical point in a map in order to obtain an infinite but discrete object in the limit. In the context of planar maps, this was first considered by Angel and Schramm who proved the convergence of large uniform triangulations towards the Uniform Infinite Planar Triangulation (UIPT)~\cite{AS03}. The study of the UIPT using Markovian explorations called \emph{peeling} explorations was then initiated by Angel~\cite{Ang03}. More general models have followed since, such as general planar maps with Boltzmann weights on the face degrees~\cite{St18}. The bipartite case, which will be of particular interest for us, is investigated in~\cite{Bud15, BC16}, see also~\cite{C-StFlour} for a complete survey.

\paragraph{Maps of higher genus.}
It seems natural to try to extend the combinatorial and probabilistic study of planar maps to maps of higher genus. On the combinatorial side, the enumeration of maps with any genus is a very rich topic, with links to irreducible representations of the symetric group and integrable hierarchies~\cite{MJD00,Okounkov00}. In particular, double recursions on both the size and the genus are known for the counting of maps, see~\cite{GJ08} for triangulations and \cite{Lo19} for general classes of bipartite maps. However, explicit enumeration formulas are lacking. Asymptotics can be obtained when the genus is fixed and the size goes to infinity~\cite{BC86}, but are still missing when the genus goes to infinity as well.

Similarly, on the probabilistic side, higher genus versions of random surface models have been constructed such as Brownian surfaces~\cite{Bet16} or Liouville quantum gravity on complex tori \cite{DRV16}. However, their behaviour when the genus goes to infinity is still poorly understood. Finally, a regime that is much easier to handle is the regime where the genus is not constrained, and the faces are simply glued uniformly at random \cite{Gam06, CP16, BCP19}. In this case, the genus is concentrated very close to its maximal possible value.

More recently, some progress was made in the study of high genus maps, namely when the genus grows linearly in the size of the map. In this case, the Euler formula shows that maps satisfy a discrete notion of "negative average curvature", which suggests that the neighbourhood of a typical vertex should look hyperbolic. The first category of maps that was investigated in this setting were uniform unicellular maps (i.e. maps with one face). See \cite{ACCR13} for the proof of local convergence to a supercritical Galton--Watson tree, and \cite{Ray13a} for the study of more global properties such as logarithmic diameter.

Shortly after, Curien introduced a one-parameter family of random hyperbolic triangulations of the plane \cite{CurPSHIT}, following the work of Angel and Ray in the half-planar case \cite{AR13}. More precisely, random maps of this family are called Planar Stochastic Hyperbolic Triangulations (PSHT) $\left( \mathbb{T}_{\lambda} \right)_{0<\lambda \leq (12\sqrt{3})^{-1}}$ and they are the only random triangulations satisfying the following spatial Markov property: for any finite triangulation $t$ with $|t|$ vertices in total and a hole of perimeter $p$, we have
\begin{equation}
\P \left( t \subset \mathbb{T}_{\lambda} \right) = C_{p}(\lambda) \lambda^{|t|}.
\end{equation}
In particular, such a triangulation exists if and only if $\lambda \in \left( 0, \frac{1}{12\sqrt{3}} \right]$. Except for the critical case $\lambda=\frac{1}{12\sqrt{3}}$ (which corresponds to the UIPT), these objects exhibit hyperbolic properties \cite{CurPSHIT, B18}. Benjamini and Curien conjectured in~\cite{CurPSHIT} that they are the local limits of uniform high genus triangulations. 

In a recent paper \cite{BL19}, the authors of the present work proved this conjecture. Asymptotics for the enumeration of high genus triangulations up to subexponential factors were derived as a byproduct.

\paragraph{Infinite Boltzmann Planar Maps.}
The goal of the present work is to generalize the results of~\cite{BL19} to a much wider family of models, where faces do not have to be triangles. For combinatorial reasons\footnote{More precisely, the enumeration results in the planar case are simpler for bipartite maps, and the recursion of~\cite{Lo19} holds only for bipartite maps.}, we will restrict ourselves to bipartite maps, which means that the face degrees have to be even. The limiting objects appearing in the limits are the \emph{Infinite Boltzmann Bipartite Planar Maps} (IBPM) introduced in~\cite[Appendix C]{B18these} as an analog of the PSHT for bipartite maps. We also refer to~\cite[Chapter 8]{C-StFlour} for the study of basic properties of these objects.

The IBPM are characterized by a spatial Markov property similar to the one satisfied by the PSHT. If $m$ is a finite map with one hole, we write $m \subset M$ if $M$ can be obtained by filling the hole of $m$, possibly with a map with a nonsimple boundary\footnote{More precisely, we use the sense introduced by Budd in~\cite{Bud15}, i.e. the sense corresponding to the "lazy" peeling process, as opposed to the one introduced by Angel in~\cite{Ang03}. See Section~\ref{subsec_univ_defns} for precise definitions.}. Let $\q=(q_j)_{j \geq 1}$ be a sequence of nonnegative numbers. An infinite random planar map $M$ is called a $\q$-IBPM if there are numbers $\left( C_p \right)_{p \geq 1}$ such that, for any finite map $m$ with one hole of perimeter $2p$, we have
\[ \P \left( m \subset M \right) = C_p \times \prod_{f \in m} q_{\deg(f)/2}, \]
where the product is over all internal faces of $m$. In particular, they generalize the infinite \emph{critical} bipartite Boltzmann maps defined and studied by Budd in \cite{Bud15}. It was proved in~\cite[Appendix C]{B18these} that there is at most one $\q$-IBPM, that we denote by $\mathbb{M}_{\q}$. Moreover \cite[Appendix C]{B18these} provides both necessary conditions and sufficient ones on $\q$ for the existence of $\mathbb{M}_{\q}$, but no explicit characterization (these results are recalled in Section~\ref{subsec_IBPM} below). The present work improves on these results by giving an explicit parametrization of the weight families $\q$ for which $\M_{\q}$ exists. More precisely, as stated in Theorem~\ref{thm_prametrization_rootface} below, such families $\q$ can be parametrized by the law of the degree of the root face of $\MM_{\q}$, which may be any law on $\{2,4,6,\dots\}$, and an additional hyperbolicity parameter $\omega \in [1, +\infty)$. The critical maps of~\cite{Bud15} correspond to the case $\omega=1$, and are already known to be local limits of \emph{planar} maps. On the other hand, the case $\omega>1$ has a hyperbolic flavour.

\paragraph{Local limits of high genus bipartite maps.}
The main result of this work is that uniform maps with high genus and prescribed face degrees converge locally to the $\q$-IBPM when the size goes to infinity. This can be seen as a \textit{universality} result in the domain of high genus maps, in the sense that regardless of the precise model of maps, the same phenomenon is observed.

More precisely, we will use the notation $\mathbf{f}=(f_j)_{j \geq 1}$ for face degree sequences (i.e. $f_j \geq 0$ for all $j$, and $f_j=0$ eventually). For such a sequence $\mathbf{f}$, we set $|\mathbf{f}|=\sum_{j \geq 1} j f_j$, which describes the number of edges of a map with $f_j$ faces of degree $2j$ for all $j \geq 1$. For $g \geq 0$, we also write
\begin{equation}\label{defn_v_f_g}
v(\mathbf{f}, g) = 2-2g+\sum_{j \geq 1} (j-1)f_j.
\end{equation}
By the Euler formula, a bipartite map with genus $g$ and face degrees described by $\mathbf{f}$ exists if and only if $v(\mathbf{f}, g) \geq 2$, and in this case $v(\mathbf{f}, g)$ is the number of vertices of such a map. For such $\mathbf{f}$ and $g$, we denote by $M_{\ff,g}$ a uniform bipartite map with genus $g$ and $f_j$ faces of degree $2j$ for all $j \geq 1$.

\begin{thm}\label{univ_main_thm}
Let $(\mathbf{f}^{n})_{n \geq 1}$ be a sequence of face degree sequences, and let $(g_n)$ be a sequence such that $v(\mathbf{f}^{n}, g_n) \geq 2$ for all $n \geq 1$. We assume that $|\mathbf{f}^n| \to +\infty$ when $n \to +\infty$ and that $\frac{f^n_j}{|\mathbf{f}^n|} \to \alpha_j$ for all $j \geq 1$, where $\sum_{j \geq 1} j \alpha_j=1$.
We also assume $\frac{g_n}{|\mathbf{f}^n|} \to \theta$, where $0 \leq \theta < \frac{1}{2} \sum_{j \geq 1} (j-1) \alpha_j$.
Finally, assume that $\sum_{j \geq 1} j^2 \alpha_j  <+\infty$.

Then we have the convergence in distribution
\[ M_{\ff^n, g_n} \xrightarrow[n \to +\infty]{(d)} \MM_{\q}\]
for the local topology, where the weight sequence $\q$ depends only on $\theta$ and $\left( \alpha_j \right)_{j \geq 1}$, in an injective way.
\end{thm}

Let us now make a few comments on the various assumptions of the main theorem.
\begin{itemize}
\item[$\bullet$]
We first note that the assumption that $\sum_{j \geq 1} j \alpha_j=1$ means that the proportion of the edges that are incident to a face with degree larger than $A$ goes to $0$ as $A \to +\infty$, uniformly in $n$. This is equivalent to saying that the root face stays almost surely finite in the limit, so this assumption is necessary to obtain a local limit with finite faces. If this assumption is waived, we expect to obtain different limit objects with infinitely many infinite faces.
\item[$\bullet$]
The assumption $\theta < \frac{1}{2} \sum_{j \geq 1} (j-1) \alpha_j$ means that the number of vertices is roughly proportional to $|\ff^n|$, so that the average degree of the vertices stays bounded. Therefore, it is also necessary in order to have a proper local limit with finite vertex degrees. Note that this assumption also implies $\alpha_1<1$, i.e. it is not possible that almost all faces are $2$-gons.
\item[$\bullet$]
The assumption $\sum_{j \geq 1} j^2 \alpha_j  <+\infty$ means that the \emph{expected} degree of the root face stays finite in the limit. We do not expect this assumption to be necessary. However, one of the steps of our proof (the "two-holes argument" of Section~\ref{sec_arg_deux_trous}) crucially requires a bound on the tail of the degrees of the faces.
\end{itemize}
Finally, the application associating the weight sequence $\q$ given by Theorem~\ref{univ_main_thm} to $\left( \theta, \left( \alpha_j \right)_{j \geq 1}\right)$ is surjective in the sense that every IBPM $\MM_{\q}$ for which the degree of the root face has finite expectation can be obtained as a local limit through Theorem~\ref{univ_main_thm}.

\paragraph{The heavy tail case.}
Although we could not remove the assumption $\sum_{j \geq 1} j^2 \alpha_j$ in Theorem~\ref{univ_main_thm}, most of the steps of the proof do not require this assumption. In particular, we can still obtain the following partial result in the general case.

\begin{thm}\label{thm_main_more_general}
Let $(\mathbf{f}^{n})_{n \geq 1}$ be a sequence of face degree sequences, and let $(g_n)$ be a sequence such that $v(\mathbf{f}^{n}, g_n) \geq 2$ for all $n \geq 1$. We assume that $|\mathbf{f}^n| \to +\infty$ when $n \to +\infty$ and that $\frac{f^n_j}{|\mathbf{f}^n|} \to \alpha_j$ for all $j \geq 1$, where $\sum_{j \geq 1} j \alpha_j=1$.
We also assume $\frac{g_n}{|\mathbf{f}^n|} \to \theta$, where $0 \leq \theta < \frac{1}{2} \sum_{j \geq 1} (j-1) \alpha_j$.

Then the sequence of random maps $\left( M_{\ff^n, g_n} \right)_{n \geq 1}$ is tight for the local topology. Moreover, all its subsequential limits are of the form $\MM_{\QQ}$, where $\QQ$ is a random Boltzmann weight sequence.
\end{thm}

\paragraph{A parametrization of Infinite Bipartite Boltzmann Planar Maps.}
As explained briefly above, we also provide a new parametrization of the Boltzmann weight families $\q$ associated to an IBPM: instead of directly using the Boltzmann weights $q_j$, we parametrize them according to the law of the degree of the root face.

\begin{thm}\label{thm_prametrization_rootface}
Let $\boldsymbol{\alpha}$ be a probability measure on $\N^*$. Then the set of IBPM for which the half-degree of the root face has law $\boldsymbol{\alpha}$ forms a one parameter family $\left( \MM_{\q^{(\omega)}} \right)_{\omega \geq 1}$. Moreover, $\q^{(\omega)}$ is critical if and only if $\omega=1$, and the vertex degrees in $\MM_{\q^{(\omega)}}$ go to infinity when $\omega \to +\infty$.
\end{thm}

In particular, the existence of hyperbolic Boltzmann maps with arbitrarily heavy-tailed face degrees answers a question from~\cite{C-StFlour}, that was not settled by the results of \cite[Appendix C]{B18these}. Moreover, we can think of $\q^{(\omega)}$ as interpolating between a critical non-hyperbolic map, and a degenerate map with infinite vertex degrees.

\paragraph{Asymptotic enumeration.}
Like for triangulations, the most natural way to try to prove Theorem~\ref{univ_main_thm} would be to obtain precise asymptotics on the counting of maps with prescribed genus and face degrees, in order to mimic classical arguments going back to~\cite{AS03}. However, such asymptotics are not available and seem difficult to obtain. On the other hand, just like in~\cite{BL19} for triangulations, once Theorem~\ref{univ_main_thm} is proved, applying the arguments of~\cite{AS03} "backwards" allows to obtain a result about convergence of the ratio when we add one face of fixed degree. We denote by $\beta_g(\mathbf{f})$ the number of bipartite maps of genus $g$ with face degrees prescribed by $\mathbf{f}$.

\begin{corr}\label{prop_cv_ratio}
Let $\left( \mathbf{f}^{n} \right)_{n \geq 0}$ and $\left( g_n \right)_{n \geq 0}$ be such that $\frac{g_n}{n}\to\theta$ and $\frac{f^n_j}{|\mathbf{f}^n|}\to \alpha_j$ for all $j \geq 1$. We assume that $\sum j \alpha_j=1$, that $0 \leq \theta < \frac{1}{2} \sum (j-1) \alpha_j$ and that $\sum j^2 \alpha_j < +\infty$. We recall that by Theorem~\ref{univ_main_thm}, there is a weight sequence $\q$ such that $M_{\mathbf{f}^{n},g_n}$ converges locally to $\mathbb{M}_{\q}$. Then for all $j \geq 1$, we have
\begin{equation}\label{eq_cv_ratio}
\frac{\beta_{g_n}(\mathbf{f}^{n}-\mathbf{1}_j)}{\beta_{g_n}(\mathbf{f}^{n})} \xrightarrow[n \to +\infty]{} C_2(\q)q_j.
\end{equation}
\end{corr}

We also believe that the following is true:

\begin{conj}\label{thm_univ_asympto}
Let $(\mathbf{f}^{n})_{n \geq 1}$ be a sequence of face degree sequences, and let $(g_n)$ be a sequence such that $v(\mathbf{f}^{n}, g_n) \geq 2$ for all $n \geq 1$. We assume that $|\mathbf{f}^n| \to +\infty$ when $n \to +\infty$ and that $\frac{f^n_j}{|\mathbf{f}^n|} \to \alpha_j$ for all $j \geq 1$, where $\sum_{j \geq 1} j \alpha_j=1$.
We also assume $\frac{g_n}{|\mathbf{f}^n|} \to \theta$, where $0 \leq \theta \leq \frac{1}{2} \sum_{j \geq 1} (j-1) \alpha_j$.
Finally, assume that $\sum_{j \geq 1} j^2 \alpha_j  <+\infty$.
Then
\[\beta_{g_n}(\mathbf{f}^{n})= |\mathbf{f}^n|^{2g_n} \exp \left( \varphi \left( \theta,\left( \alpha_j \right)_{j \geq 1} \right) |\mathbf{f}^n| + o \left( |\mathbf{f}^n| \right) \right),\]
where $\varphi$ is some function.
\end{conj}

More precisely, in~\cite{BL19}, the proof consists of first using the analog of Corollary~\ref{prop_cv_ratio} to estimate the ratio between triangulations with any genus and triangulations with a genus close to maximal (say with $\eps |\mathbf{f}|$ vertices). To count such triangulations, we contracted a spanning tree to reduce the problem to triangulations with only one vertex, for which explicit formulas are known. This "contraction" is the step that is difficult to extend to our setting here: while for triangulations we simply obtained a triangulation with less faces, here the impact on the face degrees may become much more complex. This is why we leave the question as open.

\paragraph{Sketch of the proof of Theorem~\ref{univ_main_thm}: common points and differences with the triangular case.}
The proof is a combination of combinatorial and probabilistic ideas. It follows the same global strategy as in \cite{BL19}, which shows the robustness of the approach of~\cite{BL19}. However, new difficulties arise at each of the steps, which makes the overall proof much longer. 

More precisely, the first step consists of showing the tightness of $M_{\ff^n, g_n}$. This follows from a \emph{bounded ratio lemma} (Lemma~\ref{lem_BRL}), stating that under certain assumptions the ratio $\frac{\beta_g \left(\ff+\mathbf{1}_j \right)}{\beta_g \left(\ff \right)}$ is bounded. As in~\cite{BL19}, this lemma is established using surgery operations to remove a face, but this surgery can affect a larger region than in the triangular case, which makes it more elaborated. The second step is to prove that any subsequential limit is planar and one-ended. This relies on the recurrence proved by the second author in~\cite{Lo19}, and only requires minor adaptations compared to~\cite{BL19}. We then notice that any subsequential limit enjoys a weak spatial Markov property, which implies that it must be of the form $\MM_{\mathbf{Q}}$, for some random weight sequence $\QQ$. This part is also similar to~\cite{BL19}, although additional technicalities arise. All these arguments do not use any assumption on the tail of the degrees of the faces, and prove Theorem~\ref{thm_main_more_general}.

The end of the proof consists in showing that $\mathbf{Q}$ is actually deterministic. As in~\cite{BL19}, this step relies on a surgery argument called the \emph{two holes argument}, for which we need to explore two pieces of maps with the exact same boundary length. This is where the assumption $\sum j^2 \alpha_j<+\infty$ is crucial: without it, when we explore a piece of map "face by face", the perimeter makes large positive jumps and misses too many values. Another major difference with \cite{BL19} is in the last step, where we match the average degree in finite models (computed with the Euler formula) and in infinite ones. In particular, we need to argue that a weight sequence $\q$ is determined by the law of the root face of $\mathbb{M}_{\q}$ and the average vertex degree. While for triangulations we were able to obtain an explicit formula for the average vertex degree, this is not the case here. Therefore, we need to develop new arguments making use of the local limit results obtained earlier in the paper. This is also the reason why the link between $\theta$, $\left( \alpha_j\right)$ and $\q$ in Theorem~\ref{univ_main_thm} is not explicit.

\paragraph{Weakly Markovian bipartite maps.}
Just like in~\cite{BL19}, the argument showing that a subsequential limit is a mixture of IBPM is a result of independent interest, so we give its statement here. Let $M$ be a random infinite, one-ended, bipartite planar map. We call $M$ \emph{weakly Markovian} if for any finite map $m$ with one hole, the probability $m \subset M$ only depends on the perimeter of $m$ and on the family of degrees of its internal faces. We denote by $\mathcal{Q}_h$ the set of weight sequences $\q$ for which $\mathbb{M}_{\q}$ exists, and by $\mathcal{Q}_f \subset \mathcal{Q}_h$ the set of those $\q$ for which the expected degree of the root face in $\mathbb{M}_{\q}$ is finite (this will be useful to handle the last assumption in Theorem~\ref{univ_main_thm}).

\begin{thm}\label{thm_weak_Markov_general}
Let $M$ be a weakly Markovian infinite, one-ended, bipartite random planar map. Then there is a random weight sequence $\QQ \in \mathcal{Q}_h$ such that $M$ has the same distribution as $\MM_{\QQ}$. Moreover, if the degree of the root face of $M$ has finite expectation, then $\QQ \in \mathcal{Q}_f$ almost surely.
\end{thm}

\paragraph{Structure of the paper.}
In Section~\ref{sec_univ_prelim}, we review basic definitions on maps, and previous combinatorial results that will be used in all the paper. We also introduce the IBPM and describe various parametrizations of the set of IBPMs, and in particular prove Theorem~\ref{thm_prametrization_rootface}. In Section~\ref{sec_univ_tight}, under the assumptions of Theorem~\ref{thm_main_more_general} (i.e. without the assumption on the tail of face degrees), we prove that the maps $M_{\ff^n, g_n}$ are tight for the local topology, and that any subsequential limit is a.s. planar and one-ended. Section~\ref{sec_univ_markov} is devoted to the proof of Theorem~\ref{thm_weak_Markov_general}, which implies that any subsequential limit of $M_{\ff^n, g_n}$ is an IBPM with random parameters. This is sufficient to prove Theorem~\ref{thm_main_more_general}. In Section~\ref{sec_univ_end}, we conclude the proof of Theorem~\ref{univ_main_thm} by showing that the parameters are deterministic and depend only on $\theta$ and $\left( \alpha_j \right)_{j \geq 1}$. In Section~\ref{sec_univ_asymp}, we deduce the combinatorial estimate of Corollary~\ref{prop_cv_ratio} from Theorem~\ref{univ_main_thm}. Finally, the Appendices contain the proofs of some technical results.

\tableofcontents

\newpage

\section*{Index of notations} \label{sec_univ_index}

\addcontentsline{toc}{section}{\protect\numberline{}Index of notations}

In general, we will we will use lower case letters such as $m$ to denote deterministic objects or quantities, upper case letters such as $M$ for random objects and $\mathtt{mathcal}$ letters such as $\mathcal{M}$ for sets of objects. We will use $\mathtt{mathbf}$ characters such as $\mathbf{q}$ for sequences, and normal characters such as $q_j$ for their terms.

\begin{itemize}
\item
$\mathbf{f}=(f_j)_{j \geq 1}$: denotes a face degree sequence ($\mathbf{F}$ will denote a random face degree sequence).
\item
$g$: will denote the genus.
\item
$\mathcal{B}_{g}(\mathbf{f})$: set of finite bipartite maps with genus $g$ and $f_j$ faces of degree $2j$ for all $j \geq 1$.
\item
$\beta_{g}(\mathbf{f})$: cardinality of $\mathcal{B}_{g}(\mathbf{f})$.
\item
$M_{\mathbf{f},g}$: uniform random map in $\mathcal{B}_{g}(\mathbf{f})$.
\item
$|\mathbf{f}|=\sum_{j \geq 1} j f_j$ (i.e. the number of edges of a map in $\mathcal{B}_{g}(\mathbf{f})$).
\item
$v(\mathbf{f}, g) = 2-2g+\sum_{j \geq 1} (j-1)f_j$ (i.e. the number of vertices of a map in $\mathcal{B}_{g}(\mathbf{f})$).
\item
$\overline{\B}$: space of finite or infinite bipartite maps with finite vertex degrees, equipped with the local distance $d_{\loc}$.
\item
$\overline{\B}^*$: space of finite or infinite bipartite maps with finite or infinite vertex degrees, equipped with the dual local distance $d_{\loc}^*$.
\item
$\theta$: limit value of $\frac{g}{|\mathbf{f}|}$ when $|\mathbf{f}| \to +\infty$.
\item
$\q=(q_j)_{j \geq 1}$: denotes a weight sequence ($\QQ$ denotes a random weight sequence).
\item
$\MM_{\q}$: the infinite bipartite Boltzmann planar map with weight sequence $\q$.
\item
$W_p(\q)$: partition function of finite Boltzmann bipartite maps of the $2p$-gon with weights $\q$.
\item
$\cQ=[0,1]^{\N^*}$.
\item
$\cQ^*=\left\{ \q=(q_j)_{j \geq 1} \in \cQ | \exists j \geq 2, q_j>0 \right\}$.
\item
$\cQ_a$: set of admissible families of Boltzmann weights $\q$, i.e. such that $W_p(\q)<+\infty$.
\item
$\cQ_h$: set of Boltzmann weights for which $\MM_{\q}$ exists. We have $\cQ_h \subset \cQ_a \cap \cQ^*$.
\item
$\cQ_f$: set of Boltzmann weights $\q \in \cQ_h$ for which the expectation of the degree of the root face of $\MM_{\q}$ is finite.
\item
$c_{\q}$: for $\q \in \cQ_a$, denotes the solution of the equation \[\sum_{j \geq 1} q_j \frac{1}{4^{j-1}} \binom{2j-1}{j-1} c_{\q}^{j-1} = 1-\frac{4}{c_{\q}}.\]
\item
$\nu_{\q}(i)= \left\{ 
\begin{array}{ll}
q_{i+1} \, c_{\mathbf{q}}^{i} & \mbox{ if $i \geq 0$,} \\
2 W_{-1-i}(\mathbf{q}) \, c_{\mathbf{q}}^i & \mbox{ if $i \leq -1$.}
\end{array} \right.$. Step distribution of the random walk associated to the perimeter process of a finite $\q$-Boltzmann planar map.
\item
$\omega_{\q} \geq 1$: for $\q \in \cQ_h$, denotes the solution (other than $1$, unless $\q$ is critical) of \[\sum_{i \in \Z} \omega^i \nu_{\q}(i)=1.\]
\item
$\widetilde{\nu}_{\q}(i)=\omega_{\q}^i \nu_{\q}(i)$. Step distribution of the random walk on $\Z$ associated to the peeling process of $\MM_{\q}$.
\item
$C_p(\q)$: for $\q \in \cQ_h$ and $p \geq 1$, constants such that \[\P \left( m \subset \MM_{\q} \right)=C_p(\q) \times \prod_{f \in m} q_{\deg(f)/2}\] for every finite map $m$ with one hole of perimeter $2p$.
\item
$h_p(\omega)=\sum_{i=0}^{p-1} (4 \omega)^{-i} \binom{2i}{i}$. We have $C_p(\q)= (c_{\q} \omega_{\q})^{p-1} h_p(\omega_{\q})$. Also, the perimeter process associated to a peeling exploration of $\MM_{\q}$ is a Doob transform of the $\widetilde{\nu}_{\q}$-random walk by the harmonic function $\left( h_p(\omega_{\q}) \right)_{p \geq 1}$.
\item
$a_j(\q)=\frac{1}{j} \P \left( \mbox{the degree of the root face of $\MM_{\q}$ is $2j$}\right)$ for all $j \geq 1$.
\item
$\alpha_j$: denotes a possible value of $a_j(\q)$, or the limit of the ratio $\frac{f_j}{|\mathbf{f}|}$ when $|\mathbf{f}| \to +\infty$. We will always have $\sum_{j \geq 1} j \alpha_j=1$ and $\alpha_1<1$.
\item
$\q^{(\omega)}$: once $\left( \alpha_j \right)_{j \geq 1}$ has been fixed, denotes the weight sequence for which $\omega_{\q}=\omega$, and the law of the degree of the root face is described by $\left( \alpha_j \right)_{j \geq 1}$.
\item
$\mathcal{A}$: denotes a peeling algorithm.
\item
$\expl_t^{\mathcal{A}}(m)$: explored map after $t$ filled-in peeling steps on the map $m$ using algorithm $\mathcal{A}$. This is a finite map with holes.
\item
$P_t, V_t$: denote respectively the perimeter (i.e. the half-length of the boundary of the hole) and the volume (i.e. the total number of edges) of the explored map after $t$ steps during a peeling exploration.
\item
$d(\q)=\E \left[ \left( \mbox{degree of the root vertex in $\MM_{\q}$} \right)^{-1} \right]$.
\item
$r_j(\mathbf{q})=\left( c_{\q} \omega_{\q} \right)^{j-1} q_j=\lim_{t \to +\infty} \frac{1}{t} \sum_{i=0}^{t-1} \mathbbm{1}_{P_{i+1}-P_i=j-1}$ for $\q \in \cQ_h$ and $j \geq 1$, where $P$ is the perimeter process associated to a peeling exploration of $\MM_{\q}$ (see Proposition~\ref{prop_q_as_limit}).
\item
$r_{\infty}(\q) = \frac{ \left(\sqrt{\omega_{\q}}-\sqrt{\omega_{\q}-1} \right)^2}{2 \sqrt{\omega_{\q}(\omega_{\q}-1)}} = \lim_{n \to +\infty} \frac{V_n-2P_n}{n}$ for $\q \in \cQ_h$ and $j \geq 1$, where $P$ and $V$ are the perimeter and volume processes associated to a peeling exploration of $\MM_{\q}$ (see Proposition~\ref{prop_q_as_limit}).
\end{itemize}

\newpage

\section{Preliminaries}\label{sec_univ_prelim}

Our purpose in this section is to recall basic definitions related to maps, local topology and peeling explorations as well as combinatorial results from previous works, and to introduce precisely the infinite objects that will appear in this paper.

\subsection{Definitions: maps and local topology}
\label{subsec_univ_defns}
\label{subsec_local_topology_univ}

\paragraph{Maps.} A (finite or infinite) \emph{map} $M$ is a way to glue a finite or countable collection of finite oriented polygons, called the \emph{faces}, along their edges in a connected way. By forgetting the faces of $M$ and looking only at its vertices and edges, we obtain a graph $G$ (if $M$ is infinite, then $G$ may have vertices with infinite degree). The maps that we consider will always be \emph{rooted}, i.e. equipped with a distinguished oriented edge called the \emph{root edge}. The face on the right of the root edge is the \emph{root face}, and the vertex at the start of the root edge is the \emph{root vertex}.

The \emph{dual map} of a map $m$ is the map $m^*$ whose vertices are the faces of $m$, and whose edges are the dual edges to those of $m$. We root $m^*$ at the oriented edge crossing the root edge of $m$ from left to right (see Figure~\ref{fig_map_dual}).
If the number of faces is finite, then $M$ is always homeomorphic to an orientable topological surface, so we can define the genus of $M$ as the genus of this surface. In particular, we call a map \emph{planar} if it has genus $0$.

A \emph{bipartite map} is a rooted map where it is possible to color every vertex in black or white without any monochromatic edge. By convention, we may assume that the root is always oriented from white to black, and each edge of the map has a natural orientation from white to black. In a bipartite map, all faces have an even degree.  In what follows, we will only deal with bipartite maps (except when dealing with dual maps). Therefore, we will not always specify that the map we consider is bipartite.

For every $\mathbf{f}=(f_j)_{j \geq 1}$ and $g\geq 0$, we will denote by $\B_g(\mathbf{f})$ the set of bipartite maps of genus $g$ with exactly $f_j$ faces of degree $2j$ for all $j\geq 1$.
A map of $\B_g(\mathbf{f})$ has $|\mathbf{f}| =\sum_{j\geq 1} j f_j$ edges, $\sum_{j\geq 1} f_j$ faces and $v(\mathbf{f}, g)=2-2g+\sum_{j\geq 1} (j-1)f_j$ vertices by the Euler formula. In particular, such a map exists if and only if $v(\mathbf{f},g) \geq 2$. We will denote by $\beta_g(\mathbf{f})$ the cardinality of $\B_g(\mathbf{f})$.

\begin{figure}
\center
\includegraphics[scale=0.5]{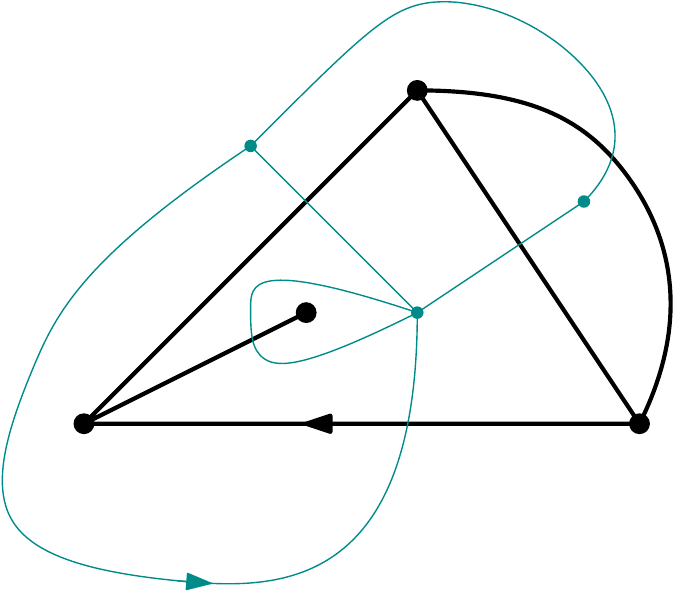}
\caption{A map (in black) and its dual (in blue). The arrows mark the roots.}\label{fig_map_dual}
\end{figure}

\paragraph{Maps with boundaries.} We will need to consider two different notions of bipartite maps with boundaries, that we call \emph{maps with holes} and \emph{maps of multi-polygons}. Roughly speaking, the first ones will be used to describe a small neighbourhood of the root in a larger map, and the second ones to describe the complement of this neighbourhood. Note that, since we will only consider bipartite maps in this work, we assume in both definitions that the maps are bipartite.

\begin{defn}
A \emph{map with holes} is a finite, bipartite map with a set of marked faces (called \emph{holes}) such that:
\begin{itemize}
\item
the boundary of each hole is a simple cycle,
\item
the boundaries of the different holes are vertex-disjoint,
\item
the adjacency graph of the \emph{internal faces} (i.e. the faces that are not holes) is connected,
\item
the root edge may be any oriented edge of the map.
\end{itemize}
By convention, the map consisting of two vertices joined by a single edge is a map with one hole and no internal face.
If $m$ is a map with holes, we denote by $\partial m$ its boundary, i.e. the union of the boundaries of its holes.
\end{defn}

\begin{defn}
Let $\ell \geq 1$ and $p_1, p_2, \dots, p_{\ell} \geq 1$. A \emph{map of the $(2p_1, \dots, 2p_{\ell})$-gon} is a finite or infinite bipartite map with $\ell$ marked oriented edges $(e_i)_{1 \leq i \leq \ell}$, such that:
\begin{itemize}
\item
$e_1$ is the root edge,
\item
the faces on the right of the $e_i$ are distinct,
\item
for all $1 \leq i \leq \ell$, the face on the right of $e_i$ has degree $2p_i$.
\end{itemize}
The faces on the right of the marked edges are called \emph{external faces}, and the other ones are called the \emph{internal faces}.
We denote by $\B^{(p_1,p_2,\dots,p_{\ell})}_g(\mathbf{f})$ the set of bipartite maps of the $(2p_1,2p_2,\dots,2p_{\ell})$-gon of genus $g$ with interior faces given by $\textbf{f}$. We also denote by $\beta^{(p_1,p_2,\dots,p_{\ell})}_g(\mathbf{f})$ its cardinal, with the convention that $\beta_g^{(0)}(\ff)$ is $1$ if $g=0$ and $\ff=\mathbf{0}$, and $0$ otherwise.
\end{defn}

Note that, in this second definition, we do not ask that the boundaries are simple or disjoint. The convention for the $0$-gon can be interpreted as saying that the only map of the $0$-gon is the map with $1$ vertex, no edge and no internal face.

\paragraph{Map inclusion.}
Given a map $m$, let $m^*$ be its dual map. Let $\mathfrak{e}$ be a finite, connected subset of edges of $m^*$ such that the root vertex of $m^*$ is incident to $\mathfrak{e}$. To $\mathfrak{e}$, we associate the map $m_{\mathfrak{e}}$ that is obtained by gluing the faces of $m$ corresponding to the vertices of $m^*$ incident to $\mathfrak{e}$ along the dual of the edges of $\mathfrak{e}$ (see Figure~\ref{fig_univ_lazy_inclusion}). Note that $m_{\mathfrak{e}}$, once rooted at the root edge of $m$, is a map with holes. We will refer to $m_{\mathfrak{e}}$ as the \emph{submap of $m$ spanned by $\mathfrak{e}$}.

If $m'$ is a map with holes and $m$ is a (finite or infinite) map, we write
\[m' \subset m\]
if $m'$ can be obtained from $m$ by the procedure described above. By convention, we also write $m' \subset m$ if $m'$ is the trivial map with two vertices and one edge, or if $m'$ consists of a simple cycle with the same perimeter as the root face of $m$ (which corresponds to the case $\mathfrak{e} = \emptyset$).

\begin{figure}[!ht]
\center
\includegraphics[scale=1]{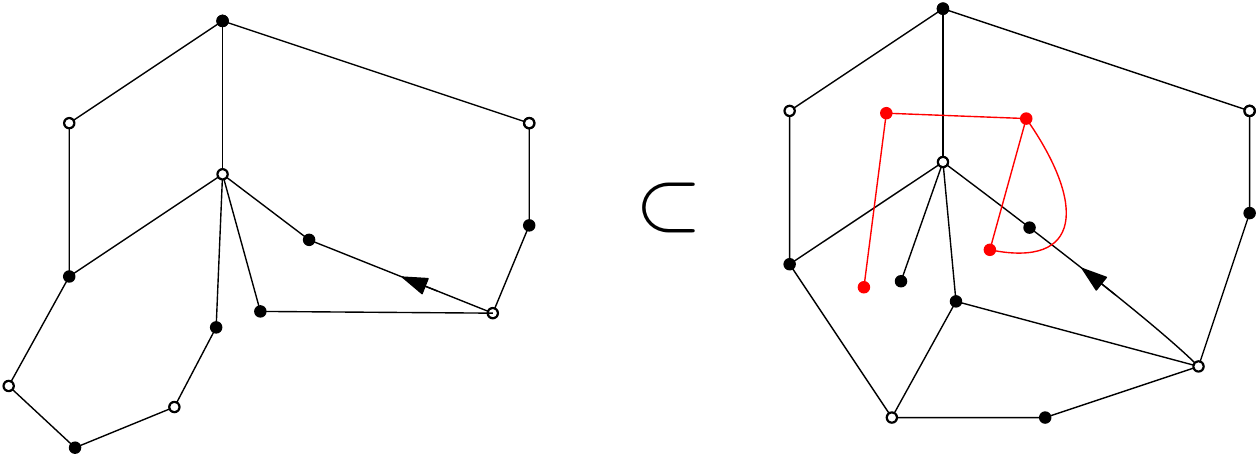}
\caption{Inclusion of bipartite maps, on an example. On the right, the map $m$ and, in red, the set of dual edges $\mathfrak{e}$. On the left, the map $m_{\mathfrak{e}}$.}\label{fig_univ_lazy_inclusion}
\end{figure}

Equivalently, we have $m' \subset m$ if $m$ can be obtained from $m'$ by gluing one or several maps of multipolygons in the holes of $m'$. We highlight that this definition of map inclusion is taken from \cite{C-StFlour} and is tailored for the \emph{lazy peeling process} of \cite{Bud15}. More precisely, maps of multipolygons may not have simple nor disjoint boundaries, so if $m' \subset m$, it is possible that two boundary edges of $m'$ actually coincide in $m$.

\paragraph{Local convergence and dual local convergence.}
The goal of this paragraph is to recall the definition of local convergence in a setting that is not restricted to planar maps. We denote by $\overline{\B}$ the set of finite or infinite bipartite maps in which all the vertices have finite degrees. A map $m$ is naturally equipped with a \emph{graph distance} $d_m$ on the set of its vertices. If $m \in \overline{\B}$, for every $r \geq 1$, we denote by $B_r(m)$ the submap of $m$ spanned by the duals of those edges of $m$ which have an endpoint at distance $d_m$ at most $r-1$ from the root vertex. The map $B_r(m)$ is then a map with holes. We also write $B_0(m)$ for the trivial bipartite map consisting of only one edge.

For any two maps $m,m' \in \overline{\B}$, we write
\[d_{\loc}(m,m')=\left( 1+\max \{r \geq 0 | B_r(m)=B_r(m')\} \right)^{-1}.\]
This is the \emph{local distance} on $\overline{\B}$. As in the planar case, the space $\overline{\B}$ is a Polish space and is the completion of the space of finite bipartite maps for $d_{\loc}$. However, this space is not compact, since $B_1(m)$ may take infinitely many values.

In our tightness argument, it will be more convenient to first work with a weaker notion of convergence which we call the \emph{dual local convergence}. We denote by $\overline{\B}^*$ the set of finite or infinite bipartite maps (regardless of whether vertex degrees are finite or not).
Let $m \in \overline{\B}^*$, and let $d_{m^*}$ be the graph distance on its dual. For $r \geq 1$, we denote by $B_r^{*}(m)$ the submap of $m$ spanned by those edges of $m^*$ which are incident to a face of $m$ lying at distance $d_{m^*}$ at most $r-1$ from the root face of $m$. By convention, let also $B_0^*(m)$ be the map consisting of a simple cycle with the same length as the boundary of the root face. Like $B_r(m)$, the "ball" $B^*_r(m)$ is a finite map with holes. For any $m,m' \in \overline{\B}^*$, we write
\[d_{\loc}^*(m,m')=\left( 1+\max \{r \geq 0 | B_r^*(m)=B_r^*(m')\} \right)^{-1}.\]
We call $d^*_{\loc}$ the \emph{dual local distance}. Then $\overline{\B}^*$ is a Polish space for $d_{\loc}^*$ and is the completion of the set of finite bipartite maps.

The reason why we introduced $d_{\loc}^*$ is that it will be very easy to obtain tightness for this distance. This will allow us to work directly on infinite objects and deduce tightness for $d_{\loc}$ later. Tightness for $d_{\loc}^*$ will be deduced from the next result.

\begin{lem}\label{lem_tight_degree_in_ball}
Let $A(\cdot)$ be a function from $(0,1)$ to $\N$ and let $r \geq 1$. There is a function $A_r(\cdot)$ from $(0,1)$ to $\N$ with the following property. Let $G$ be a stationary (for the simple random walk) random graph such that, for all $\eps>0$, we have
\[ \P \left( \deg_G(\rho)>A(\eps) \right) \leq \eps,\]
where $\rho$ is the root vertex. Then for all $\eps>0$, we have
\[ \P \left( \max_{x \in B_r(G)} \deg_G(x)>A_r(\eps) \right) \leq \eps,\]
where $B_r(G)$ is the ball of radius $r$ centered at the root vertex in $G$.
\end{lem}

\begin{proof}
This result goes back to~\cite{AS03}. More precisely, although not stated explicitely as such, it is proved by induction on $r$ in the proof of tightness of uniform triangulations for the local topology~\cite[Lemma 4.4]{AS03}. See also~\cite[Theorem 3.1]{BLS13} for a general statement with minimal assumptions.
\end{proof}

From here, we easily obtain tightness for $d_{\loc}^*$ in our setting.

\begin{lem}\label{lem_easy_dual_convergence}
Let $\mathbf{f}^{n}$ be face degree sequences such that $\frac{1}{\left| \mathbf{f}^{n} \right|} \sum_{j> A} j f_j^{n}\rightarrow 0$ as $A \to +\infty$ uniformly in $n$, and let $(g_n)$ be any sequence such that $\B_{g_n}(\mathbf{f}^{n}) \ne \emptyset$ for every $n$. Recall that $M_{\ff^n, g_n}$ is a uniform map in $\B_{g_n}(\mathbf{f}^{n})$. Then $(M_{\ff^n, g_n})$ is tight for $d_{\loc}^*$.
\end{lem}

\begin{proof}
Let $M_{\ff^n, g_n}^*$ be the dual map of $(M_{\ff^n, g_n})$. Since $M_{\ff^n, g_n}$ is invariant under rerooting at a uniform edge, the probability that the root vertex of $M_{\ff^n, g_n}^*$ has degree $2j$ is equal to $\frac{j f^{n}_j}{\left| \mathbf{f}^{n} \right|}$. Therefore, it follows from the assumption of the lemma that the root degree of $M_{\ff^n, g_n}^*$ is tight. Moreover $M_{\ff^n, g_n}^*$ is invariant by rerooting along the simple random walk. Therefore, by Lemma~\ref{lem_tight_degree_in_ball}, for every $r \geq 1$, the maximal degree in the ball of radius $r$ centered at the root in $M_{\ff^n, g_n}^*$ is tight. This implies tightness for $d_{\loc}^*$.
\end{proof}

Finally, as in \cite{BL19}, tightness for $d_{\loc}$ will be deduced from tightnesss for $d^*_{\loc}$ using the following result (the proof is the same as for triangulations, and is therefore omitted).

\begin{lem}\label{lem_dual_convergence_univ}
Let $(m_n)$ be a sequence of maps of $\overline{\B}$. Assume that
\[ m_n \xrightarrow[n \to +\infty]{d_{\loc}^*} m,\]
with $m \in \overline{\B}$ (i.e. with finite vertex degrees). Then $m_n \to m$ for $d_{\loc}$ as $n \to +\infty$.
\end{lem}

\subsection{The lazy peeling process of bipartite maps}
\label{subsec_lazy_peeling}

We now recall the definition of the \emph{lazy peeling process} of maps introduced in \cite{Bud15} (see also \cite{C-StFlour} for an extensive study). We will make heavy use of this notion in our proofs.

A \emph{peeling algorithm} is a function $\A$ that takes as input a finite bipartite map $m$ with holes, and that outputs an edge $\A(m)$ on $\partial m$ (i.e. on the boundary of one of the holes). Given an infinite, planar, one-ended bipartite map $m$ and a peeling algorithm $\mathcal{A}$, we can define an increasing sequence $\left( \expl_t^{\mathcal{A}}(m) \right)_{t \geq 0}$ of maps with one hole, such that $\expl_t^{\mathcal{A}}(m) \subset m$ for every $t$, in the following way. First, the map $\expl_0^{\mathcal{A}}(m)$ is the trivial map consisting of the root edge only. For every $t \geq 0$, we call the edge $\A \left( \expl_t^{\mathcal{A}}(m) \right)$ the \emph{peeled edge at time $t$}. Let $F_t$ be the face of $m$ on the other side of this peeled edge (i.e. the side incident to a hole in $\expl_t^{\mathcal{A}}(m)$). There are two possible cases, as summed up on Figure~\ref{fig_univ_lazy_peeling}:
\begin{itemize}
\item either $F_t$ doesn't belong to $\expl_t^{\mathcal{A}}(m)$, and then $\expl_{t+1}^{\mathcal{A}}(m)$ is the map obtained from $\expl_t^{\mathcal{A}}(m)$ by gluing a simple face of size $\deg(F_t)$ along $\A \left( \expl_t^{\mathcal{A}}(m) \right)$;
\item or $F_t$ belongs to $\expl_t^{\mathcal{A}}(m)$. In that case, by planarity, there exists an edge $e_t \in \expl_t^{\mathcal{A}}(m)$ on the same hole as $\A \left( \expl_t^{\mathcal{A}}(m) \right)$ such that $e_t$ and $\A \left( \expl_t^{\mathcal{A}}(m) \right)$ are actually identified in $m$. The map $\expl_{t+1}^{\mathcal{A}}(m)$ is then obtained from $\expl_t^{\mathcal{A}}(m)$ by gluing $\mathcal{A} \left( \expl_t^{\mathcal{A}}(m) \right)$ and $e_t$ together and, if this creates a finite hole, by filling it in the same way as in $m$.
\end{itemize}
Such an exploration is called \emph{filled-in}, because all the finite holes are filled at each step.

Let us now discuss two different ways to define peeling explorations on finite or nonplanar maps. We first note that, if we do not fill the region in the second case, then the definition of a peeling exploration still makes sense for any finite or infinite map, with the only difference that the explored part may now have several holes. This is what we will call a \emph{non-filled-in} peeling exploration, and this will only be used briefly in Section \ref{subsec_planarity}.

Finally, for a finite map $m$, we define a filled-in exploration using the following convention. Assume that the peeled edge at time $t$ is glued to another boundary edge of $\expl^{\A}_t(m)$ and forms two holes:
\begin{itemize}
\item
if these two holes are connected in $m \backslash \expl^{\A}_t(m)$ (which may occur if $m$ is not planar), we stop the exploration at time $t$;
\item
if not, we obtain $\expl^{\A}_{t+1}(m)$ by filling completely the hole which contains the smallest number of edges in $m$.
\end{itemize}
Note that with these conventions, the map $\expl^{\A}_t(m)$ always have exactly one hole. This definition will be used to compare peeling explorations of finite and infinite maps in Section \ref{sec_univ_end}. At this point, the local planarity results from Section \ref{sec_univ_tight} will allow us to assume that with high probability, the explorations are not stopped too early.

\begin{figure}
\center
\includegraphics[scale=1]{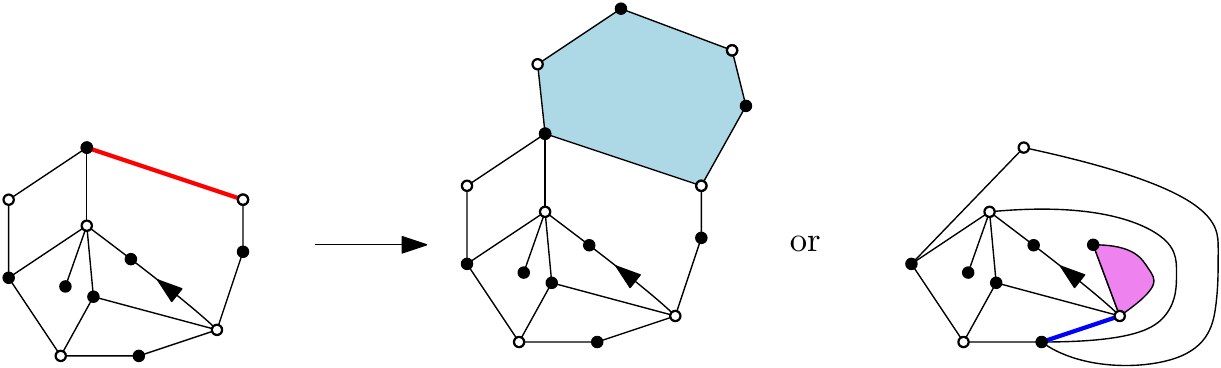}
\caption{The lazy peeling on an example. The peeled edge is in red. Either a new face is discovered (center case), or the chosen edge is glued to another boundary edge (right case, the glued edge is in blue and the filled part in pink).}\label{fig_univ_lazy_peeling}
\end{figure}

\subsection{Combinatorial enumeration}\label{subsec_prelim_combi}

\paragraph{Partition functions for Boltzmann planar maps.}
Before describing infinite Boltzmann models in detail, we recall well-known enumeration results in the finite, planar case. We write $\cQ=[0,1]^{\N^*}$. Fix a sequence $\q \in \cQ$. The partition function of bipartite, planar maps of the $2p$-gon with Boltzmann weights $\q$ is defined as
\[W_p(\q)=\sum_{m} \prod_{f \in m} q_{\deg(f)/2},\]
where the sum spans over all planar bipartite maps $m$ of the $2p$-gon, and the product is over internal faces of $m$. We also denote by $W_p^{\bullet}(\q)$ the \emph{pointed} partition function, i.e. the sum obtained by multiplying the weight of a map $m$ by its total number of vertices. Note that $W_1(\q)$ can also be interpreted as the partition function of maps of the sphere.

We recall from \cite{MM07} the classical necessary and sufficient condition for the finiteness of these partition functions. Given a weight sequence $\q \in \cQ$, let 
\begin{equation*}
f_{\mathbf{q}}(x)=\sum_{j \geq 1} q_{j}\binom{2 j-1}{j-1} x^{j-1}.
\end{equation*}
If the equation 
\begin{equation}\label{eq_univ_admissible}
f_{\mathbf{q}}(x)=1-\frac{1}{x}
\end{equation}
has a positive solution $Z_{\q}$ we call $\mathbf{q}$ \emph{admissible}, and write $c_{\q}=4Z_{\q}$. Then by results from \cite{MM07}, for all $p \geq 1$, the partition functions $W_p(\q)$ and $W_p^{\bullet}(\q)$ are finite if and only if $\q$ is admissible. Moreover, in this case, for $p \geq 0$, we have
\begin{equation}\label{eqn_exact_pointed_partition_function}
W_p^{\bullet}(\q) = c_{\q}^p \times \frac{1}{4^p} \binom{2p}{p}.
\end{equation}
It is also possible to derive simple integral formulas for $W_p(\q)$ in terms of $c_{\q}$ but this will not be needed here, see \cite{C-StFlour} for more details. We denote by $\cQ_a$ the set of admissible weight sequences.

Finally, let $\cQ^*$ be the set of those $\q=(q_j)_{j \geq 1} \in \cQ$ for which there exists $j \geq 2$ such that $q_j>0$ (which ensures $W_p(\q)>0$ for all $p \geq 1$). For $\q \in \cQ^* \cap \cQ_a$, we define the \emph{Boltzmann distribution} with weights $\q$ on finite planar bipartite maps of the $2p$-gon as
\[\P(m)=\frac{1}{W_p(\q)} \prod_{f \in m} q_{\deg(f)/2}\]
for all bipartite planar map $m$ of the $2p$-gon.

\paragraph{A general recursion for bipartite maps.}
As in \cite{BL19}, we are lacking precise asymptotics on the enumeration of maps when both the genus and the size go to infinity. The following recurrence formula, proved in \cite{Lo19}, will play the same role as the Goulden--Jackson formula for triangulations \cite{GJ08}. We set the convention $\beta_g(\mathbf{0})=0$ for all $g$. Then, for every $g\geq 0$ and every face degree sequence $\ff$, we have
\begin{equation}\label{rec_biparti_genre_univ}
\binom{|\mathbf{f}|+1}{2}\beta_g(\mathbf{f})=\hspace{-0.5cm}\sum_{\substack{\mathbf{h}^{(1)}+\mathbf{h}^{(2)}=\mathbf{f}\\g^{(1)}+g^{(2)}+g^*=g}}\hspace{-0.5cm}(1+|\mathbf{h}^{(1)}|)\binom{v \left( \mathbf{h}^{(2)},g^{(2)} \right) }{2g^*+2}\beta_{g^{(1)}}(\mathbf{h}^{(1)}) \beta_{g^{(2)}}(\mathbf{h}^{(2)})+\sum_{g^*\geq 0}\binom{v \left( \mathbf{f},g \right) +2g^*}{2g^*+2}\beta_{g-g*}(\mathbf{f}),
\end{equation}
where we recall that $|\mathbf{f}|=\sum_{j \geq 1} j f_j$ and $v(\ff,g)=2-2g+\sum_j (j-1) f_j$ (i.e. it is the number of vertices of a map with face degrees $\mathbf{f}$ and genus $g$).

\subsection{Infinite Boltzmann bipartite planar maps}
\label{subsec_IBPM}

\paragraph{Definition of the models.}
Our goal here is to recall the definition of infinite Boltzmann bipartite planar maps introduced in \cite[Appendix C]{B18these} (and earlier in \cite{Bud15} in the critical case). We also refer to \cite[Chapter 8]{C-StFlour} for some basic properties of these objects that we will state below.

Let $\q=(q_j)_{j \geq 1}$ be a sequence of nonnegative real numbers that we will call the \emph{Boltzmann weights}. A random infinite bipartite planar map $M$ is called $\q$-Boltzmann if there are constants $\left( C_p(\q) \right)_{p \geq 1}$ such that, for every finite bipartite map $m$ with one hole of perimeter $2p$, we have
\[ \P \left( m \subset M \right)=C_p(\q) \prod_{f \in m} q_{\deg(f)/2},\]
where the product is over all internal faces of $m$.

We will see that given $\q$, such a map does not always exist, but when it does, it is unique, i.e. the constants $C_p(\q)$ are determined by $\q$, which justifies the notation $C_p(\q)$. More precisely, as noted in \cite[Appendix C]{B18these}, if a $\q$-Boltzmann map exists, then the partition function of maps of a $2$-gon with Boltzmann weights $\q$ must be finite, which is equivalent to the admissibility criterion \eqref{eq_univ_admissible}. Moreover, with the notation of Section~\ref{subsec_univ_defns}, we call $\q$ \emph{critical} if $f'_{\q}(Z_{\q})=\frac{1}{Z_{\q}^2}$ and \emph{subcritical} if this is not the case.

Finally, we define a measure $\nu_{\q}$ on $\Z$ as follows:
\begin{equation}\label{eqn_defn_nu}
\nu_{\q}(i)= \left\{ 
\begin{array}{ll}
q_{i+1} \, c_{\mathbf{q}}^{i} & \mbox{ if $i \geq 0$,} \\
2 W_{-1-i}(\mathbf{q}) \, c_{\mathbf{q}}^i & \mbox{ if $i \leq -1$.}
\end{array} \right.
\end{equation}
As noted in \cite{Bud15}, this is the step distribution of the random walk on $\Z$ describing the evolution of the perimeter of a finite $\q$-Boltzmann map with a large perimeter (see also \cite[Chapter 5.1]{C-StFlour}). Then previous results about the existence of $\q$-IBPM can be summed up as follows.

\begin{thm}\label{thm_rappel_these}
\begin{enumerate}
\item
If a $\q$-IBPM exists, it is unique (in distribution), so we can denote it by $\MM_{\q}$.
\item
If $\q \notin \cQ^* \cap \cQ_a$, then $\MM_{\q}$ does not exist.
\item
If $\q \in \cQ^* \cap \cQ_a$ is critical, then $\MM_{\q}$ exists and $C_p(\q)=c_{\q}^{p-1} \times \frac{2p}{4^p} \binom{2p}{p}$.
\item
If $\q \in \cQ^* \cap \cQ_a$ is subcritical, then $\MM_{\q}$ exists if and only if the equation
\begin{equation}\label{boltzmann_equation_omega}
\sum_{i \in \Z} \nu_{\q}(i) \omega^i =1
\end{equation}
has a solution $\omega_{\q}>1$.
\item
In this case, the solution $\omega_{\q}$ is unique and, for every $p \geq 1$, we have
\begin{equation}\label{boltzmann_formula_cp}
C_p(\mathbf{q})= \left( c_{\q} \omega_{\q} \right)^{p-1} \sum_{i=0}^{p-1} (4\omega_{\q})^{-i} \binom{2i}{i}.
\end{equation}
\end{enumerate}
\end{thm}
The third point is from \cite{Bud15}, and the others are from \cite[Appendix C]{B18these}\footnote{We have fixed a small mistake from \cite[Appendix C]{B18these}, where $c_{\q}$ was omitted in the formula for $C_p(\q)$.}. When it exists, we will call the map $\MM_{\q}$ the \emph{$\q$-IBPM} (for \emph{Infinite Boltzmann Planar Map)}.
We denote by $\cQ_h \subset \cQ^* \cap \cQ_a$ the set of weight sequences $\q$ for which $\MM_{\q}$ exists. We also note that the formula for $C_p(\q)$ in the critical case is a particular case of the subcritical one where $\omega=1$. Since this function will appear many times later, for $\omega \geq 1$ and $p \geq 1$, we write:
\begin{equation}\label{eqn_defn_homega}
h_p(\omega)=\sum_{i=0}^{p-1} (4\omega)^{-i} \binom{2i}{i}.
\end{equation}
In particular, if $\omega=1$, then $h_p(\omega)=\frac{2p}{4^p} \binom{2p}{p} \sim \frac{2}{\sqrt{\pi}} \sqrt{p}$ as $p \to +\infty$. If $\omega>1$, then $h_p(\omega) \to \sqrt{\frac{\omega}{\omega-1}}$ as $p \to +\infty$.

\paragraph{The random walk $\widetilde{\nu}_{\q}$.}
To study the $\q$-IBPM, we define the measure $\widetilde{\nu}_{\q}$ on $\Z$ by $\widetilde{\nu}_{\q}(i)=\omega_{\q}^i \nu_{\q}(i)$, where $\omega_{\q}$ is given by \eqref{boltzmann_equation_omega} if $\q$ is subcritical, and $\omega=1$ if $\q$ is critical. The random walk with step distribution $\widetilde{\nu}_{\q}$ plays an important role when studying $\MM_{\q}$. We first note that, if $\q$ is not critical, then this walk has a positive drift. Indeed, denoting by $F_{\q}$ the generating function of $\nu_{\q}$, we have
\[ \sum_{i \in \Z} i \, \widetilde{\nu}_{\q}(i) = F'_{\q}(\omega_{\q})>0,\]
since $F_{\q}$ is convex and takes the value $1$ both at $1$ and at $\omega_{\q}>1$. Note also that it is possible that the drift is $+\infty$.

\paragraph{Lazy peeling explorations of the $\q$-IBPM.}
We now perform a few computations related to lazy peeling explorations of the $\q$-IBPM. For this, we fix a peeling algorithm $\A$, and consider a filled-in exploration of $\MM_{\q}$ according to $\A$. We recall that $\expl_t^{\A}(\MM_{\q})$ is the explored region after $t$ steps, and we denote by $\left( \mathcal{F}_t \right)_{t \geq 0}$ the filtration generated by this exploration. We denote by $P_t$ (resp. $V_t$) the half-perimeter (resp. total number of edges) of $\expl_t^{\A}(\MM_{\q})$. We will call $P$ and $V$ the \emph{perimeter and volume processes} associated to a peeling exploration of $\MM_{\q}$.

It follows from the definition of $\MM_{\q}$ that $(P_t, V_t)_{t \geq 0}$ is a Markov chain on $\N^2$ and that its law does not depend on the algorithm $\mathcal{A}$. More precisely $P$ is a Doob transform of the random walk with step distribution $\widetilde{\nu}_{\q}$, i.e. it has the following transitions:
\begin{equation}\label{peeling_transitions}
\P \left( P_{t+1}=P_t+i | \mathcal{F}_t \right)= \widetilde{\nu}_{\q}(i) \frac{h_{P_t+i}(\omega_{\q})}{h_{P_t}(\omega_{\q})},
\end{equation}
where $h_p(\omega)$ is given by \eqref{eqn_defn_homega}. As noted in \cite{C-StFlour}, this implies that $\left( h_p(\omega_{\q}) \right)_{p \geq 1}$ is harmonic on $\{1,2, \dots\}$ for the random walk with step distribution $\widetilde{\nu}_{\q}$, and that for $\q$ subcritical $P$ has the distribution of this random walk, conditioned to stay positive (if $\q$ is critical, the conditioning is degenerate, but this can still be made sense of).

\paragraph{IBPM with finite expected degree of the root face.}
We denote by $\mathcal{Q}_f$ the set of $\q \in \cQ_h$ such that the degree of the root face of $\MM_{\q}$ has finite expectation. Since our Theorem~\ref{univ_main_thm} only holds if the expected degree of the root face is finite in the limit, we will need to gather a few consequences of this assumption on $\q$ and the peeling process of $\MM_{\q}$. Note that, for all $\q \in \cQ_h$, the degree of the root face is determined by the first peeling step on $\MM_{\q}$. More precisely, by \eqref{peeling_transitions}, we have for all $j \geq 1$:
\begin{equation}\label{eqn_walk_to_rootface}
\P \left( \mbox{the root face of $\MM_{\q}$ has degree $2j$} \right)=\frac{h_j(\omega_{\q})}{h_1(\omega_{\q})} \, \widetilde{\nu}_{\q}(j-1) = \frac{h_j(\omega_{\q})}{h_1(\omega_{\q})} \left( c_{\q} \omega_{\q} \right)^{j-1} q_j.
\end{equation}
If $\q$ is critical, the right hand-side is equivalent to $\frac{2}{\sqrt{\pi}}\sqrt{j} c_{\q}^{j-1} q_j$ as $j \to +\infty$, so $\q \in \mathcal{Q}_f$ if and only if
\begin{equation}\label{eqn_finite_32_moment}
\sum_{j \geq 1} j^{3/2} c_{\q}^j q_j <+\infty.
\end{equation}
On the other hand, we recall (see e.g. \cite[Chapter 5.2]{C-StFlour}) that a critical weight sequence $\q$ is called \emph{of type $\frac{5}{2}$}, or \emph{critical generic}, if
\[ \sum_{j \geq 1} (j-1)(j-2) \binom{2j-1}{j-1} q_j \left( \frac{c_{\q}}{4} \right)^{j-3} <+\infty, \]
which is clearly equivalent to \eqref{eqn_finite_32_moment}.
In the subcritical case, by \eqref{eqn_walk_to_rootface}, $\q \in \cQ_f$ is equivalent to
\[\sum_{j \geq 1} j \nnu_{\q}(j) <+\infty,\] i.e. the drift of $\nnu$ is finite. To sum up:
\begin{itemize}
\item
In the critical case, $\q \in \mathcal{Q}_f$ if and only if $\q$ is critical generic, which means that the perimeter process $(P_n)$ converges to a $3/2$-stable Lévy process with no positive jump, conditioned to be positive (see \cite[Theorem 10.1]{C-StFlour}). This basically means that $\q$-Boltzmann finite maps for $\q \in \cQ_f$ lie in the domain of attraction of the Brownian map \cite{MM07}.
\item
In the subcritical case, $\q \in \mathcal{Q}_f$ if and only if the measure $\nnu$ has finite expectation. Since the perimeter process $P$ has the law of a $\nnu$-random walk conditioned on an event of positive probability, this means that $P$ has linear growth (instead of super-linear if the expectation of $\nnu$ was infinite).
\end{itemize}

\subsection{Four ways to describe Boltzmann weights}

\paragraph{Four parametrizations of $\mathcal{Q}_h$.}
In this work, we will make use of four different "coordinate systems" to navigate through the spaces $\mathcal{Q}_h$ and $\mathcal{Q}_f$, each with its own advantages. The goal of this section is to define these parametrizations and to establish some relations between them. In particular, we will prove Theorem \ref{thm_prametrization_rootface}.

Our first coordinate system, already used in the last pages, consists in using directly the Boltzmann weights $q_j$ for $j \geq 1$. It is the simplest way to define the model $\MM_{\q}$ and gives the simplest description of its law.

The second parametrization we will use is the one given by Proposition~\ref{prop_q_as_limit} below: we describe $\q$ by parameters $r_j(\q) \in [0,1)$ for $j \geq 1$ and $r_{\infty}(\q) \in (0,+\infty]$. Here $r_j(\q)$ describes the proportion of peeling steps where we discover a face of degree $2j$ during a peeling exploration of $\MM_{\q}$, and $r_{\infty}(\q)$ comes from a comparison between the volume and perimeter growths. The advantage of these parameters is that they allow to "read" $\q$ as an almost sure observable on a peeling exploration of the map $\MM_{\q}$. This will be useful in Section~\ref{sec_arg_deux_trous}.

The third parametrization consists in using on the one hand the law of the root face, and on the other hand the average degree of the vertices. More precisely, for $j \geq 1$, we write
\[a_j(\q)=\frac{1}{j} \P \left( \mbox{the root face of $\MM_{\q}$ has degree $2j$} \right).\]
We note that $\sum_{j \geq 1} j a_j(\q)=1$ and that $a_1(\q)<1$ since a map consisting only of $2$-gons would have vertices with infinite degrees.
We also write $d(\q)=\E \left[ \frac{1}{\deg_{\MM_{\q}}(\rho)} \right]$, where $\rho$ is the root vertex. The advantage of this parametrization is that the analogues of these quantities are easy to compute if we replace $\MM_{\q}$ by a finite uniform map with prescribed genus and face degrees. These parameters are our only way to link the finite and infinite models, and will therefore be useful in the end of the proof of Theorem~\ref{univ_main_thm}. However, it is not obvious at all that $\left( a_j(\q) \right)_{j \geq 1}$ and $d(\q)$ are sufficient to characterize $\q$. We will actually prove this in the end of the paper, only for $\q \in \mathcal{Q}_f$, as a consequence of local convergence arguments (Proposition~\ref{prop_monotonicity_deg}).  

Finally, the fourth coordinate system is the one from Theorem \ref{thm_prametrization_rootface}: it is a variant of the third one where we replace $d(\q)$ by $\omega_{\q}$, which makes it easier to handle. This one is useful as an intermediate step towards the third one. Moreover, contrary to the third one, we can prove rather quickly (Theorem \ref{thm_prametrization_rootface}) that it provides a nice parametrization of the whole space $\cQ_h$.

\paragraph{Recovering $\q$ from explorations of $\MM_{\q}$.}
We now describe more precisely our second parametrization of $\cQ_h$. The next result basically states that we can recover the weight sequence $\q$ by just observing the perimeter and volume processes defined above (we recall that the volume is measured by the total number of edges).

\begin{prop}\label{prop_q_as_limit}
Let $\q \in \cQ_h$, and let $P$ and $V$ be the perimeter and volume processes associated to a peeling exploration of $\MM_{\q}$. We have the following almost sure convergences:
\begin{equation}\label{limit_number_jgons}
\frac{1}{t} \sum_{i=0}^{t-1} \mathbbm{1}_{P_{i+1}-P_i=j-1} \xrightarrow[t \to +\infty]{a.s.} \left( c_{\q} \omega_{\q} \right)^{j-1} q_j =: r_j(\mathbf{q}) \in [0,1)
\end{equation}
for every $j \geq 1$, and
\begin{equation}\label{limit_volume_growth}
\frac{V_t-2P_t}{t} \xrightarrow[t \to +\infty]{a.s.} \frac{ \left(\sqrt{\omega_{\q}}-\sqrt{\omega_{\q}-1} \right)^2}{2 \sqrt{\omega_{\q}(\omega_{\q}-1)}} =: r_{\infty}(\mathbf{q}) \in (0,+\infty].
\end{equation}
Moreover, the weight sequence $\q$ is a measurable function of the numbers $r_j(\mathbf{q})$ for $j \in \N^* \cup \{\infty\}$.
\end{prop}

\begin{proof}
In the subcritical case, the second convergence is Proposition 10.12 of \cite{C-StFlour}. In the critical case, we have $\omega_{\q}=1$ so the right-hand side of \eqref{limit_volume_growth} is infinite, and the result follows from Lemma 10.9 of \cite{C-StFlour}.

Let us now prove the first convergence. For this, we first note that we have $P_t \to +\infty$ almost surely as $t \to +\infty$. Indeed, this again follows from \cite[Proposition 10.12]{C-StFlour} in the subcritical case and from \cite[Lemma 10.9]{C-StFlour} in the critical case. On the other hand, given the asymptotics for $h_p(\omega)$ right after \eqref{eqn_defn_homega}, for any fixed $j \geq 1$, we have $\frac{h_{p+j-1}(\omega_{\q})}{h_p(\omega_{\q})} \to 1$ as $p \to +\infty$. It follows that
\[ \P \left( P_{t+1}-P_t=j-1 | \mathcal{F}_t \right) \xrightarrow[t \to +\infty]{a.s.} \widetilde{\nu}_{\q}(j-1) = \left( c_{\q} \omega_{\q} \right)^{j-1} q_j, \]
and the first convergence follows by the law of large numbers.

Finally, the function $\omega \to \frac{ \left(\sqrt{\omega}-\sqrt{\omega-1} \right)^2}{2 \sqrt{\omega(\omega-1)}}$ is a decreasing homeomorphism from $[1,+\infty)$ to $(0,+\infty]$, so $\omega_{\q}$ is a measurable function of $r_{\infty}(\mathbf{q})$. Moreover, by the
definition of $c_{\q}$~\eqref{eq_univ_admissible}, we have
\[ 1-\frac{4}{c_{\q}} = \sum_{j \geq 1} \binom{2j-1}{j-1} q_j \left( \frac{c_{\q}}{4} \right)^{j-1} = \sum_{j \geq 1} \frac{1}{4^j} \binom{2j-1}{j-1} \frac{r_j(\mathbf{q})}{\omega_{\q}^{j-1}},\]
which implies that $c_{\q}$ is a measurable function of $\omega_{\q}$ and the numbers $r_j(\mathbf{q})$ for $j \in \N^*$. Finally, given $c_{\q}$ and the $r_j(\q)$, we easily recover the $q_j$ from~\eqref{limit_number_jgons}.
\end{proof}

\paragraph{Weight sequences corresponding to a given distribution of the root face.}
We now prove Theorem~\ref{thm_prametrization_rootface} by showing that our fourth parametrization is indeed bijective. We first state the precise version of Theorem~\ref{thm_prametrization_rootface}.  We recall that for $\q \in \mathcal{Q}_h$, the numbers $a_j(\q)$ satisfy $\sum_{j \geq 1} j a_j(\q)=1$ and $\alpha_1<1$, and we have $\omega_{\q} \geq 1$.

\begin{prop}\label{prop_third_parametrization}
Let $(\alpha_j)_{j \geq 1}$ be such that $\sum_{j \geq 1} j \alpha_j=1$ and $\alpha_1 < 1$, and let $\omega \geq 1$. Then there is a unique $\q \in \mathcal{Q}_h$ such that
\[ \omega_{\q}=\omega \mbox{ and } \forall j \geq 1, \, a_j(\q)=\alpha_j.\]
Moreover, this weight sequence $\q$ is given by
\begin{equation}\label{eqn_qjomega_univ}
q_j=\frac{j \alpha_j}{\omega^{j-1}h_{\omega}(j)} \left( \frac{1-\sum_{i \geq 1} \frac{1}{4^{i-1}} \binom{2i-1}{i-1} \frac{i \alpha_i}{\omega^{i-1} h_{\omega}(i)} }{4} \right)^{j-1}.
\end{equation}
\end{prop}

\begin{proof}[Proof of Proposition~\ref{prop_third_parametrization}]
We start with uniqueness. We note that
\begin{equation}\label{eqn_alpha_from_q}
a_j(\q)=\frac{1}{j} C_j(\q) q_j = \frac{1}{j} \left( c_{\q} \omega_{\q}\right)^{j-1} h_{j}(\omega_{\q}) q_j,
\end{equation}
so $q_j$ can be obtained as a function of $a_j(\q)=\alpha_j$, $\omega_{\q}$ and $c_{\q}$. Moreover, by the definition~\eqref{eq_univ_admissible} of $c_{\q}$, we have
\begin{equation}\label{eqn_g_function_of_aj}
1-\frac{4}{c_{\q}}=\sum_{i \geq 1} \frac{1}{4^{i-1}} \binom{2i-1}{i-1} q_i c_{\q}^{i-1} = \sum_{i \geq 1} \frac{1}{4^{i-1}} \binom{2i-1}{i-1} \frac{i a_i(\q)}{\omega_{\q}^{i-1} h_{i}(\omega_{\q})},
\end{equation}
so $c_{\q}$, and therefore $q_j$ for all $j \geq 1$, can be deduced from $\omega_{\q}$ and $\left( a_j(\q) \right)_{j \geq 1}$. More precisely, we obtain the formula~\eqref{eqn_qjomega_univ}, which in particular proves the uniqueness.

To prove the existence, it is enough to check that, for all $\omega \geq 1$ and $(\alpha_j)_{j \geq 1}$ with $\sum j \alpha_j=1$ and $\alpha_1<1$, the sequence $\q$ given by \eqref{eqn_qjomega_univ} is indeed in $\mathcal{Q}_h$, with $\omega_{\q}=\omega$ and $a_j(\q)=\alpha_j$ for all $j$. 
Following \eqref{eqn_g_function_of_aj}, we first write
\begin{equation}\label{eqn_c_of_alpha_omega}
c=\frac{4}{1-\sum_{i \geq 1} \frac{1}{4^{i-1}} \binom{2i-1}{i-1} \frac{i \alpha_i}{\omega^{i-1} h_{i}(\omega) }},
\end{equation}
and check that $\q$ is admissible with $c_{\q}=c$. First $\omega^{i-1} h_{i}(\omega)$ is a polynomial in $\omega$ with nonnegative coefficients so $\omega^{i-1} h_{i}(\omega) \geq h_i(1)=\frac{2i}{4^i} \binom{2i}{i}$. From here, we get
\[ \sum_{i \geq 1} \frac{1}{4^{i-1}} \binom{2i-1}{i-1} \frac{i \alpha_i}{\omega^{i-1} h_{i}(\omega) } \leq \sum_{i \geq 1} \alpha_i < \sum_{i \geq 1} i \alpha_i = 1 \]
because $\alpha_1 <1$. Therefore, the numbers $q_j$ are nonnegative and $c>0$, and we can rewrite \eqref{eqn_qjomega_univ} as
\[q_j=\frac{j \alpha_j}{(\omega c)^{j-1} h_{j}(\omega)}. \]
From here, we get
\[ \sum_{i \geq 1} \frac{1}{4^{i-1}} \binom{2i-1}{i-1} q_i c^{i-1} = 1-\frac{4}{c}\]
immediately by the definition of $c$, which proves $\q \in \mathcal{Q}_a$ and $c_{\q}=c$.
Also, we know that $\alpha_1<1$ so there is $j \geq 2$ with $\alpha_j>0$, which implies $q_j>0$, so $\q \in \mathcal{Q}^*$.

We now prove $\q \in \mathcal{Q}_h$ with $\omega_{\q}=\omega$, which is equivalent to proving
\[ \sum_{i \in \Z} \nu_{\q}(i) \omega^i = 1,\]
where we recall that $\nu_{\q}$ is defined by \eqref{eqn_defn_nu}. For this,
inspired by similar arguments in the critical case (see e.g. \cite[Lemma 5.2]{C-StFlour}), the basic idea will be to show that $\left( \omega^i h_{i}(\omega) \right)_{i \geq 1}$ is harmonic for $\nu_{\q}$. More precisely, the equality $\sum_{i \geq 1} i \alpha_i=1$ can be interpreted as a harmonicity relation at $1$: setting $h_i(\omega)=0$ for $i \leq -1$, we have
\begin{equation}\label{eqn_harmo_at_one}
\sum_{i \in \mathbb{Z}} h_{i+1}(\omega) \omega^i \nu_{\q}(i) = \sum_{j \geq 1} \omega^{j-1} h_{j}(\omega) c^{j-1} q_j = \sum_{j \geq 1} j \alpha_j = 1 = h_{1}(\omega),
\end{equation}
where in the beginning we do the change of variables $j=i+1$. On the other hand, we know that $h_{p}(\omega)=\sum_{i=0}^{p-1} \omega^{-i} u(i)$, where $u(i)=\frac{1}{4^i} \binom{2i}{i}$ for $i \geq 0$ (and we set the convention $u(i)=0$ for $i \leq -1$). But the same function $u$ plays an important role in the description of the law of the peeling process of finite Boltzmann maps. In particular, we know that $u$ is $\nu_{\q}$-harmonic on positive integers for any admissible weight sequence $\q$ (this can be found in the proof of Lemma 5.2 in \cite{C-StFlour}). That is, for all $j \geq 1$, we have
\[ u(j)=\sum_{i \in \mathbb{Z}} \nu_{\q}(i) u(i+j).\]
Multiplying by $\omega^{-j}$ and summing over $1 \leq j \leq p-1$, we get, for all $p \geq 1$:
\[ h_{p}(\omega)-h_{1}(\omega) = \sum_{i \in \Z} \omega^i \nu_{\q}(i) \left( h_{p+i}(\omega) - h_{i+1}(\omega) \right). \]
Summing this with \eqref{eqn_harmo_at_one} and dividing by $h_{p}(\omega)$, we obtain
\[ \sum_{i \in \Z} \omega^i \nu_{\q}(i) \frac{h_{p+i}(\omega)}{h_{p}(\omega)} =1. \]
for all $p \geq 1$. When $p \to +\infty$, we have that
$h_{p}(\omega)$ has a positive limit if $\omega>1$ and is equivalent to $\frac{2}{\sqrt{\pi}}\sqrt{p}$ if $\omega=1$, so $\frac{h_{p+i}(\omega)}{h_{p}(\omega)} \to 1$ in every case. Therefore, by dominated convergence, we get
\[ \sum_{i \in \Z} \nu_{\q}(i) \omega^i=1,\]
where the domination $\sum_{i \in \Z} \nu_{\q}(i) \omega^i <+\infty$ is immediate for negative values of $i$ since $\omega \geq 1$, and comes from the convergence of the sum \eqref{eqn_harmo_at_one} for positive values of $i$. This proves $\q \in \mathcal{Q}_h$ with $\omega_{\q}=\omega$, and from here $a_j(\q)=\alpha_j$ is immediate using \eqref{eqn_alpha_from_q}.
\end{proof}

\begin{proof}[Proof of Theorem \ref{thm_prametrization_rootface}]
It is clear from Proposition \ref{prop_third_parametrization} that weight sequences with a given root face distribution are parametrized by $\omega \in [1,+\infty)$. We denote by $\q^{(\omega)}$ the unique weight sequence for which the law of the root face is given by $(\alpha_j)_{j \geq 1}$ and for which $\omega_{\q^{(\omega)}}=\omega$. Then $q^{(1)}$ is critical by definition. Moreover, using \eqref{eqn_qjomega_univ} and \eqref{eqn_c_of_alpha_omega}, we get for $i \geq 0$:
\begin{equation}\label{eqn_omega_infinite}
\nnu_{\q^{(\omega)}}(i)=q^{(\omega)}_{i+1} \omega^i c_{\q^{(\omega)}}^i \xrightarrow[\omega \to +\infty]{} (i+1) \alpha_{i+1}.
\end{equation}
The sum over $i \geq 0$ of the right-hand side is equal to $1$, so $\nnu_{\q^{(\omega)}} \left( (-\infty,-1] \right) \to 0$ as $\omega \to +\infty$. By \eqref{peeling_transitions}, this means that the probability of peeling cases decreasing the perimeter goes to $0$. Since these are the cases creating cycles in the dual, the dual of $\MM_{\q^{(\omega)}}$ becomes close to a tree when $\omega \to +\infty$, and the vertex degrees in $\MM_{\q^{(\omega)}}$ go to infinity.
\end{proof}

\paragraph{Two technical results on the dependance in $\omega$.}
We conclude this section with two technical results that we will need in the end of the proof (Section~\ref{subsec_last_step}). Both deal with the way that some quantities depend on the parameter $\omega$. We fix $(\alpha_j)_{j \geq 1}$ such that $\sum_{j \geq 1} j \alpha_j = 1$ and $\alpha_1<1$. By Proposition~\ref{prop_third_parametrization}, we can denote by $\q^{(\omega)}$ the unique weight sequence for which the law of the root face is given by $(\alpha_j)_{j \geq 1}$ and $\omega_{\q^{(\omega)}}=\omega$.

The first technical lemma states that we can recover $\q$ from the law of the root face $(\alpha_j)_{j \geq 1}$ and a single weight $q_j$, provided $j \geq 2$.
\begin{lem}\label{lem_qj_is_monotone}
For every $j \geq 1$, the function $\omega \to q_j^{(\omega)}$ is nonincreasing. Moreover, if $j \geq 2$ and $\alpha_j>0$, this function is decreasing.
\end{lem}
Since the proof is not particularly enlightening, we postpone it to Appendix \ref{appendix_qj_monotone}.

Our second technical lemma is a reinforcement of a part of Proposition~\ref{prop_q_as_limit} above. It states that the second convergence result~\eqref{limit_volume_growth} is uniform in $\omega$ as long as $\omega$ is bounded away from $1$ and $+\infty$.

\begin{lem}\label{lem_unif_volume}
Let $\left( P_t^{(\omega)} \right)_{t \geq 0}$ and $\left( V_t^{(\omega)}  \right)_{t \geq 0}$ denote respectively the perimeter and volume processes associated to a peeling exploration of $\MM_{\q^{(\omega)}}$.
The convergence in probability
\[ \frac{V_t^{(\omega)}-2P_t^{(\omega)}}{t} \xrightarrow[t \to +\infty]{P)} \frac{ \left(\sqrt{\omega}-\sqrt{\omega-1} \right)^2}{2 \sqrt{\omega(\omega-1)}} \]
is uniform in $\omega$ over any compact subset $K$ of $(1,+\infty)$ in the sense that for all $\eps>0$, there is $t_0>0$ such that, for all $t \geq t_0$ and $\omega \in K$:
\[ \P \left( \left| \frac{V_t^{(\omega)}-2P_t^{(\omega)}}{t} - \frac{ \left(\sqrt{\omega}-\sqrt{\omega-1} \right)^2}{2 \sqrt{\omega(\omega-1)}} \right| >\eps \right) < \eps. \]
\end{lem}

The proof of Lemma~\ref{lem_unif_volume} is an adaptation of the proof of~\eqref{limit_volume_growth} in~\cite{C-StFlour}, but using a uniform weak law of large numbers. It is delayed to Appendix~\ref{subsec_unif_volume}.

\section{Tightness, planarity and one-endedness}\label{sec_univ_tight}

In all this section, we will work in the general setting of Theorem \ref{thm_main_more_general}, i.e. we do not assume $\sum_{j \geq 1} j^2 \alpha_j<+\infty$.

\begin{prop}\label{prop_tightness_dloc_univ}
Let $(\ff^n, g_n)_{n \geq 1}$ be as in Theorem \ref{thm_main_more_general}. Then the sequence $\left( M_{\ff^n, g_n} \right)_{n \geq 1}$ is tight for $d_{\loc}$, and every subsequential limit is a.s. planar and one-ended.
\end{prop}

Our strategy to prove Proposition~\ref{prop_tightness_dloc_univ} will be similar to \cite{BL19}, and in particular relies on a Bounded ratio Lemma (Lemma \ref{lem_BRL}). Sections \ref{subsec_BRL}, \ref{subsec_good_sets} and \ref{subsec_proof_BRL} are devoted to the proof of the Bounded ratio Lemma, which is significantly more complicated than in \cite{BL19}. In Section \ref{subsec_planarity}, we prove that any subsequential limit of $\left( M_{\ff^n, g_n} \right)_{n \geq 1}$ for $d^*_{\loc}$ (which exist by Lemma~\ref{lem_easy_dual_convergence}) is planar and one-ended. Finally, in Section \ref{subsec_finite_degrees}, we finish the proof of Proposition \ref{prop_tightness_dloc_univ} using Lemma \ref{lem_dual_convergence_univ} and the Bounded ratio Lemma.

\subsection{The Bounded ratio Lemma}
\label{subsec_BRL}

The Bounded ratio Lemma below means that, as long as the faces are not too large and the number of vertices remains proportional to the number of edges (i.e. basically under the assumptions of Theorem~\ref{thm_main_more_general}), removing a face of degree $2j_0$ changes the number of maps by at most a constant factor, provided the faces of degree $2j_0$ represent a positive proportion of the faces. We recall from \eqref{defn_v_f_g} that $|\ff|$ and $v(\ff,g)$ are respectively the number of edges and of vertices of a map with genus $g$ and face degrees given by $\ff$. For $j \geq 1$, we denote by $\mathbf{1}_j$ the face degree sequence consisting of a single face of degree $2j$, i.e. $\left( \mathbf{1}_j \right)_i$ is $1$ if $i=j$ and $0$ otherwise.

\begin{lem}[Bounded ratio Lemma]\label{lem_BRL}
We fix $\kappa, \delta>0$ and a function $A:(0,1]\rightarrow \N$. Let $\mathbf{f}$ be a face degree sequence, and let $g \geq 0$. We assume that
\begin{equation}\label{eq_petites_faces}
v(\ff, g) > \kappa |\ff| \quad \mbox{ and } \quad \forall \eps>0, \, \sum_{i>A(\eps)} i f_i< \eps |\ff|.
\end{equation}
Let also $j_0 \geq 1$ be such that $j_0 f_{j_0}>\delta |\ff|$. Then the ratio
\[\frac{\beta_g(\mathbf{f})}{\beta_g(\mathbf{f}-\mathbf{1}_{j_0})}\]
is bounded by a constant depending only on $\delta, \kappa$ and the function $A$.
\end{lem}

We will not try to obtain an explicit constant. As in \cite{BL19}, we will use the Bounded ratio Lemma to estimate the probability of certain events during peeling explorations, so we will need versions with a boundary. Here are the precise versions that we will need later in the paper.

\begin{corr}\label{lem_BRL_boundaries}
Let $\kappa, \delta>0$ and $A(\cdot)$ be as in Lemma \ref{lem_BRL}. Then there is a constant $C$ such that the following holds.
\begin{enumerate}
\item
Let $p,p',j\geq 1$. Then there is $N$ such that, for all $\ff$ and $g$ satisfying \eqref{eq_petites_faces} and $j f_j > \delta |\ff|$ and $|\ff|>N$, we have
\[\frac{\beta_g^{(p,p')}(\mathbf{f})}{\beta_g^{(p,p')}(\mathbf{f}-\mathbf{1}_j)}< C\]
and in particular
\begin{equation}\label{eqn_BRL_one_boundary}
\frac{\beta_g^{(p)}(\mathbf{f})}{\beta_g^{(p)}(\mathbf{f}-\mathbf{1}_j)}< 2C.
\end{equation}
\item
Let $p_1, p_2 \geq 1$ and $i_1, i_2 \geq 0$. Then there is $N$ such that, for all $\ff$ and $g$ satisfying \eqref{eq_petites_faces} and $|\ff|>N$, we have
\[ \frac{\beta_g^{(p_1+i_1,p_2+i_2)}(\mathbf{f})}{\beta_g^{(p_1,p_2)}(\mathbf{f})} < C^{i_1+i_2}. \]
\end{enumerate}
\end{corr}

Since this will be very important later, we highlight that the constant $C$ does not depend on $p, p', j$ but that $N$ does. The inequality~\eqref{eqn_BRL_one_boundary} will be used in the tightness argument just like in \cite{BL19}, whereas the statements with two boundaries will be needed in the two hole argument (Section \ref{subsec_same_perimeter}).

\begin{proof}
We first claim that we have the identity
\begin{equation}\label{eqn_add_boundary}
\beta_g^{(p,p')}(\mathbf{f})=\frac{2p(f_{p}+1)p'(f_{p'}+\mathbbm{1}_{p \ne p'})}{2(|\ff|+p+p')}\beta_g(\mathbf{f}+\mathbf{1}_{p}+\mathbf{1}_{p'}).
\end{equation}
Indeed, the factor $p(f_p+1)$ corresponds to the number of ways to add a second root to a map of $\mathcal{B}_g(\mathbf{f}+\mathbf{1}_p)$ such that this second root has a face of degree $2p$ on its right (with respect to the canonical white to black orientation of edges). The factor $p'(f_{p'}+\mathbbm{1}_{p \ne p'})$ corresponds to the number of ways of adding a third root next to a face of degree $2p'$ so that the two root faces are distinct. The $(|\ff|+p+p')$ in the denominator corresponds to forgetting the original root. Moreover, if $(\ff,g)$ satisfy the assumptions of Lemma \ref{lem_BRL} for $\delta, \kappa, A(\cdot)$ and $|\ff|$ is large enough, then $(\ff+\mathbf{1}_p+\mathbf{1}_{p'},g)$ also satisfies the assumptions of Lemma \ref{lem_BRL} for $\frac{\delta}{2}, \frac{\kappa}{2}, A \left( \frac{\cdot}{2}\right)$. Therefore, the first point of the corollary follows from Lemma \ref{lem_BRL} and \eqref{eqn_add_boundary}. To deduce \eqref{eqn_BRL_one_boundary}, just take $p'=1$ and use the identity $\beta_g^{(p,1)}(\mathbf{f})=|\ff|\beta_g^{(p)}(\mathbf{f})$ (adding a $2$-gon is equivalent to marking an edge) and the fact that $|\ff|$ is large enough.

For the second point, we first note that it is sufficient to prove it for $\{i_1, i_2\}=\{0,1\}$. Since $C$ does not depend on $(p_1, p_2)$, the general case easily follows by induction on $i_1+i_2$. Without loss of generality, we assume $i_1=1, i_2=0$.

We now note that there is $\delta, j_1>0$ depending only on $A(\cdot)$ such that, if \eqref{eq_petites_faces} is satisfied, then there is $2 \leq j \leq j_1$ such that $jf_j > \delta |\ff|$ (we can assume $j \geq 2$ because if there are too many $2$-gons, then the number of vertices cannot be macroscopic). We fix such a $j$.
Then, by the injection that consists in gluing a $2j$-gon on the first boundary as on Figure~\ref{fig_reducing_boundary}, we have 
\[\beta_g^{(p_1+1,p_2)}(\mathbf{f}) \leq \beta_g^{(p_1,p_2)}(\mathbf{f}+\mathbf{1}_j) \leq C \beta_g^{(p_1,p_2)}(\mathbf{f}),\]
where the last inequality uses the first item of the Corollary. This proves the second point.
\begin{figure}
\center
\includegraphics[scale=0.8]{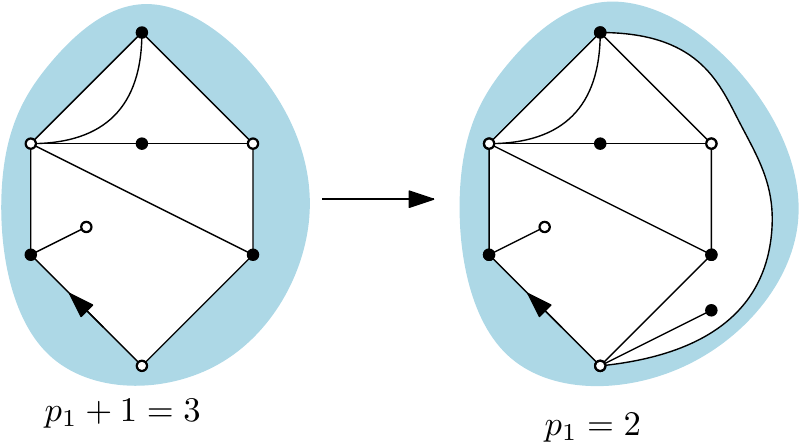}
\caption{Reducing the size of a boundary by $2$ by adding a $2j$-gon (here, the boundary is in blue, and $j=3$).}\label{fig_reducing_boundary}
\end{figure}
\end{proof}

\paragraph{Outline of the proof of Lemma \ref{lem_BRL}.}
The general idea is the same as in \cite{BL19}, namely building an injection that removes a small piece of a map (here, we would like to remove a face of degree $2j_0$). Just like in \cite{BL19}, this implies to merge vertices, so we will try to bound the degrees of the vertices involved, so that the number of ways to do the surgery backwards is not too high. However, since we work in a more general setting, several new constraints appear. First, the degrees of the faces are not bounded, so we must make sure that our surgery operations do not involve faces of huge degrees. This is the purpose of finding "very nice edges" in Section \ref{subsec_good_sets} below. Also, we will not always be able to remove a face of degree exactly $2j_0$. We will therefore remove either a face with degree higher than $2j_0$, or several faces which combined are "larger" than a face of degree $2j_0$. We will then use the two (easy) Lemmas~\ref{lem_grande_face} and~\ref{lem_petites_faces} to conclude.

\begin{lem}\label{lem_grande_face}
If $p\geq j_0 \geq 1$, then
\[ |\ff| \beta_g \left( \mathbf{f}-\mathbf{1}_{j_0} \right) \geq j_0 f_{j_0} \, \beta_g \left( \mathbf{f}-\mathbf{1}_p \right).\]
In particular, if $j_0 f_{j_0} \geq \delta |\ff|$, then
\begin{equation}\label{eq_transfer_grande_face}
\beta_g(\mathbf{f}-\mathbf{1}_{j_0}) \geq \delta \beta_g(\mathbf{f}-\mathbf{1}_p)
\end{equation}
\end{lem}

\begin{proof}
The second point is immediate from the first. For the first point, the right-hand side counts maps in $\mathcal{B}_g \left( \mathbf{f}-\mathbf{1}_p \right)$ with a marked edge such that the face on its right has degree $2j_0$. The left-hand side counts maps in $\mathcal{B}_g \left( \mathbf{f}-\mathbf{1}_{j_0} \right)$ with a marked edge, so it is enough to build an injection from the first set to the second. Take a map $m$ in $\mathcal{B}_g \left( \mathbf{f}-\mathbf{1}_p \right)$ and mark an edge $e$ of $m$ with a face of degree $2j_0$ on its right. We glue a path of $p-j_0$ edges to the starting point of $e$, just on the right of $e$ as on Figure~\ref{fig_plus_grande_face}. One obtains a map of $\mathcal{B}_g \left( \mathbf{f}-\mathbf{1}_{j_0} \right)$ with a marked edge, and going backwards is straightforward.

\begin{figure}[!ht]
\center
\includegraphics[scale=0.7]{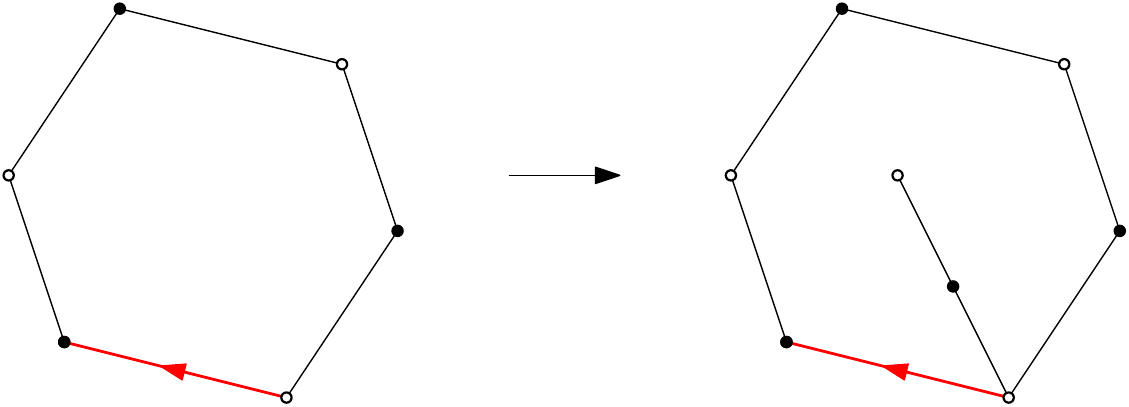}
\caption{The injection of Lemma~\ref{lem_grande_face} (here with $j_0=3$ and $p=5$). The marked edge is in red.}\label{fig_plus_grande_face}
\end{figure}
\end{proof}

\begin{lem}\label{lem_petites_faces}
If $1<d_1, d_2, \ldots,d_k<j_0$ and ${\sum_{i=1}^k (d_i-1) \geq j_0-1}$, then 
\[2|\ff| \beta_g \left( \mathbf{f}-\mathbf{1}_{j_0} \right) \geq 2j_0 f_{j_0} \beta_g \left( \mathbf{f}-\sum_{i=1}^k \mathbf{1}_{d_i} \right).\]
In particular, if $j_0 f_{j_0} \geq \delta |\ff|$, then
\begin{equation}\label{eq_transfer_petites_faces}
\beta_g  \left( \mathbf{f}-\mathbf{1}_{j_0} \right) \geq \delta \beta_g \left( \mathbf{f}-\sum_{i=1}^k\mathbf{1}_{d_i} \right).
\end{equation}
\end{lem}

\begin{proof}
The proof is very similar to Lemma~\ref{lem_grande_face}. This time, consider a map $m$ in $\mathcal{B}_g \left( \mathbf{f}-\sum_{i=1}^k\mathbf{1}_{d_i} \right)$ with a marked oriented edge $e$ that has a face of degree $2j_0$ on its right. Let $d=1+\sum_{i=1}^k {(d_i-1)}$, and note that $d \geq j_0$. If $d>j_0$, transform the face of degree $2j_0$ into a face of degree $2d$ by adding a path of $d-j_0$ edges like in the previous proof. Then tessellate this face of degree $2d$ as on Figure~\ref{fig_plus_petites_faces}. We obtain a map of $\mathcal{B}_g \left( \mathbf{f}-\mathbf{1}_{j_0} \right)$ with a marked edge, and this operation is also injective.

\begin{figure}[!ht]
\center
\includegraphics[scale=0.8]{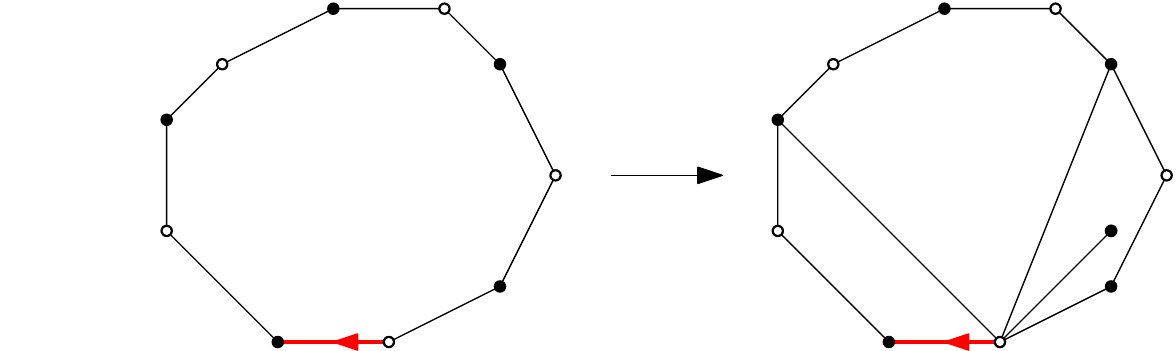}
\caption{The injection of Lemma~\ref{lem_petites_faces} (here with $j_0=5$ and $(d_1,d_2,d_3)=(2,3,3)$).}\label{fig_plus_petites_faces}
\end{figure}
\end{proof}

\subsection{Good sets of edges}
\label{subsec_good_sets}

The injection we will build to prove the Bounded ratio Lemma takes as input a pair $(m, E)$, where $m$ is a map and $E$ is a set of edges of $m$ satisfying the properties we will need to perform some surgery around $E$. We will call such a set a \emph{good set}. Our goal in this subsection is to define a good set and to prove that any map contains a linear number of good sets (Proposition~\ref{prop_good_sets}). We recall that we consider that the edges are oriented from white to black, and therefore it makes sense to define the left or right side of an edge.

Throughout this section, we work under the assumptions of Lemma~\ref{lem_BRL}. Let $m \in \mathcal{B}_g(\mathbf{f})$. Let $A_1:=2A \left( \min \left( \frac{\kappa}{32},\delta \right) \right)$.

\begin{rem}\label{rem_causal_graph}
We will have several different constants (depending on $A(\cdot)$, $\delta$ and $\kappa$) defined in terms of each other in this subsection. To help convince the reader there is no circular dependency between them, we provide a "causal graph" of all the involved constants.
\begin{figure}[!ht]
\center
\includegraphics[scale=0.8]{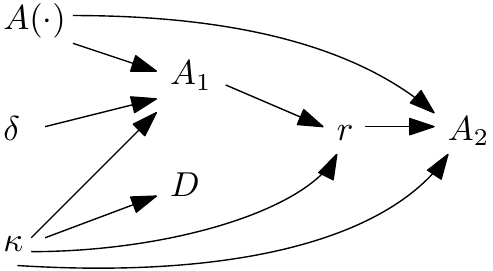}
\end{figure}
\end{rem}

We say that an edge $e$ of $m$ is \emph{nice} if it is not incident to a face of degree larger than $2A_1$.
\begin{fact} At least $\left( 1-\frac{\kappa}{16} \right) |\mathbf{f}|$ of the edges in $m$ are nice.
\end{fact}
\begin{proof}
Draw an edge $e$ of $m$ uniformly at random. The face $f$ sitting to the right of $e$ is drawn at random with a probability proportional to its degree. By the second assumption of \eqref{eq_petites_faces}, the probability that $f$ has degree larger than $2A_1$ is less than $\frac{\kappa}{32}$. The same is true for the face sitting to the left of $e$.
\end{proof}

We will need to bound the degrees not only of the faces incident to an edge, but also of the faces close to this edge for the dual distance. More precisely, we will define the \emph{dual distance between two edges $e_1, e_2$ of $m$} as the dual distance between the face on the right of $e_1$ and the face on the right of $e_2$. We fix a value $r$ (depending on $A_1$ and $\kappa$) that we will specify later. Let $A_r$ be the function given by Lemma~\ref{lem_tight_degree_in_ball} for $A(\cdot)$ and $r$, and let $A_2=A_r \left( \frac{\kappa}{16} \right)$. We will call an edge $e$ of $m$ \emph{very nice} if it is nice and no edge at dual distance $r$ or less from $e$  is incident to a face of degree larger than $2A_2$. By applying Lemma~\ref{lem_tight_degree_in_ball} to the stationary random graph obtained by rooting the dual map $m^*$ at a uniform edge, the proportion of edges of $m$ at dual distance $r$ or less from a face larger than $2A_2$ is at most $\frac{\kappa}{16}$. Hence, we get the following observation.

\begin{fact}At least $\left( 1-\frac{\kappa}{8} \right) |\mathbf{f}|$ of the edges of $m$ are very nice.
\end{fact}

Let $D=\frac{4}{\kappa}$. By the first assumption of \eqref{eq_petites_faces}, we know that $D$ is larger than twice the average vertex degree in $m$. Since at most half of the vertices have degree at least twice the average degree, and since there are more than $\kappa |\mathbf{f}|$ vertices, we have the following.

\begin{fact}
There are at least $\frac{\kappa}{4}|\mathbf{f}|$ vertices of the same colour with degree less than $D$ in $m$.
\end{fact}

Without loss of generality, assume that this colour is white (we recall that the vertices are coloured black and white so that each edge joins two vertices of different colours and the root is oriented from white to black).  We say that a white vertex is \emph{fine} if it has degree at most $D$, and that an edge is \emph{fine} if it is incident to a fine white vertex and the face on its right is not of degree $2$. By the previous fact, and since every vertex is incident to at least a face of degree $>2$, there are at least $\frac{\kappa}{4} |\mathbf{f}|$ fine edges in $m$, incident to $\frac{\kappa}{4}|\mathbf{f}|$ distinct white vertices.
An edge is said to be \textit{good} if it is both very nice and fine. Summing up the last results, we have the following.

\begin{lem}\label{lem_good_edges}
There are at least $\frac{\kappa }{8} |\mathbf{f}|$ good edges in $m$, incident to $\frac{\kappa}{8} |\mathbf{f}|$ distinct white vertices.
\end{lem}

We now fix the value of $r$ at $r=\frac{16A_1}{\kappa}+1$ (which is possible since $A_1$ does not depend on $r$, see Remark~\ref{rem_causal_graph} above). We call a set $S$ of $A_1$ edges of $m$ a \textit{good set} if all the edges of $S$ are good, they are incident to distinct white vertices and they are all at dual distance less than $2r$ from each other. Our next goal is to find a large number of good sets in the map $m$. Note that these good sets do not need to be disjoint.

\begin{prop}\label{prop_good_sets}
There are at least $\frac{\kappa}{16} |\mathbf{f}|$ good sets of edges in $m$.
\end{prop}

\begin{proof}
The proof follows the argument from~\cite{BL19}.
Let $G$ be a set of $\frac{\kappa }{8} |\mathbf{f}|$ good edges incident to distinct white vertices given by Lemma~\ref{lem_good_edges}. In this proof, the balls $B_r^*(e)$ that we will consider will be for the dual distance. We can assume that for every $e\in G$, the ball $B_r^*(e)$ does not contain all the edges of $m$, since otherwise the proposition is obviously true.

In that case, for all $e\in G$, since $m^*$ is connected we must have $|B^*_r(e)|>r$.
We are going to find a collection of distinct good sets $(S_i)$. For this, we build by induction a decreasing sequence of sets of good edges $(G_i)$, such that for each $i$, the set $G_{i+1}$ is obtained from $G_i$ by removing one element. We set $G_0=G$. Let $0 \leq i<\frac{\kappa }{16} |\mathbf{f}|$, and assume that we have built $G_0, G_1, \dots, G_i$. Then $|G_i|=|G|-i$, so
\begin{align*}
\sum_{e\in G_i} |B_r^*(e)|> \left( |G|-i \right)r >\frac{\kappa }{16} |\mathbf{f}| r> A_1 |\mathbf{f}|
\end{align*}
by our choice of $r$. Therefore, there must be an "$A_1$-overlap", i.e. there exist $A_1$ edges whose balls of radius $r$ have a nonempty intersection. Thus they are all at distance at most $2r$ of each other, and we just found a good set $S_{i+1}$. Choose $e_{i+1} \in S_{i+1}$ arbitrarily, and let $G_{i+1}=G_{i}\setminus \{e_{i+1}\}$. This way we can build $G_i$ and $S_i$ for $1 \leq i < \frac{\kappa}{16} |\mathbf{f}|$, which proves the lemma.
\end{proof}

\subsection{Proof of the Bounded ratio Lemma: the injection}
\label{subsec_proof_BRL}

We now prove the Bounded ratio Lemma (Lemma~\ref{lem_BRL}). We start with the easy case $j_0=1$: a marked digon can be contracted into a marked edge (see Figure~\ref{fig_digon}), and if $f_1>\delta |\mathbf{f}|$, we have
\[|\mathbf{f}|\beta_g(\mathbf{f}-\mathbf{1}_1)>\delta |\mathbf{f}|  \beta_g(\mathbf{f})\]
which yields the result. 

\begin{figure}[!ht]
\centering
\includegraphics[scale=0.5]{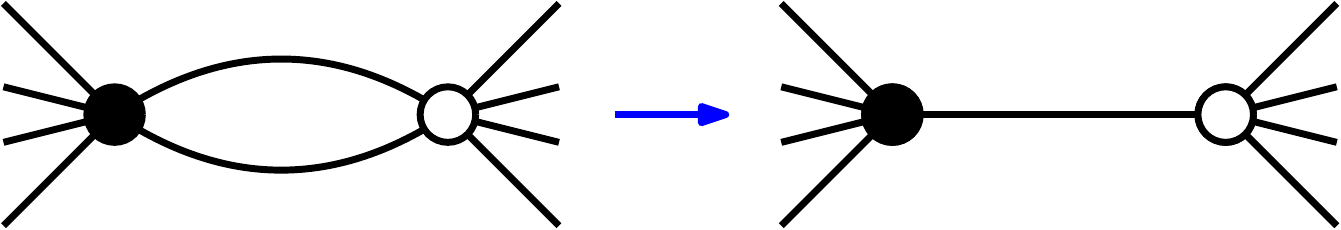}
\caption{Contraction of a digon.}\label{fig_digon}
\end{figure}

We can now assume that $j_0>1$ and $j_0 f_{j_0} > \delta |\mathbf{f}|$. The injection we will build takes as input a map of $\mathcal{B}_g(\mathbf{f})$ with a marked good set of edges, and outputs a map of $\mathcal{B}_g(\mathbf{\tilde{f}})$ with a marked edge and some finite information (i.e. with values in a finite set whose size depends only on $\delta, \kappa$ and the function $A$), with $\mathbf{\tilde{f}}$ of the form:
\begin{enumerate}
\item[(i)]
either
\begin{equation}\label{brl_ftilde_first_condition}
\mathbf{\tilde{f}}=\mathbf{f}-\mathbf{1}_p \mbox{ where } j_0 \leq p < A_1,
\end{equation}
\item[(ii)]
or
\begin{equation}\label{brl_ftilde_second_condition}
\mathbf{\tilde{f}}=\mathbf{f}-\sum_{i=1}^k\mathbf{1}_{d_i} \mbox{ where } k \leq A_1 \mbox{ and } 1<d_i<j_0 \mbox{ for all } i \mbox{ but } \sum_{i=1}^k d_i\geq j_0+k-1.
\end{equation}
\end{enumerate}
Since the number of possibilities for $\mathbf{\tilde{f}}$ is bounded in terms of $A_1$, by \eqref{eq_transfer_grande_face} and~\eqref{eq_transfer_petites_faces}, such an injection will prove Lemma~\ref{lem_BRL}.

The surgery operation is quite complicated and, contrary to \cite{BL19}, some intermediate steps affect the topology of the map (although the genus remains unchanged in the end). It is broken down into four steps for better understanding.

Before describing this operation in details, we first need to give a few definitions.
Consider a map $m$ with a good set $S$ of edges, and a distinguished edge $e^*\in S$ called the \textit{anchor} (the way to choose $e^*$ will be specified later). For $e\in S\setminus\{e^*\}$, let $p_{e}$ be the leftmost shortest path\footnote{the leftmost shortest path is constructed by, at each step, taking the leftmost face that gets strictly closer to the target.} in the dual $m^*$ from the face on the right of $e^*$ to the face on the right of $e$. We denote by $P(S)$ the union of all the paths $p_e$ for $e\in S\setminus\{e^*\}$ (see the left of Figure~\ref{fig_path_carving}). With this definition $P(S)$ forms a tree. Since the edges of $S$ are at dual distance at most $2r$ from each other, we also know that the number of dual edges in $P(S)$ satisfies $|P(S)|<2A_1r$.

We will now describe four injective operations (four "steps"). To make things less cumbersome, we will use the term \emph{finite} in lieu of "bounded by a constant that depends only on $\delta,\kappa$ and the function $A$" (without trying to make the bounds explicit). These steps will involve marked faces of possibly high but finite degree that we will call \emph{megafaces}. Since one of the steps might temporarily disconnect the map, we also precise that by a \emph{possibly non-connected map}, we mean a finite gluing of polygons which is not necessarily connected, and where only one of the connected components bears a root edge.

After having defined the four steps, we will explain how to apply them to a map of $\B_g(\mathbf{f})$ with a marked good set to obtain the desired injection. In particular, some of the steps depend on additional parameters (such as the anchor edge in the good set $S$ for step $1$), and the way to choose the parameters will be specified in the end.

\paragraph{Step 1: Carving.} 
\textit{This step inputs a map $m\in \B_g(\mathbf{f})$, with a distinguished good set $S$, anchored at some edge $e^*$, and outputs a map $m_1$ of genus $g$ with a marked face of finite degree, $A_1$ distinguished fine vertices incident to this face, and some finite information.}

Start from a map $m\in \B_g(\mathbf{f})$, with a distinguished good set $S$, anchored at some edge $e^*$. Draw $P(S)$ on $m^*$ and cut each edge of $m$ that is crossed by $P(S)$. For every such edge, mark the two corners corresponding to the ends of this edge (see Figure~\ref{fig_path_carving}).
The tree structure of $P(S)$ ensures that the map $m_1$ that we obtain is still connected and has genus $g$. Moreover $m_1$ has a distinguished face called the \emph{megaface}, resulting from merging several faces of $m$. The $A_1$ white vertices incident to the edges of $S$ will be the $A_1$ distinguished vertices output by Step 1. These vertices are fine (i.e. have degree less than $\frac{4}{\kappa}$) by definition of a good set.

Moreover, we claim that there are at most $2A_1 r$ marked corners of each colour in the megaface, and that the megaface is of degree less than $4A_1A_2r$.
Indeed, each time an edge is deleted, two faces merge, and the megaface is the consequence of the mergings of the faces created by cutting edges. The number of marked corners is twice the number of edges cut, which is less than $|P(S)|$. Let $2j_1,2j_2,\ldots,2j_{\ell}$ be the degrees of the faces of $m$ that were merged in the carving process. Then the megaface of $m_1$ has degree $2F=2(j_1+j_2+\ldots+j_{\ell}-\ell+1)$. By definition of a good set, $\ell\leq |P(S)|+1\leq 2A_1r$, and also for all $i$, we have $j_i<A_2$, which implies $2F<4A_1A_2r$. Therefore, given the map $m_1$ and its megaface, the number of ways to choose the marked corners is bounded by $\binom{4A_1 A_2 r}{2 A_1 r}$, so the marked corners are finite information, which will allow us to forget about them in the next steps.

Finally, note that step 1 is injective up to finite information: given the map $m_1$, its megaface and the marked corners, there is a finite number of ways to pair the marked corners in a planar way inside the megaface.

\begin{figure}[!ht]
\center
\includegraphics[scale=0.7]{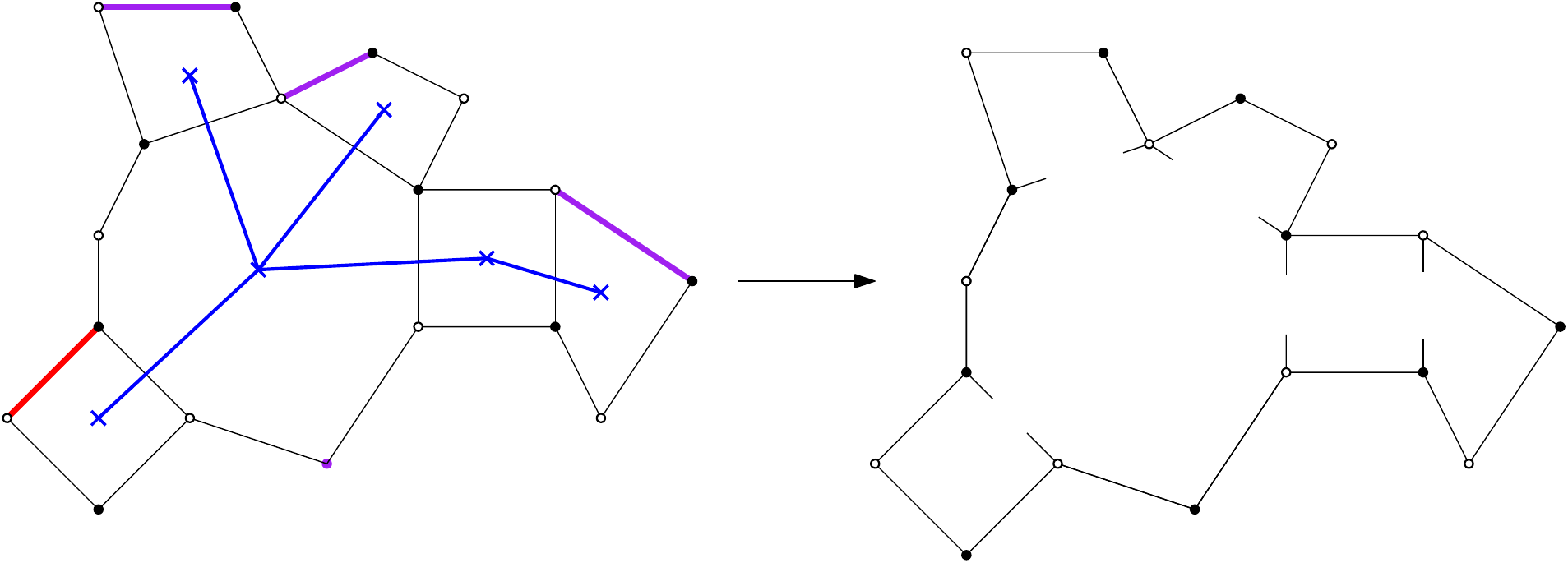}
\caption{The carving operation, from $m$ (on the left) to $m_1$ (on the right). The good set $S$ of edges is in purple with the anchor in red), and the paths $P(S)$ are in blue. Marked corners in $m_1$ are represented as dangling half-edges.}\label{fig_path_carving}
\end{figure}

\paragraph{Step 2: Vertex deletion.}
\textit{This step inputs a map $m_1$ of genus $g$, with a marked face of finite degree and $A_1$ distinguished fine vertices incident to this face, as well as a number $d<A_1$. It outputs (together with some finite information) a possibly disconnected map $m_2$ with $K$ connected components, of total genus $g' \leq g$ with $g-g'$ finite, with $v(\mathbf{f},g)-d$ vertices in total and with $M$ marked faces of finite total degree, such that there is at least one marked face in each connected component, and we have
\begin{equation}\label{eq_degenerescence}
M \leq (g-g')+K.
\end{equation}
}

Start with a map $m_1$ of genus $g$, with a marked megaface to which $A_1$ marked fine vertices are incident, as well as a number\footnote{the number $d$ is a parameter that will be fixed when defining the injection.} $d<A_1$. Pick $d$ of the marked good vertices arbitrarily, and delete them as well as all their incident edges (see Figure~\ref{fig_vertex_deletion}). We obtain a map $m_2$ that might be disconnected (see right before Step 1 for a definition). 

\begin{figure}[!ht]
\center
\includegraphics[scale=0.8]{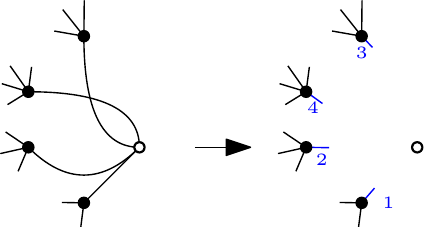}
\caption{Deleting a vertex.}\label{fig_vertex_deletion}
\end{figure}

We now precise which are the marked faces. Everytime an edge is deleted in the process, three cases may happen:
\begin{itemize}
\item two faces are merged,
\item or a face is split in two and the genus decreases by one,
\item or a face is split in two and the number of connected components increases by one.
\end{itemize}

The marked faces (or megafaces) of $m_2$ are then the result of the consecutive splittings and mergings. In particular, the last two cases explain \eqref{eq_degenerescence}. Moreover, in the third case, if a connected component is split in two, the marked face gets split between both components, so in the end there is at least one marked face on each connected component.

We now bound the total degree of the megafaces. The edges incident to the megafaces of $m_2$ come from the megaface of $m_1$, plus the "small" faces we merged with it in the process. By definition of a good set, all these small faces were of degree $A_2$ or less, and the number of mergings is bounded by the total number of edges deleted in the operation, which itself is at most $DA_1$ ($d<A_1$ vertices of degree bounded by $D$ were deleted). Hence the total degree of the megafaces is bounded by $DA_1 A_2 + 2F$, where $2F$ is the degree of the megaface of $m_1$. In particular, this total degree is finite. On the other hand, the genus variation $g-g'$ is bounded by the number of removed edges, so it is finite as well.

Finally, we argue that this step is injective up to finite information. To go back to $m_1$, one only needs to recreate the white vertices and reattach them to the right black corners. The number of edges to add back is finite (bounded by $D A_1$), and the black corners to which we may attach them are on the megafaces, so their number is also finite. Therefore, up to adding finite information in the output, Step 2 is injective.

\paragraph{Step 3: Reconstruction.} 
\textit{This step inputs a possibly disconnected map $m_2$ with $v(\mathbf{f},g)-d$ vertices in total for some $d$, of genus $g' \leq g$ with $g-g'$ finite, with $K$ connected components and $M$ marked faces of finite total degree, satisfying~\eqref{eq_degenerescence} and such that there is at least one marked face in each connected component. It outputs a connected map $m_3$ with $v(\mathbf{f},g)-d+1$ vertices, of genus $g$, with one marked face of finite degree and a marked edge lying in that face, plus some finite information.}

We start with a map $m_2$ as above, and denote by $2F$ the total degree of its megafaces. We create one new white vertex $v$, select arbitrarily a black corner in each megaface $f$ of $m_2$, and draw an edge between $v$ and these black corners (see the left part of Figure~\ref{fig_reconstruction}). We keep track of one of these newly created edges and call it the \emph{special edge}. We obtain a map $m'_2$ that is now connected and has a unique megaface of degree at most $2F+2M$ (because we added $M$ edges). This megaface is incident to the vertex $v$ and to the special edge. Let $g''$ be the genus of $m'_2$. We note that when we add one by one the edges joining $v$ to the megafaces of $m_2$, the genus increases by one everytime we connect a megaface which is not the first one in its connected component of $m_2$. Hence $g''-g'=M-K$, so~\eqref{eq_degenerescence} implies $g''\leq g$. If $g''=g$, we set $m_3=m'_2$. If $g''<g$, we attach a pair of edges to the right of the special edge as on the right of Figure~\ref{fig_reconstruction}. Now the resulting map has still only one megaface and is of genus $g''+1$. Repeat if necessary to obtain a map $m_3$ of genus $g$ with one megaface of degree bounded by $2F+2K+4(g-g')$ (because we added at most $2(g-g')$ edges to recover the genus). Using~\eqref{eq_degenerescence} again, the degree of the megaface is bounded by $2F+4M \leq 6F$ so it is finite. Moreover, the only vertex we have added is $v$, so $m_3$ has $v(\mathbf{f},g)-d+1$ vertices.

\begin{figure}[!ht]
\center
\includegraphics[scale=0.4]{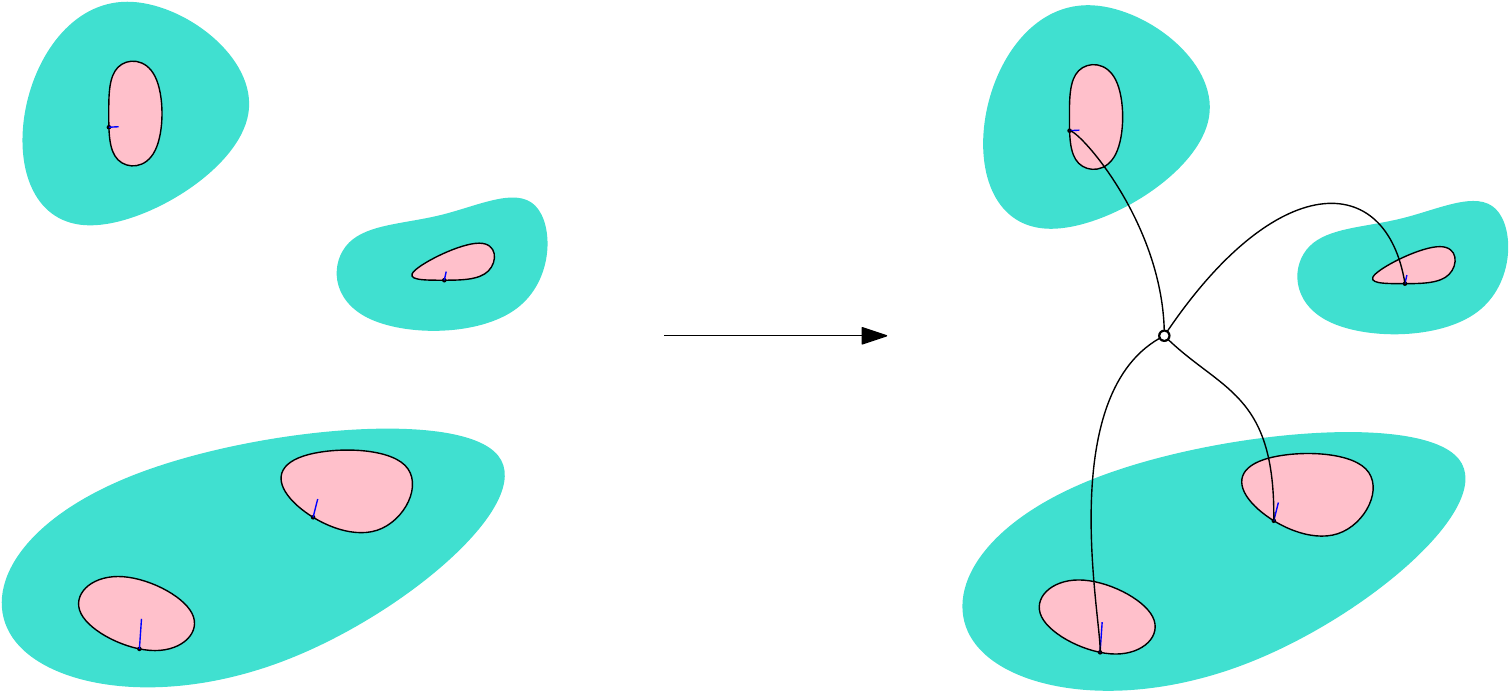}
\hspace{1cm}
\includegraphics[scale=0.4]{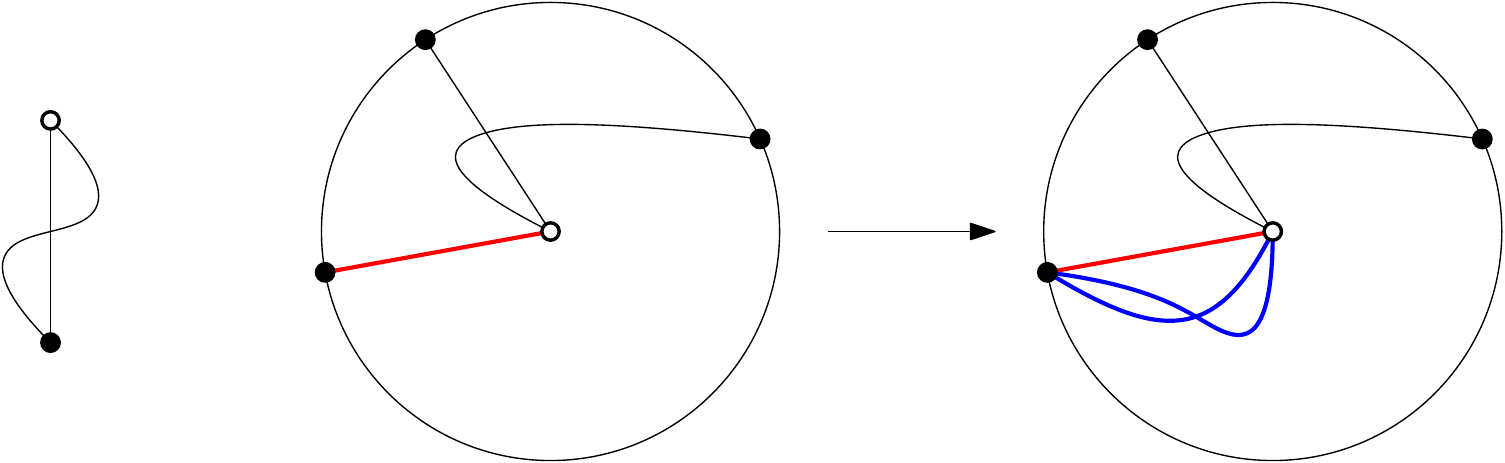}
\caption{The reconstruction. Left: reconnecting the map. The megafaces are in pink, and the rest of the map is in turquoise. Marked corners are represented as dangling half-edges. Right: recovering the genus. The special edge is in red, and the added pair of edges is in blue.}\label{fig_reconstruction}
\end{figure}

Finally, to recover $m_2$ from $m_3$, one only has to remember the value of $g-g''$. This is finite information since $g-g'' \leq g-g'$ is finite. We can then simply remove the $2(g-g'')$ edges on the right of the special edge to recover $m'_2$, and then delete the vertex $v$ (i.e. the white vertex incident to the marked edge of $m_3$), together with all its incident edges, to recover $m_2$. Therefore, Step 3 is injective up to finite information.

\paragraph{Step 4: Filling of the megaface.} 
\textit{This step inputs a map $m_3$ in $\mathcal{B}_g(\mathbf{f'})$ for some $\mathbf{f'}$ with a marked face of finite degree $2F$ and a marked edge incident to that face. We also assume that $\mathbf{f'}$ satisfies the following assumptions:
\begin{enumerate}
\item
we have $\mathbf{f'}=\mathbf{f}+\mathbf{1}_F-\mathbf{1}_{j_1}-\mathbf{1}_{j_2}-\ldots-\mathbf{1}_{j_{\ell}}$, where $\sum_{i=1}^{\ell} j_i$ is finite;
\item
$v(\mathbf{f'},g) = v(\mathbf{f},d)-d+1$ for some finite $d \geq 1$;
\item
there is an index $k< \ell$ such that 
\begin{equation}\label{eq_condition_BRL}
(j_1-1)+(j_2-1)+\ldots+(j_k-1)=d-1.
\end{equation}
\end{enumerate}
It outputs a map $m_4\in \B_g (\mathbf{\tilde{f}} )$ with a marked edge, where $\mathbf{\tilde{f}} \leq \mathbf{f}$ and $\left| \mathbf{f}-\mathbf{\tilde{f}} \right|$ is finite.}

We start with a map $m_3 \in \mathcal{B}_g(\mathbf{f'})$ as above, and let $d$ and $j_1, \dots, j_{\ell}$ be as in the assumptions on $\mathbf{f'}$. By the Euler formula on $\mathbf{f}$ and $\mathbf{f'}$, we must have 
\begin{equation}\label{eqn_megaface_decomposition}
F-1=(j_1-1)+(j_2-1)+\ldots+(j_{\ell}-1)-(d-1).
\end{equation}
Combining this with the assumption~\eqref{eq_condition_BRL}, we obtain
\[F-1=(i_{k+1}-1)+(i_{k+2}-1)+\ldots+(i_{\ell}-1).\]
Hence, we can tessellate the megaface of degree $2F$ into faces of degrees $2j_{k+1}, \dots, 2j_{\ell}$ as in the proof of \eqref{lem_grande_face}. We obtain a map $m_4$ (keeping the same marked edge) with face degrees given by $\mathbf{\tilde{f}}=\mathbf{f}-\sum_{i=1}^k \mathbf{1}_{j_i}$, in an injective way.

\vspace{1cm}
Now we can describe our injection in terms of the different operations mentioned above. Roughly speaking, we will successively apply our four steps. We will specify the choice of some of the parameters in the first three steps in such a way that the face degree distribution $\mathbf{f'}$ after Step $3$ satisfies all the assumptions required by Step $4$, and the face degree distribution $\mathbf{\tilde{f}}$ obtained after Step $4$ is of one of the two forms given by~\eqref{brl_ftilde_first_condition} and~\eqref{brl_ftilde_second_condition}.

More precisely, we start with a map $m \in \B_g(\mathbf{f})$ with a marked good set $S$. For all $e\in S$, let $d_e$ be the half-degree of the face sitting on the right of $e$. We treat two cases separately (which will in the end correspond to~\eqref{brl_ftilde_first_condition} and to~\eqref{brl_ftilde_second_condition}).

\paragraph{Case 1:} Assume there exists $e^*$ in $S$ such that $d_{e^*}\geq j_0$. Take $e^*$ as the anchor, and apply Step 1. Note that the face of degree $2d_{e^*}$ sitting to the right of $e^*$ is one of the faces that have been destroyed. We then apply Step 2 with $d=d_{e^*}$, i.e. we remove $d$ good vertices, including the one that was incident to $e^*$. We then apply Step $3$ to obtain a map $m_3$, and denote by $\mathbf{f'}$ the face degree sequence of $m_3$. Note that $m_3$ has $d-1$ vertices less than $m$, which ensures that the $d$ as defined in the assumptions of Step $4$ is the same as the one used in Step 2 and 3. Moreover, it follows from Steps 1, 2 and 3 that $\mathbf{f'}$ is of the form $\mathbf{f}+\mathbf{1}_F-\mathbf{1}_{j_1}-\mathbf{1}_{j_2}-\ldots-\mathbf{1}_{j_{\ell}}$, where $2F$ is the degree of the megaface, and $2j_1, \dots, 2j_{\ell}$ are the degrees of the faces destroyed in Step 1. In particular $\sum_{i=1}^{\ell} j_i$ is finite and, up to reordering the $j_i$'s, we may assume $j_1=d=d_{e^*}$. This implies that~\eqref{eq_condition_BRL} is satisfied for $k=1$, so we can apply Step 4 to obtain a map $m_4$. Finally, the face degree sequence of $m_4$ is $\mathbf{\tilde{f}}=\mathbf{f}-\mathbf{1}_d$, where $d=d_{e^*} \geq j_0$, and $d<A_1$ by definition of a nice edge. Hence $\mathbf{\tilde{f}}$ is of the form given by~\eqref{brl_ftilde_first_condition}.

\paragraph{Case 2:} We now assume $d_e<j_0$ for all $e \in S$. Take any arbitrary ordering $e_1,e_2, \ldots, e_{A_1}$ of the edges of $S$, and for all $i$ let $d_i=d_{e_i}$ be the half-degree of the face on the right of $e_i$. We first claim that there is an index $k$ such that
\begin{equation}\label{eqn_brl_case_2}
j_0-1<\sum_{i=1}^k (d_i-1)<A_1-1
\end{equation}
Indeed, by definition of a fine edge, we have $d_i \geq 2$ for all $i$ so $\sum_{j=1}^{A_1}(d_j-1)\geq A_1>j_0-1$. Therefore, we can consider the first index $k$ such that $\sum_{i=1}^k (d_i-1)>j_0-1$. Then we have $d_k < j_0$ by assumption, so $\sum_{i=1}^k (d_i-1) < 2j_0-1$. Finally, by definition of $A_1$, we have $A_1 \geq 2A(\delta) \geq 2j_0$ (the second inequality follows from the assumption $j_0 f_{j_0} > \delta|\mathbf{f}|$), so~\eqref{eqn_brl_case_2} is indeed satisfied for this $k$.

We now choose the anchor of $S$ arbitrarily, and apply Step 1. We then apply Step 2 with $d=1+\sum_{i=1}^k (d_i-1)$, i.e. we delete $d$ good vertices, including the ones incident to $e_1, \ldots, e_k$. Note that at this point, all the faces that were incident to the edges $e_1, \dots, e_k$ have been destroyed. We then apply Step 3 and obtain a map $m_3$ with face degree sequence $\mathbf{f'}$. The first two assumptions on $\mathbf{f'}$ in Step 4 are satisfied for the same reason as in Case 1, with $\mathbf{f'}$ of the form $\mathbf{f}+\mathbf{1}_F-\mathbf{1}_{j_1}-\ldots-\mathbf{1}_{j_{\ell}}$. Moreover, since the faces incident to $e_1, \dots, e_k$ have been destroyed previously, up to reordering the $j_i$'s, we may assume $j_i=d_i$ for $1 \leq i \leq k$. Therefore, by our choice of $d$, the third assumption~\eqref{eq_condition_BRL} of Step 4 is also satisfied. After applying Step 4, we obtain a map $m_4$ with face degree sequence
\[\mathbf{\tilde{f}}=\mathbf{f}-\sum_{i=1}^{k} d_i.\]
Finally, we have $1<d_i<j_0$ for all $i$ by the assumption of Case 2, and $\sum_{i=1}^k d_i \geq j_0+k-1$ by~\eqref{eqn_brl_case_2}, so $\mathbf{\tilde{f}}$ is of the form given by~\eqref{brl_ftilde_second_condition}.

\begin{proof}[Conclusion of the proof of Lemma~\ref{lem_BRL}]
Consider the injection that we have just built. By Proposition~\ref{prop_good_sets}, the number of inputs $m$ is at least $\frac{\kappa}{16} |\mathbf{{f}}|\beta_g(\mathbf{f})$, whereas the number of outputs $m_4$ is at most $\sum_{\mathbf{\tilde{f}}} |\mathbf{\tilde{f}}| \beta_g(\mathbf{\tilde{f}})$, where the sum is over face degree sequences $\mathbf{\tilde{f}}$ of the form either~\eqref{brl_ftilde_first_condition} of~\eqref{brl_ftilde_second_condition}.

When we apply successively Steps 1, 2, 3 and 4, we note that each step is injective up to finite information. This means that there is a constant $c$ depending only on $\delta$, $\kappa$ and $A(\cdot)$ such that
\[ \frac{\kappa}{16}|\mathbf{{f}}| \beta_g(\mathbf{f}) \leq c \sum_{\mathbf{\tilde{f}}} |\mathbf{\tilde{f}}| \beta_g(\mathbf{\tilde{f}}) \leq c |\mathbf{f}| \sum_{\mathbf{\tilde{f}}} \beta_g(\mathbf{\tilde{f}}), \]
where the second inequality uses $\mathbf{\tilde{f}} \leq \mathbf{f}$.
Moreover, the number of possible sequences $\mathbf{\tilde{f}}$ is bounded in terms of $\delta$, $\kappa$ and $A(\cdot)$. Hence, there is a sequence $\mathbf{\tilde{f}}$ of the form~\eqref{brl_ftilde_first_condition} or~\eqref{brl_ftilde_second_condition} such that
\[ \beta_g(\mathbf{f}) \leq c' \beta_g(\mathbf{\tilde{f}}), \]
where $c'$ depends only on $\delta$, $\kappa$ and $A(\cdot)$. We can finally conclude using Lemma~\ref{lem_grande_face} if $\mathbf{\tilde{f}}$ is of the form~\eqref{brl_ftilde_first_condition}, or Lemma~\ref{lem_petites_faces} if $\mathbf{\tilde{f}}$ is of the form~\eqref{brl_ftilde_second_condition}.
\end{proof}

\subsection{Planarity and One-Endedness}
\label{subsec_planarity}

We fix $(\mathbf{f}^{n})_{n \geq 1}$ and $(g_n)_{n \geq 1}$ satisfying the assumptions of Theorem~\ref{thm_main_more_general}. We recall that by Lemma~\ref{lem_easy_dual_convergence} $(\M_{\ff^{n}, g_n})$ is tight for $d_{\loc}^*$. In all this subsection, we will denote by $M$ a subsequential limit in distribution. Since $|\ff^{n}| \to +\infty$, it must be an infinite map. To prove Proposition~\ref{prop_tightness_dloc_univ}, we will first prove that $M$ is a.s. planar and one-ended (Corolaries~\ref{cor_planar} and \ref{cor_OE}). Like in \cite{BL19}, the proofs will rely heavily on a recursion counting maps. More precisely, the formula~\eqref{rec_biparti_genre_univ} proved in \cite{Lo19} will play the role of the Goulden--Jackson formula in \cite{BL19}. This step of the proof only differs from \cite{BL19} by minor adaptations, so we will not give the full details. The proofs of Lemmas \ref{lem_calcul_planar} and \ref{lem_calcul_OE} using~\eqref{rec_biparti_genre_univ} can be found in the Appendix~\ref{sec_univ_estim_lemmas}.

As in \cite{BL19}, to establish planarity, we want to bound, for any non-planar finite map $m$, the probability $\P \left( m \subset M_{\ff^{n},g_n} \right)$ for $n$ large. For this, we will need the following combinatorial estimate.

\begin{lem}\label{lem_calcul_planar}
We fix $(\mathbf{f}^{n})_{n \geq 1}$ and $(g_n)_{n \geq 1}$ which satisfy the assumptions of Theorem~\ref{thm_main_more_general}. We also fix $\mathbf{h}^{(0)}$ a face degree sequence such that $\mathbf{h}^{(0)} \leq\mathbf{f}^{n}$ for $n$ large enough.  Let $k\geq 2$, numbers $\ell_1, \ell_2,\ldots,\ell_k$ and perimeters $p_i^j$ for $1\leq j\leq k$ and $1\leq i \leq \ell_j$. Then
\[\sum_{\substack{\mathbf{h}^{(1)}+\mathbf{h}^{(2)}+\ldots+\mathbf{h}^{(k)}=\ff^{n}-\mathbf{h}^{(0)} \\ g^{(1)}+g^{(2)}+\ldots+g^{(k)}=g_n-1-\sum_j(\ell_j-1)}}\prod_{j=1}^k \beta_{g^{(j)}}^{(p^j_1,p^j_2,\ldots,p^j_{\ell_j})}(\mathbf{h}^{(j)})=o\left(\beta_{g_n}(\mathbf{f}^{n})\right)\]
as $n\rightarrow \infty$.
\end{lem}

\begin{corr}\label{cor_planar}
Let $(\mathbf{f}^{n})_{n \geq 1}$ and $(g_n)_{n \geq 1}$ satisfy the assumptions of Theorem~\ref{thm_main_more_general}. Then every subsequential limit $M$ of $\left( M_{\ff^n}, g_n \right)$ for $d^*_{\loc}$ is a.s. planar.
\end{corr}

\begin{proof}
The proof is basically the same as~\cite[Corollary 7]{BL19}, with the exception of the assumption $\mathbf{h}^{(0)} \leq\mathbf{f}^{n}$ in Lemma~\ref{lem_calcul_planar} (this assumption was automatically satisfied in~\cite{BL19}). More precisely, it is sufficient to prove $\P \left( m \subset M \right)=0$ for every finite map $m$ with holes and of genus $1$. We fix such a map $m$.

We also note that almost surely, for every face $f$ of $M$, we have $\alpha_{\deg(f)/2}>0$.  Indeed, this is true for the root face since the root face of $M$ has degree $j$ with probability $j \alpha_j$ for all $j \geq 1$, and this can be extended to all faces using stationarity with respect to the simple random walk on the dual of $M$. Therefore, if $m$ has a face of degree $2j$ with $\alpha_j=0$, then $\P \left( m \subset M \right)=0$.

If not, let $\mathbf{h}^{(0)}$ be the internal face degree sequence of $m$. Then $f^n_j \to +\infty$ as $n \to +\infty$ for every $j$ such that $h^{(0)}_j>0$, so $\mathbf{h}^{(0)} \leq\mathbf{f}^{n}$ for $n$ large enough. In particular, we are in position to use Lemma~\ref{lem_calcul_planar}. The proof is now exactly the same as in~\cite{BL19}: we use the fact that $\P \left( m \subset M \right)$ can be expressed using the number of ways to fill the holes of $m$ with maps of multipolygons, which is given by the left-hand side of Lemma~\ref{lem_calcul_planar}.
\end{proof}

We now move on to one-endedness. The proof is quite similar, and relies on the following estimate.

\begin{lem}\label{lem_calcul_OE}
We fix $(\mathbf{f}^{n})_{n \geq 1}$ and $(g_n)_{n \geq 1}$ which satisfy the assumptions of Theorem~\ref{thm_main_more_general}. We also fix $\mathbf{h}^{(0)}$ a face degree sequence such that $\mathbf{h}^{(0)} \leq\mathbf{f}^{n}$ for $n$ large enough.
\begin{itemize}
\item
Let $k\geq 1$, numbers $\ell_1, \ell_2,\ldots,\ell_k$ \textbf{not all equal to $1$} and perimeters $p_i^j$ for $1\leq j\leq k$ and $1\leq i\leq \ell_j$. Then
\[\sum_{\substack{\mathbf{h}^{(1)}+\mathbf{h}^{(2)}+\ldots+\mathbf{h}^{(k)}=\ff^n-\mathbf{h}^{(0)} \\ g^{(1)}+g^{(2)}+\ldots+g^{(k)}=g_n-\sum_j(\ell_j-1)}}\prod_{j=1}^k \beta_{g^{(j)}}^{(p^j_1,p^j_2,\ldots,p^j_{\ell_j})}(\mathbf{h}^{(j)})=o\left(\beta_{g_n}(\mathbf{f}^n)\right)\]
as $n\rightarrow \infty$.
\item 
Let $k\geq 1$ and perimeters $p_1,\ldots,p_k$. There is a constant $C$ (that may depend on everything above) such that for every $a$ and $n$ large enough we have
\[\sum_{\substack{\mathbf{h}^{(1)}+\mathbf{h}^{(2)}+\ldots+\mathbf{h}^{(k)}=\ff^n-\mathbf{h}^{(0)} \\ g^{(1)}+g^{(2)}+\ldots+g^{(k)}=g_n\\  |\mathbf{h}^{(1)}|, |\mathbf{h}^{(2)}|>a}} \, \prod_{j=1}^k \beta_{g_j}^{(p_j)}(\mathbf{h}^{(j)})\leq \frac{C}{a} \beta_{g_n}(\mathbf{f}^n).\]
\end{itemize}
\end{lem}

\begin{corr}\label{cor_OE}
Let $(\mathbf{f}^{n})_{n \geq 1}$ and $(g_n)_{n \geq 1}$ satisfy the assumptions of Theorem~\ref{thm_main_more_general}. Every subsequential limit $M$ of $\left( M_{\ff^n}, g_n \right)$ for $d^*_{\loc}$ is a.s. one-ended in the sense that, for every finite map $m$ with holes such that $m \subset M$, only one hole of $m$ is filled with infinitely many faces\footnote{This is a "weak" definition of one-endedness. For example, it does not prevent $m$ to be the dual of a tree. However, once we will have proved that $M$ has finite vertex degrees, this will be equivalent to the usual definition.}.
\end{corr}

\begin{proof}
The proof given Lemma~\ref{lem_calcul_OE} is exactly the same as~\cite[Corollary 9]{BL19}, except for the additional assumption in Lemma~\ref{lem_calcul_OE} that $\mathbf{h}^{(0)} \leq\mathbf{f}^{n}$ for $n$ large enough. We take care of it in the same way as in the proof of Corollary~\ref{cor_planar}.
\end{proof}

\subsection{Finiteness of the root degree}
\label{subsec_finite_degrees}

We now finish the proof of tightness for $d_{\loc}$ (Proposition~\ref{prop_tightness_dloc_univ}). Let $M$ be a subsequential limit of $\left( M_{\ff^n, g_n} \right)$ for $d_{\loc}^*$. By Lemma~\ref{lem_dual_convergence_univ}, to get tightness for $d_{\loc}$, we need to show that almost surely, all the vertices of $M$ have finite degree. Our argument is now very similar to \cite{BL19} and inspired by \cite{AS03}: we will first study the degree of the root vertex by using the Bounded ratio Lemma, and then extend finiteness by using invariance under the simple random walk.

\begin{lem}\label{lem_root_degree_is_finite_univ}
The root vertex of $M$ has a.s. finite degree.
\end{lem}

\begin{proof}
Following the approach of \cite{AS03}, we perform a filled-in lazy peeling exploration of $M$. Note that we already know by Corollary~\ref{cor_planar} that the explored part will always be planar, so no peeling step will merge two different existing holes. Moreover, by Corollary~\ref{cor_OE}, if a peeling step separates the boundary into two holes, then one of them is finite and will be filled with a finite map. Therefore, at each step, the explored part will have only one hole.

The peeling algorithm $\mathcal{A}$ that we use is the following: if the root vertex $\rho$ belongs to $\partial m$, then $\mathcal{A}(m)$ is the edge on $\partial m$ on the left of $\rho$. If $\rho \notin \partial m$, then the exploration is stopped. Let $\tau$ be the time at which the exploration is stopped. Since only finitely many edges incident to $\rho$ are added at each step, it is enough to prove $\tau<+\infty$ a.s.. We recall that $\expl_t^{\mathcal{A}}(M)$ is the explored part at time $t$.

We will prove that at each step, conditionally on $\expl_t^{\mathcal{A}}(M)$, the probability to swallow the root and finish the exploration in a bounded amount of time is bounded from below by a positive constant. We fix $j^* \geq 2$ with $\alpha_{j^*}>0$. Note that such a $j^*$ exists because of the assumption $\theta<\frac{1}{2} \sum_{j \geq 1} (j-1) \alpha_j$. For every map $m$ with one hole such that $\rho \in \partial m$, we denote by $m^+$ the map constructed from $m$ as follows (see Figure~\ref{fig_swallowing_root_univ}):
\begin{itemize}
\item
we first glue a "face" of degree $2j^*$ to $m$ along the edge of $\partial m$ on the left of $\rho$;
\item
we then glue together the two edges of the boundary incident to $\rho$ together;
\item
during the next $j^*-2$ steps, at each step, we pick two consecutive edges of the boundary according to some fixed convention and glue them together.
\end{itemize}

\begin{figure}
\centering
\includegraphics[scale=0.7]{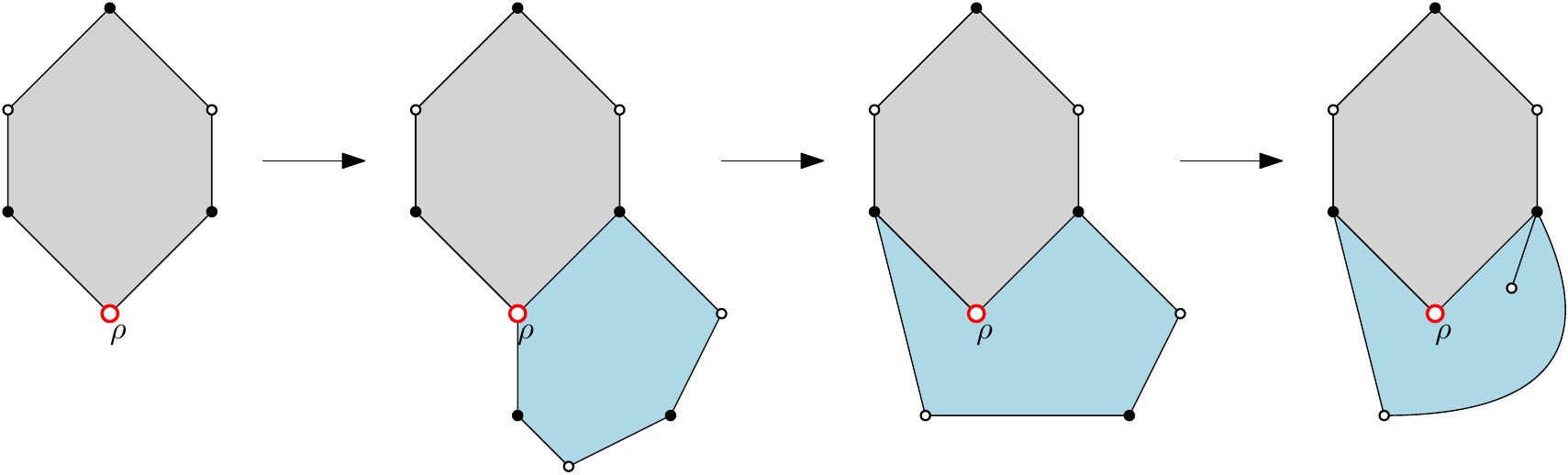}
\caption{The construction of $m^+$ from $m$. In gray, the map $m$. In red, the root vertex. In blue, the new face. Here, $|\partial m|=j^*=3$.}
\label{fig_swallowing_root_univ}
\end{figure}

Note that $m^+$ is a planar map with the same perimeter as $m$ but one more face (of degree $2j^*$). By the choice of our peeling algorithm, if we have $\tau \geq t$ and $\expl_t^{\mathcal{A}}(M)^+ \subset M$, then we have $\tau \leq t+2$. Hence it is enough to prove that the quantity
\[ \P \left( m^+ \subset M | m \subset M \right)\]
is bounded from below over finite, planar maps $m$ with one hole such that $\rho \in \partial m$.

We fix such an $m$, with half-perimeter $p$ and internal face degrees given by $\mathbf{h}$. Along some subsequence, we have $M_{\ff^{n},g_{n}} \to M$ in distribution (for $d_{\loc}^*$). Along the same subsequence, it holds that
\[ \P \left( m^+ \subset M | m \subset M \right)\hspace{-0.1cm}=\hspace{-0.1cm}\lim_{n \to +\infty}\hspace{-0.1cm} \frac{\P \left( m^+ \in M_{\ff^{n},g_{n}} \right)}{\P \left( m\in M_{\ff^{n},g_{n}} \right)} \hspace{-0.1cm} = \hspace{-0.1cm}\lim_{n \to +\infty}\hspace{-0.1cm}  \frac{\beta^{(p)}_{g_n} \left( \mathbf{f}^n-\mathbf{h}-\mathbf{1}_{j^*} \right)}{\beta^{(p)}_{g_n} \left( \mathbf{f}^n-\mathbf{h} \right)}. \]
By our choice of $j^*$, we have $f^n_{j^*} \geq \frac{\alpha_{j^*}}{2} |\ff^n|$ for $n$ large enough, so we can apply the Bounded ratio Lemma, which concludes the proof.
\end{proof}

\begin{proof}[Proof of Proposition~\ref{prop_tightness_dloc_univ}]
Let $M$ be a subsequential limit of $\left( M_{\ff^{n},g_{n}} \right)$. We recall that for all $n$, the map $\left( M_{\ff^{n},g_{n}} \right)$ is stationary for the simple random walk on its vertices. Therefore, by Lemma~\ref{lem_root_degree_is_finite_univ} and the same argument as in \cite{AS03} (see also the proof of Lemma~\ref{lem_easy_dual_convergence} above), almost surely all the vertices of $M$ have finite degree. By Lemma~\ref{lem_dual_convergence_univ}, this guarantees that $\left( M_{\ff^{n},g_{n}} \right)$ is tight for $d_{\loc}$.

The a.s. planarity of $M$ is proved in~Corollary \ref{cor_planar}. Finally, it easy to check that for maps with finite vertex degrees, the weak version of one-endedness proved in Corollary~\ref{cor_OE} implies the usual one. Indeed, if $V$ is a finite set of vertices of $M$, one can consider a finite, connected submap of $M$ containing all the faces and edges incident to vertices of $V$. Then Corollary~\ref{cor_planar} ensures that this submap does not separate $M$ into two infinite maps.
\end{proof}

\section{Weakly Markovian bipartite maps}\label{sec_univ_markov}

Our goal in this Section is to prove Theorem~\ref{thm_weak_Markov_general}.

\paragraph{Weakly Markovian bipartite maps.}
For a finite, bipartite map $m$ with one hole, we denote by $|\partial m|$ the half-perimeter of the hole of $m$. For all $j \geq 1$, we also denote by $v_j(m)$ the number of internal faces of $m$ with degree $2j$.

\begin{defn}\label{defn_weak_Markov}
Let $M$ be a random infinite, one-ended, bipartite planar map. We say that $M$ is \emph{weakly Markovian} if for every finite map $m$ with one hole, the probability $\P \left( m \subset M \right)$ only depends on $|\partial m|$ and $\left( v_j(m) \right)_{j \geq 1}$.
\end{defn}

Let $\VV$ be the set of sequences $\vv=(v_j)_{j \geq 1}$ such that $v_j=0$ for $j$ large enough. If $M$ is weakly Markovian and $\vv \in \VV$, we will denote by $a^p_{\vv}$ the probability $\P \left( m \subset M \right)$ for a map $m$ with $|\partial m|=p$ and $v_j(m)=v_j$ for all $j$. Note that this only makes sense if there is such a map $m$, which is equivalent to
\begin{equation}\label{eqn_good_pvv}
p \leq 1+\sum_{j \geq 1} (j-1)v_j.
\end{equation}
Therefore, if $p \geq 1$, we will denote by $\VV_p \subset \VV$ the set of those $\vv$ satisfying \eqref{eqn_good_pvv}. Note that $\VV_1=\VV$.
In particular, by definition, for $\q \in \mathcal{Q}_h$, the $\q$-IBPM is weakly Markovian, and the corresponding constants $a^p_{\vv}$ are:
\[a^p_{\vv}(\q):= C_p(\q) \q^{\vv},\]
where $\q^{\vv} := \prod_{j \geq 1} q_j^{v_j}$. Therefore, if $M$ is of the form $\MM_{\QQ}$ for some random weight sequence $\QQ$, we have $a^p_{\vv} = \E[ C_p(\QQ) \QQ^{\vv} ]$.

\paragraph{Sketch of the proof of Theorem~\ref{thm_weak_Markov_general}.}
We first note that the second point of Theorem~\ref{thm_weak_Markov_general} is immediate once the first point is proved. Indeed, let us write $\mathrm{Rootface}(m)$ for the degree of the root face of a map $m$. If $\mathrm{Rootface}(M)$ has finite expectation, then
\[ \E \left[ \E \left[ \mathrm{Rootface}(\MM_{\QQ}) | \QQ \right] \right] = \E \left[ \mathrm{Rootface}(\MM_{\QQ}) \right] <+\infty, \]
so $\E \left[ \mathrm{Rootface}(\MM_{\QQ}) | \QQ \right]<+\infty$ a.s., so $\QQ \in \cQ_f$ a.s..

The first point of Theorem~\ref{thm_weak_Markov_general} is the natural analogue of Theorem~2 of \cite{BL19}, where triangulations are replaced by more general maps. The proof will rely on similar ideas: we fix a weakly Markovian map $M$ with associated constants $a^p_{\vv}$, and we would like to find a random $\QQ$ such that $a^p_{\vv} = \E[ C_p(\QQ) \QQ^{\vv} ]$ for all $p$ and $\vv$. We will use peeling equations to establish inequalities between the $a^p_{\vv}$, and the existence of $\QQ$ will follow from the Hausdorff moment problem. However, compared to~\cite{BL19}, two new difficulties arise:
\begin{itemize}
\item[$\bullet$]
the random weights $\q$ form a family of real numbers instead of just one real number;
\item[$\bullet$]
in the triangular case, with the notations of Definition~\ref{defn_weak_Markov}, it was immediate that all the numbers $a^p_{\vv}$ are determined by the numbers $a^1_{\vv}$. This is not true anymore.
\end{itemize}
The first issue can be handled by using the multi-dimensional version of the Hausdorff moment problem. The second one, on the other hand, will make the proof a bit longer than in \cite{BL19}. More precisely, the Hausdorff moment problem will now provide us, for every $p \geq 1$, a $\sigma$-finite measure $\mu_p$ on the set of weight sequences, which describes $a^p_{\vv}$ for all $\vv$. We will then use the peeling equations to prove that all the $\mu_p$ are actually determined by $\mu_1$.

Because of the condition \eqref{eqn_good_pvv}, we will need to find a measure with suitable $\vv$-th moments for all $\vv \in \VV_p$, which is slightly different than the usual Hausdorff moment problem where $\vv \in \VV$. Therefore, we first need to state a suitable version of the moment problem, which will follow from the usual one. This is done in the next subsection.

\subsection{The incomplete Hausdorff moment problem}

To state our version of the moment problem (Proposition~\ref{prop_moment_Hausdorff} below), we will need to consider the space of sequences $\left( u_{\vv} \right)_{\vv \in \VV_p}$. For $j \geq 1$, we denote by $\Delta_j$ the discrete derivation operator on the $j$-th coordinate on this space. That is, if $u=(u_{\vv})$, we write
\[ \left( \Delta_j u \right)_{\vv} = u_{\vv}-u_{\vv+\mathbf{1}_j}.\]
It is easy to check that the operators $\Delta_j$ commute with each other. For all $\kk=(k_j)_{j \geq 1}$ such that $k_j=0$ for $j$ large enough (say for $j \geq j_0$), we define the operator $\Delta^{\kk}$ by
\[
\Delta^{\kk} u=  \Delta_1^{k_1} \Delta_2^{k_2} \dots \Delta_{j_0}^{k_{j_0}} u.
\]
In other words, we have
\[\Delta^{\kk} u= \sum_{\mathbf{i}} \left( \prod_{j \geq 1} (-1)^{i_j} \binom{k_j}{i_j} \right) u_{\vv+\mathbf{i}}, \]
where the sum is over families $\mathbf{i}=(i_j)_{j \geq 1}$, and the terms with a nonzero contribution are those for which $0 \leq i_j \leq k_j$ for every $j \geq 1$. The "usual" Hausdorff moment problem is then the following.

\begin{thm}\label{thm_hausdorff_usual}
Let $(u_{\vv})_{\vv \in \VV}$ be such that, for any $\vv \in \VV$ and any $\kk \geq \mathbf{0}$, we have
\[ \Delta^{\kk} u_{\vv} \geq 0.\]
Then there is a unique measure $\mu$ on $\mathcal{Q}=[0,1]^{\N^*}$ (equipped with the product $\sigma$-algebra) such that, for all $\vv \in \VV$, we have
\[ u_{\vv}=\int \q^{\vv} \mu(\dq). \]
In particular $\mu$ is finite, with total mass $u_{\mathbf{0}}$.
\end{thm}
More precisely, this is the infinite-dimensional Hausdorff moment problem, which can be deduced immediately from the finite-dimensional one by the Kolmogorov extension theorem.

For $p \geq 1$, we recall that $\VV_p \subset \VV$ is the set of $\vv \in \VV$ that satisfy $\sum_{j \geq 1} (j-1) v_j \geq p-1$. We also denote by $\VV_p^*$ the set of $\vv \in \VV_p$ for which there is $j \geq 2$ such that $v_j>0$ and $\vv-\mathbf{1}_j \in \VV_p$. In other words $\VV_p^*$ can be thought of as the "interior" of $\VV_p$. Finally, we recall that
\[ \cQ^*=\{ \q \in [0,1]^{\N^*} | \exists j \geq 2, q_j>0\}.\]

\begin{prop}\label{prop_moment_Hausdorff}
Fix $p \geq 1$, and let $\left( u_{\vv} \right)_{\vv \in \VV_p}$ be a family of real numbers. We assume that for all $\vv \in \VV_p$ and all $\kk \geq \mathbf{0}$, we have
\[ \Delta^{\kk} u_{\vv} \geq 0.\]
Then there is a $\sigma$-finite measure $\mu$ on $\cQ^*$ such that, for all $\vv \in \VV_p^*$, we have
\[ u_{\vv}=\int \q^{\vv} \mu(\dq). \]
Moreover, if $p=1$, then $\mu$ is finite and $\mu(\cQ^*) \leq u_\mathbf{0}$.
\end{prop}

Note that this version is "weaker" than Theorem~\ref{thm_hausdorff_usual} in the sense that it is not always possible to have $u_{\vv}=\int \q^{\vv} \mu(\dq)$ for $\vv \in \VV_p \backslash \VV^*_p$. A simple example of this phenomenon in dimension one is that the sequence $\left( \mathbbm{1}_{i=1} \right)_{i \geq 1}$ has all its discrete derivatives nonnegative. However, there is no measure on $[0,1]$ with first moment $1$ and all higher moments $0$. On the other hand, we can assume an additional property of our measure $\mu$, namely it is supported by $\cQ^*$ instead of $\cQ$ in Theorem~\ref{thm_hausdorff_usual}.

\begin{proof}
We start with the case $p=1$. Then $\VV_1=\VV$, so by Theorem~\ref{thm_hausdorff_usual}, there is a measure $\widetilde{\mu}$ on $\cQ$ such that, for all $\vv \in \VV_1$, we have
\[u_{\vv} = \int_{\cQ} \q^{\vv} \widetilde{\mu}(\dq). \]
Let $\mu$ be the restriction of $\widetilde{\mu}$ to $\cQ^*$. If $\vv \in \VV_1^*$ and $\q \in \cQ \backslash \cQ^*$, then there is $j \geq 2$ such that $\vv_j >0$ but $q_j=0$, so $\q^{\vv}=0$. It follows that, for all $\vv \in \VV_1^*$, we have
\[ \int_{\cQ^*} \q^{\vv} \mu(\dq) = \int_{\cQ} \q^{\vv} \widetilde{\mu}(\dq) = u_{\vv}.\]
Moreover, the total mass of $\mu$ is not larger than the total mass of $\widetilde{\mu}$, so it is at most $u_{\mathbf{0}}$.

We now assume $p \geq 2$. Let $\vv \in \VV_p$. Then $\vv+\ww \in \VV_p$ for all $\ww \in \VV$, so the sequence $\left( u_{\vv+\ww} \right)_{\ww \in \VV}$ satisfies the assumptions of Theorem~\ref{thm_hausdorff_usual}. Therefore, there is a finite measure $\mu_{\vv}$ on $\cQ$ such that
\[ u_{\vv+\ww} = \int \q^{\ww} \mu_{\vv}(\dq)\]
for all $\ww \in \VV$. Now let $\vv, \vv' \in \VV_p$. For all $\ww$, we have
\[ \int \q^{\vv'} \q^{\ww} \mu_{\vv}(\dq)= u_{\vv+\vv'+\ww}=\int \q^{\vv} \q^{\ww} \mu_{\vv'}(\dq).\]
In other words, the measures $\q^{\vv'} \mu_{\vv}(\dq)$ and $\q^{\vv} \mu_{\vv'}(\dq)$ have the same moments, so by uniqueness in Theorem~\ref{thm_hausdorff_usual}
\begin{equation}\label{eqn_consistence_muv}
\q^{\vv'} \mu_{\vv}(\dq)=\q^{\vv} \mu_{\vv'}(\dq).
\end{equation}
In particular, for all $\vv \in \VV_p$, we can consider the $\sigma$-finite measure \[\widetilde{\mu}_{\vv}(\dq)= \frac{\mu_{\vv}(\dq)}{\q^{\vv}}\] defined on $\{ \q^{\vv}>0\}$. Then \eqref{eqn_consistence_muv} implies that, for any $\vv, \vv' \in \VV_p$, the measures $\widetilde{\mu}_{\vv}$ and $\widetilde{\mu}_{\vv'}$ coincide on $\{ \q^{\vv}>0 \} \cap \{ \q^{\vv'}>0\}$. Therefore, there is a measure $\mu$ on $\bigcup_{\vv \in \VV_p} \{\q^{\vv}>0\} = \cQ^*$ such that, for all $\vv \in \VV_p$, we have
\begin{equation}\label{eqn_mu_and_muv_weak}
\mu_{\vv}(\dq) = \q^{\vv} \mu(\dq) \quad \mbox{ on } \quad \{\q^{\vv}>0\}.
\end{equation}
Since $\mu$ is finite on $\{q_j>\eps\}$ for all $\eps>0$ and $j \geq 2$, the measure $\mu$ is $\sigma$-finite. We would now like to extend the equality~\eqref{eqn_mu_and_muv_weak} to all $\cQ^*$ under the condition $\vv \in \VV_p^*$.

For this, let $\vv \in \VV_p^*$, and let $j \geq 2$ be such that $v_j>0$ and $\vv-\mathbf{1}_j \in \VV_p$. We have $p \mathbf{1}_j \in \VV_p$, so we can apply \eqref{eqn_consistence_muv} to $\vv$ and $p \mathbf{1}_j$. We obtain, on $\{ q_j>0 \}$:
\[ \mu_{\vv}(\dq) = \q^{\vv} \frac{\mu_{p \mathbf{1}_j}(\dq)}{q_j^p} = \q^{\vv} \mu(\dq), \]
using also \eqref{eqn_mu_and_muv_weak} for $p \mathbf{1}_j$. In other words, \eqref{eqn_mu_and_muv_weak} holds on $\{ q_j>0 \}$.

On the other hand, for $\vv$ and $\vv-\mathbf{1}_p$, we can obtain a stronger version of \eqref{eqn_consistence_muv}. More precisely, for all $\ww$, we have
\[  \int q_j \q^{\ww} \mu_{\vv-\mathbf{1}_j}(\dq) = u_{\vv+\ww} = \int \q^{\ww} \mu_{\vv}(\dq),\]
so the measures $q_j \mu_{\vv-\mathbf{1}_j}(\dq)$ and $\mu_{\vv}(\dq)$ have the same moments, so they coincide. But the first one is $0$ on $\{q_j=0\}$, so it is also the case for the second. Therefore, \eqref{eqn_mu_and_muv_weak} holds on $\{ q_j=0 \}$, with both sides equal to $0$.

Therefore, we have proved that \eqref{eqn_mu_and_muv_weak} holds on $\cQ$. By integrating over $\cQ^*$ and using that the total mass of $\mu_{\vv}$ is $u_{\vv}$ and is supported by $\cQ^*$, we get the result.
\end{proof}

\subsection{Proof of Theorem~\ref{thm_weak_Markov_general}}

As in \cite{BL19}, we start by writing down the peeling equations, which are linear equations between the numbers $a^p_{\vv}$ together. For every $p \geq 1$ and $\vv \in \VV_p$, we have
\begin{equation}\label{eqn_peeling_equation_bipartite}
a^p_{\vv}= \sum_{j \geq 1} a^{p+j-1}_{\vv+\mathbf{1}_j}+2\sum_{i=1}^{p-1} \sum_{\ww \in \VV} \beta_0^{(i-1)}(\ww) a^{p-i}_{\vv+\ww},
\end{equation}
where we recall that $\beta_0^{(i-1)}(\ww)$ is the number of planar, bipartite maps of the $2(i-1)$-gon with exactly $w_j$ internal faces of degree $2j$ for all $j \geq 1$. These equations, together with the facts that $a^1_\mathbf{0}=1$ and $a^p_{\vv} \geq 0$, characterize the families $(a^p_{\vv})$ of numbers that may arise from a weakly Markovian map. In order to be able to use the Hausdorff moment problem, we now need to check that the discrete derivatives of $(a^p_{\vv})$ are nonnegative.

\begin{lem}\label{lem_abs_monotone}
Let $M$ be a weakly Markovian bipartite map, and let $\left( a^p_{\vv} \right)$ be the associated constants. For every $\kk \geq \mathbf{0}$, $p \geq 1$ and $\vv \in \VV_p$, we have
\[ \left( \Delta^{\kk} a^p \right)_{\vv} \geq 0.\]
\end{lem}

\begin{proof}
The proof is similar to the proof of Lemma 16 in \cite{BL19}, with the following modification: in \cite{BL19}, it was useful that in the same peeling equation, we had $a^p_v$ appearing on the left and $a^p_{v+1}$ on the right. However, in \eqref{eqn_peeling_equation_bipartite} $a^p_{\vv+\mathbf{1}_j}$ does not appear in the right-hand side (this is because we are using the lazy peeling process of \cite{Bud15} instead of the simple peeling of \cite{Ang03}). Therefore, instead of using directly the peeling equation, we will need to use the \emph{double peeling equation}, which corresponds to performing two peeling steps, instead of one in~\eqref{eqn_peeling_equation_bipartite}.

More precisely, the peeling equation \eqref{eqn_peeling_equation_bipartite} gives an expansion of $a^p_{\vv}$. The \emph{double peeling equation} is obtained from \eqref{eqn_peeling_equation_bipartite} by replacing all the terms in the right-hand side by their expansion given by \eqref{eqn_peeling_equation_bipartite}. Note that this indeed makes sense because if $\vv \in \VV_p$, then $\vv+\mathbf{1}_j \in \VV_{p+j-1}$ for all $j \geq 1$, and $\vv+\ww \in \VV_{p-i}$ for all $i \geq 1$ and $\ww \in \VV$.

The equation we obtain is of the form
\begin{equation}\label{eqn_double_peeling_equation}
a^p_{\vv}=\sum_{i \in \Z, \, \ww \in \VV} c^{p,i}_{\vv,\ww} a^{p+i}_{\vv+\ww},
\end{equation}
where the coefficients $c^{p,i}_{\vv,\ww}$ are nonnegative integers. An explicit formula for these could be computed in terms of the $\beta_0^{(i)}(\ww)$, but this will not be needed. Here are the facts that will be useful:
\begin{enumerate}
\item
the coefficients $c^{p,i}_{\vv,\ww}$ actually do not depend on $\vv$, so we can write them $c^{p,i}_{\ww}$,
\item\label{item_coeff_equal_one}
we have $c^{p,2j-2}_{2\cdot\mathbf{1}_j}=1$ for every $j \geq 1$,
\item\label{item_coeff_positive}
we have $c^{p,0}_{\mathbf{1}_j} \geq 1$ for every $j \geq 2$.
\end{enumerate}
The first item follows from the fact that at each time, the available next peeling steps do not depend on the internal face degrees of the explored region.
The second item expresses the fact that, for a given peeling algorithm, there is a unique way to obtain a map with half-perimeter $p+2j-2$ with internal faces $\vv+2\cdot\mathbf{1}_j$ in two peeling steps. This way is to discover a unique face of degree $2j$ at both steps. The third item means that it is possible (not necessarily in a unique way) to obtain in two peeling steps a map with the same perimeter but one more face of degree $2j$. This is achieved by discovering a new face of degree $2j$ at the first step, and gluing all but two sides of this face two by two at the second step (see Figure~\ref{figure_pconstant_vincrease}). 

\begin{figure}
\begin{center}
\includegraphics[scale=0.7]{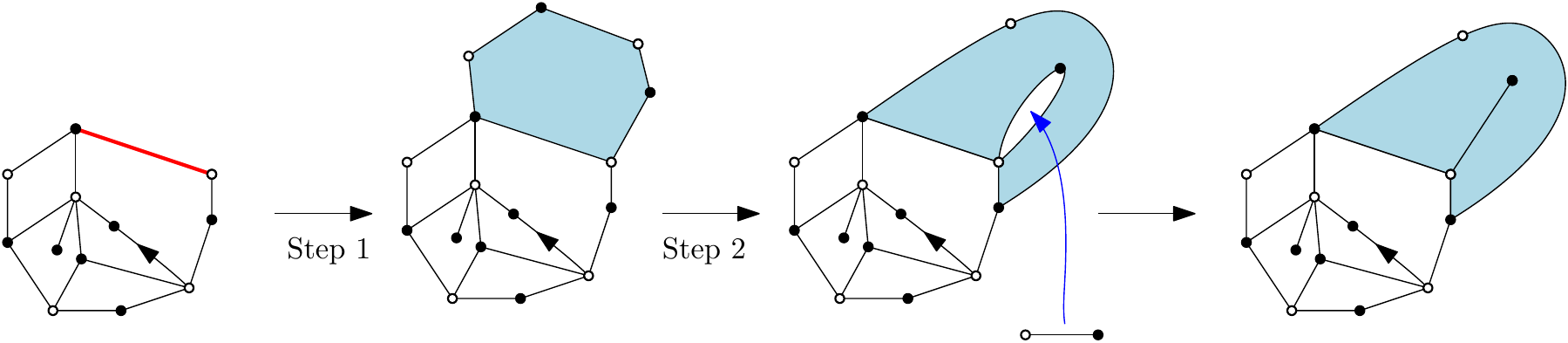}
\caption{In two peeling steps, the perimeter stays constant and one face with degree $2j$ is added (here $j=3$).}\label{figure_pconstant_vincrease}
\end{center}
\end{figure}

We now prove the lemma by induction on $|\kk|=\sum_{j \geq 1} k_j$. First, the case $|\kk|=0$ just means that $a^p_{\vv} \geq 0$ for all $p \geq 1$ and $\vv \in \VV_p$, which is immediate. Let us now assume that the lemma is true for $\kk$ and prove it for $\kk+\mathbf{1}_j$, where $j \geq 1$. We will first treat the case where $j \geq 2$.
Using the double peeling equation \eqref{eqn_double_peeling_equation} for $(p, \vv+\mathbf{i})$ for different values of $\mathbf{i}$, we have
\[ \left( \Delta^{\kk} a^p \right)_{\vv} = \sum_{i \in \Z, \, \ww \in \VV} c^{p,i}_{\ww} \left(  \Delta^{\kk} a^{p+i}\right)_{\vv+\ww}. \]
Therefore, using the induction hypothesis and Item~\ref{item_coeff_equal_one} above, we can write 
\begin{align*}
0 & \leq \left( \Delta^{\kk} a^{p+2j-2} \right)_{\vv+2\cdot\mathbf{1}_j}\\
&= c^{p,2j-2}_{2\cdot\mathbf{1}_j} \left( \Delta^{\kk} a^{p+2j-2} \right)_{\vv+2\cdot\mathbf{1}_j}\\
&=  \left( \Delta^{\kk} a^{p} \right)_{\vv} - \sum_{(i,\ww) \ne (2j-2, 2\cdot\mathbf{1}_j)} c^{p,i}_{\ww} \left( \Delta^{\kk} a^{p+i} \right)_{\vv+\ww}.
\end{align*}
Using the induction hypothesis again, we can remove all the terms in the last sum except the one where $(i,\ww)=(0,\mathbf{1}_j)$. Moreover, by Item~\ref{item_coeff_positive} above, we can replace the coefficient $c^{p,0}_{\mathbf{1}_j}$ by $1$. We obtain
\[0 \leq \left( \Delta^{\kk} a^{p} \right)_{\vv} - \left( \Delta^{\kk} a^{p} \right)_{\vv+\mathbf{1}_j}= \left( \Delta^{\kk+\mathbf{1}_j} a^p \right)_{\vv}, \]
which proves the induction step for $j \geq 2$. If $j=1$, Item~\ref{item_coeff_positive} is not true anymore (it is not possible to add only one face of degree $2$ in $2$ steps without changing the perimeter). Therefore, instead of \eqref{eqn_double_peeling_equation}, we use the simple peeling equation \eqref{eqn_peeling_equation_bipartite} like in \cite{BL19}. More precisely, in the induction step, we fix $j' \geq 2$ and write, using \eqref{eqn_peeling_equation_bipartite}:
\begin{align*}
0 &\leq \left( \Delta^{\kk} a^{p+j'-1} \right)_{\vv+\mathbf{1}_{j'}}\\
&= \left( \Delta^{\kk} a^{p} \right)_{\vv} - \sum_{j'' \ne j'} \left( \Delta^{\kk} a^{p+j''-1} \right)_{\vv+\mathbf{1}_{j''}} -2 \sum_{i=0}^{p-1} \sum_{\ww} \beta_0^{(i-1)}(\ww) \left( \Delta^{\kk} a^{p-i} \right)_{\vv+\ww}.
\end{align*}
Each term in the two sums is nonnegative by the induction hypothesis, so we can remove the second sum and keep only the term $j''=1$ in the first one to obtain
\[0  \leq \left( \Delta^{\kk} a^{p} \right)_{\vv} - \left( \Delta^{\kk} a^{p} \right)_{\vv+\mathbf{1}_1} = \left( \Delta^{\kk+\mathbf{1}_1} a^{p} \right)_{\vv}.\]
This concludes the proof of the lemma.
\end{proof}

By Lemma~\ref{lem_abs_monotone} and Proposition~\ref{prop_moment_Hausdorff}, for all $p \geq 1$, there is a $\sigma$-finite measure $\mu_p$ on $\cQ^*$ such that, for all $\vv \in \VV^*_p$,
\begin{equation}\label{eqn_apv_as_moment}
a^p_{\vv}=\int_{\cQ^*} \q^{\vv} \mu_p(\dq)
\end{equation}
and furthermore $\mu_1(\cQ^*) \leq a^1_{\mathbf{0}}=1$. We now replace $a_{\vv}^p$ by this expression in the peeling equation \eqref{eqn_peeling_equation_bipartite}. We get
\begin{align*}
\int \q^{\vv} \, \mu_p(\mathrm{d} \q) \hspace{-0.1cm}&= \hspace{-0.1cm}\sum_{j \geq 1} \int \q^{\vv+\mathbf{1}_j} \, \mu_{p+j-1}(\dq) + 2\sum_{i=1}^{p-1} \sum_{\ww \in \VV} \beta_0^{(i-1)}(\ww) \int \q^{\vv+\ww} \, \mu_{p-i}(\dq)\\
&=\hspace{-0.1cm} \int \q^{\vv} \left( \sum_{j \geq 1} q_j \, \mu_{p+j-1}(\dq) + 2 \sum_{i=1}^{p-1} W_{i-1}(\q) \, \mu_{p-i}(\dq) \right),
\end{align*}
where we recall that $W_{i-1}(\q)$ is the partition function of Boltzmann bipartite maps of the $2(i-1)$-gon with Boltzmann weights $\q$.
In particular, the right-hand side for $i=2$ must be finite, which means that $\mu_p$ is supported by the set $\cQ_a$ of admissible weight sequences. Moreover, the last display means that the two measures
\[ \mu_p(\dq) \mbox{ and } \nu_p(\dq) = \sum_{j \geq 1} q_j \, \mu_{p+j-1}(\dq) + 2 \sum_{i=1}^{p-1} W_{i-1}(\q) \, \mu_{p-i}(\dq) \]
have the same $\vv$-th moment for all $\vv \in \VV^*_p$. In particular, if we fix $j \geq 2$, this is true as soon as $v_j \geq p$, so the measures $q_j^p \mu_p(\dq)$ and $q_j^p \nu_p(\dq)$ have the same moments so they are equal, so $\mu_p$ and $\nu_p$ coincide on $\{q_j>0\}$. Since this is true for all $j \geq 2$ and $\mu_p, \nu_p$ are defined on $\cQ^*=\bigcup_{j \geq 2} \{q_j>0\}$, the measures $\mu_p$ and $\nu_p$ are the same, that is,
\begin{equation}\label{peeling_equation_mu}
\mu_p(\dq) = \sum_{j \geq 1} q_j \, \mu_{p+j-1}(\dq) + 2 \sum_{i=1}^{p-1} W_{i-1}(\q) \, \mu_{p-i}(\dq).
\end{equation}
We now note that this equation is very similar to the one satisfied by the constants $C_p(\q)$ used to define the $\q$-IBPM. More precisely, we fix a finite measure $\mu$ such that all the $\mu_p$ are absolutely continuous with respect to $\mu$ (take e.g. $\mu (\dq)=\sum_{p \geq 1} \frac{g_p(\q) \mu_p(\dq)}{2^p}$, where $g_p(\q)>0$ is such that the total mass of $g_p(\q) \mu_p(\dq)$ is at most $1$). We denote by $f_p(\q)$ the density of $\mu_p$ with respect to $\mu$. Then \eqref{peeling_equation_mu} becomes
\[ f_p(\q) = \sum_{j \geq 1} q_j f_{p+j-1}(\q) + 2 \sum_{i=1}^{p-1} W_{i-1}(\q) f_{p-i}(\q) \]
for $\mu$-almost every $\q \in \cQ^*$. In other words, $\left( f_p(\q) \right)_{p \geq 1}$ satisfies the exact same equation as $\left( C_p(\q) \right)_{p \geq 1}$ in \cite[Appendix C]{B18these}. These equations have a nonzero solution if and only if $\q \in \cQ_h$, so the measures $\mu_p$ are actually supported by $\cQ_h$. Moreover, by uniqueness of the solution (up to a multiplicative constant), we have
\[ f_p(\q)=\frac{C_p(\q)}{C_1(\q)} f_1(\q) = C_p(\q) f_1(\q)\]
for $\mu$-almost every $\q$, so $\mu_p(\dq)=C_p(\q) \mu_1(\dq)$. Now let $\alpha \leq 1$ be the total mass of the measure $\mu_1$, and let $\QQ$ be a random variable with distribution $\alpha^{-1} \mu$. We then have, for all $p \geq 1$ and $\vv \in \VV^*_p$, if $m$ is a map with half-perimeter $p$ and face degrees $\vv$:
\begin{equation}\label{eqn_apv_expectation}
\P \left( m \subset M \right) = a^p_{\vv} = \int \q^{\vv} \mu_p(\dq) = \alpha \E \left[ C_p(\QQ) \QQ^{\vv} \right] = \alpha \P \left( m \subset \MM_{\QQ} \right).
\end{equation}
Note that $\QQ$ is not well-defined if $\alpha=0$, but in this case $\mu_p=0$ for all $p$ so \eqref{eqn_apv_expectation} remains true for any choice of $\QQ$. To conclude that $M$ has the law of $\MM_{\QQ}$, all we have left to prove is that $\alpha=1$ and that \eqref{eqn_apv_expectation} can be extended to any $\vv \in \VV_p$. For this, we will show that, when we explore $M$ via a peeling exploration, the perimeter and volumes of the explored region at time $t$ satisfy $\vv \in \VV_p^*$ for $t$ large enough.

More precisely, if $\mathcal{A}$ is a peeling algorithm, we recall that $\expl^{\mathcal{A}}_t(M)$ is the explored part of $M$ after $t$ steps of a filled-in peeling exploration according to $\mathcal{A}$. We denote by $P_t$ the half-perimeter of the hole of $\expl^{\mathcal{A}}_t(M)$ and by $\mathbf{V}_t$ the sequence of degrees of its internal faces (that is, $V_{t,j}$ is the number of internal faces of $\expl^{\mathcal{A}}_t(M)$ with degree $2j$). Since $M$ is weakly Markovian, the process $\left( P_t, \mathbf{V}_t \right)_{t \geq 0}$ is a Markov chain whose law does not depend on the peeling algorithm $\mathcal{A}$.

\begin{lem}\label{lem_volumes_star}
We have
\[ \P \left( \mathbf{V}_{t} \in \VV^*_{P_t} \right) \xrightarrow[t \to +\infty]{} 1.\]
\end{lem}

\begin{proof}
Since the probability in the lemma does not depend on $\mathcal{A}$, it is sufficient to prove the result for a particular peeling algorithm. Therefore, we can assume that $\mathcal{A}$ has the following property: if the root face of $m$ and its hole have a common vertex $m$, then the peeled edge $\mathcal{A}(m)$ is incident to such a vertex. We will prove that for this algorithm, we have a.s. $\mathbf{V}_{t} \in \VV^*_{P_t}$ for $t$ large enough.

More precisely, since the vertex degrees of $M$ are a.s. finite and by definition of $\mathcal{A}$, all the vertices incident to the root face will eventually disappear from the boundary of the explored part. Therefore, for $t$ large enough, no vertex incident to the root face is on $\partial \expl^{\mathcal{A}}_t(M)$. We now fix $t$ with this property. If we denote by $\mathrm{Inn}(m)$ the number of internal vertices of a map $m$ with a hole and by $2J$ the degree of the root face of $M$, this implies $\mathrm{Inn} \left( \expl^{\mathcal{A}}_t(M) \right) \geq 2J$ for $t$ large enough.

On the other hand, the total number of edges of $\expl^{\mathcal{A}}_t(M)$ is $p+\sum_{j \geq 1} j V_{t,j}$, so by the Euler formula
\begin{align*}
\mathrm{Inn} \left( \expl^{\mathcal{A}}_t(M) \right) &= 2 + \left( P_t+ \sum_{j \geq 1} j V_{t,j} \right) - \left( 1+\sum_{j \geq 1} V_{t,j} \right) -2P_t \\&= 1-P_t+\sum_{j \geq 1} (j-1) V_{t,j}. 
\end{align*}
Taking $t$ large enough to have $\mathrm{Inn} \left( \expl^{\mathcal{A}}_t(M) \right) \geq 2J$, we obtain
\[ \left( \sum_{j \geq 1} (j-1) V_{t,j} \right) -(J-1) \geq \left( 2J +P_t -1 \right) -(J-1) = P_t+J > P_t-1,\]
so $V_{t,J}>0$ and $\mathbf{V}-\mathbf{1}_{J} \in \VV_{P_t}$. This proves $\mathbf{V}_{t} \in \VV^*_{P_t}$ for $t$ large enough.
\end{proof}

We now conclude the proof of Theorem~\ref{thm_weak_Markov_general} from \eqref{eqn_apv_expectation}. We consider a finite map $m_0$ with a hole and a peeling algorithm $\mathcal{A}$ that is consistent with $m_0$ in the sense that $m_0$ is a possible value of $\expl_{t_0}^{\mathcal{A}}$ for some $t_0 \geq 0$. We note that $\expl^{\mathcal{A}}_{t_0}(M)=m_0$ if and only if $m_0 \subset M$. Indeed, the direct implication is immediate. The indirect one comes from the fact that, if $m_0 \subset M$, then all the peeling steps until time $t_0$ must be consistent with $m_0$, so $m_0 \subset M$ determines the first $t_0$ peeling steps. We now take $t \geq t_0$. We sum \eqref{eqn_apv_expectation} over all possible values $m$ of $\expl^{\mathcal{A}}_t(M)$ such that $m_0 \subset m$ and the half-perimeter $p$ and internal face degrees $\vv$ of $m$ satisfy $\vv \in \VV_p^*$. We get
\[ \P \left( m_0 \subset M \mbox{ and } \mathbf{V}_t \in \VV^*_{P_t} \right) = \alpha \P \left( m_0 \subset \MM_{\QQ} \mbox{ and } \mathbf{V}^{\QQ}_t \in \VV^*_{P_t^{\QQ}} \right), \]
where $P_t^{\QQ}$ and $\mathbf{V}_t^{\QQ}$ are the analogues of $P_t$ and $\mathbf{V}_t$ for $\MM_{\QQ}$ instead of $M$. Since $\MM_{\QQ}$ is weakly Markovian, we can apply Lemma~\ref{lem_volumes_star} to both $M$ and $\MM_{\QQ}$. Therefore, letting $t \to +\infty$ in the last display, we get \[\P \left( m_0 \subset M  \right) = \alpha \P \left( m_0 \subset \MM_{\QQ} \right)\] for all $m_0$. In particular, if $m_0$ is the trivial map consisting only of the root edge, we get $\alpha=1$, so $M$ and $\MM_{\QQ}$ have the same law. This proves Theorem~\ref{thm_weak_Markov_general}.

\begin{proof}[Proof of Theorem~\ref{thm_main_more_general}]
By Proposition~\ref{prop_tightness_dloc_univ}, any subsequential limit $M$ of $\left( M_{\ff^n, g_n} \right)$ is planar and one-ended. Moreover, let $m$ be a map with one hole of half-perimeter $p$ and $v_{j}$ faces of degree $2j$ for all $j \geq 1$. Then
\[\P \left( m \subset M \right) = \lim_{n \to +\infty} \P \left( m \subset M_{\ff^n, g_n} \right) = \lim_{n \to +\infty} \frac{\beta_{g_n}^{(p)}(\ff^n-\vv)}{\beta_{g_n}(\ff^n)},\]
where the limits are along some subsequence. In particular, the dependence in $m$ is only in $p$ and $\vv$, so $M$ is weakly Markovian and the result follows by Theorem~\ref{thm_weak_Markov_general}.
\end{proof}

\section{The parameters are deterministic}\label{sec_univ_end}

\subsection{Outline}\label{sec_arg_deux_trous}

Our goal is now to prove Theorem~\ref{univ_main_thm}. We fix face degree sequences $\ff^n$ and genuses $g_n$ for $n \geq 0$ satisfying the assumptions of Theorem~\ref{univ_main_thm} (in particular, we now assume $\sum_{j} j^2 \alpha_j <+\infty$ until the end of the paper). By Theorem~\ref{thm_main_more_general}, up to extracting a subsequence, we can assume $M_{\ff^n, g_n}$ converges to $\MM_{\QQ}$, where $\QQ$ is a random variable with values in $\cQ_h$. Moreover, the law of the degree of the root face in $M_{\ff^n, g_n}$ converges in distribution to $\left( j \alpha_j \right)_{j \geq 1}$, which has finite expectation. By the last point of Theorem~\ref{thm_weak_Markov_general}, we have $\QQ \in \cQ_f$ a.s.. To prove Theorem~\ref{univ_main_thm}, it is enough to prove that $\QQ$ is deterministic, and only depends on $\left( \alpha_j \right)_{j \geq 1}$ and $\theta$.

\paragraph{Sketch of the end of the proof.}
Since we will follow similar ideas, let us first recall the strategy of \cite{BL19}. If $e_n$ is the root edge of $M_{\ff^n, g_n}$, the parameters $\QQ$ can be observed on a large neighbourhood of $e_n$ in $M_{\ff^n, g_n}$ for $n$ very large.
The first step of the proof (Proposition~\ref{prop_two_holes_argument}) roughly consists of showing that once $M_{\ff^n, g_n}$ is picked, the weights $\QQ$ do not depend on the choice of $e_n$. This is proved by the \emph{two holes argument}: if $e^1_n$ and $e_n^2$ are two roots chosen uniformly on $M_{\ff^n, g_n}$, we swap two large neighbourhoods of $e^1_n$ and $e^2_n$ in $M_{\ff^n, g_n}$. We then remark that if the weights observed around the two roots are too different, then the map obtained after swapping does not look like a map of the form $\MM_{\QQ}$.
The second step consists of noticing that the average value over all choices of the root of some functions of $\QQ$ is fixed by $\left( \alpha_j \right)_{j \geq 1}$ and $\theta$ (Corollary~\ref{corr_main_minus_monotonicity}).
Finally, in the third step we prove that these functions are sufficient to characterize $\QQ$ completely (Proposition~\ref{prop_monotonicity_deg}).

However, two important difficulties appear here compared to~\cite{BL19}:
\begin{itemize}
\item
in the first step, we need to find two large pieces around $e_n^1$ and $e_2^n$ with the exact same perimeter, in order to be able to swap them. This was easy for triangulations since the perimeter process associated to a peeling exploration takes all values. This is not true anymore in our general setting. This part will make crucial use of the assumption $\sum_j j^2 \alpha_j < +\infty$ (this is actually the only place in the paper where we will use it). A consequence of this difficulty is that instead of performing the swapping operation \emph{with high probability}, we will perform it \emph{with positive probability}.
\item
In the third step, one of the parameters that we control is the average vertex degree. For triangulations, it followed from an explicit formula that the average degree characterizes $\QQ$. We do not have such a formula here, so our argument will be more involved, and rely on the partial results obtained so far in the present paper.
\end{itemize}

\paragraph{Intermediate results.}
Let $\left( M_n, e_n^1, e_n^2 \right)$ be a uniform, bi-rooted map with face degrees $\ff^n$ and genus $g_n$ (i.e. $e_n^1$ and $e_n^2$ are picked uniformly and independently among the edges of $M_n$). We highlight that we will write $M_n$ instead of $M_{\ff^n, g_n}$ in this section to make notations lighter.
Up to extracting a subsequence, we can assume the joint convergence
\[ \left( (\M_n, e_n^1), (\M_n, e_n^2) \right) \xrightarrow[n \to +\infty]{(d)} \left( \MM^1_{\QQ^1}, \MM^2_{\QQ^2} \right) \]
for the local topology, where $\QQ^1$ and $\QQ^2$ have the same distribution as $\QQ$. Moreover, by the Skorokhod representation theorem, we can assume this joint convergence is almost sure. \emph{We will stay in this setting in Section~\ref{subsec_same_perimeter} and~\ref{subsec_two_holes}.} The first step of the proof will consist of proving the following.

\begin{prop}\label{prop_two_holes_argument}
We have $\QQ^1=\QQ^2$ almost surely.
\end{prop}

From here, the second step will be to deduce the next result. For $\q \in \mathcal{Q}_h$, we recall that $j a_j(\q)$ is the probability that the root face of $\MM_{\q}$ has degree $2j$, and that $d(\q)=\E \left[ \frac{1}{\deg_{\MM_{\q}} (\rho)} \right]$. 

\begin{corr}\label{corr_main_minus_monotonicity}
Under the assumptions of Theorem~\ref{univ_main_thm}, let $\MM_{\mathbf{Q}}$ be a subsequential limit. Then almost surely, we have
\[ d(\mathbf{Q})=\frac{1}{2} \left( 1-2\theta-\sum_i \alpha_i \right) \mbox{ and, for all $j \geq 1$, } a_j(\mathbf{Q})=\alpha_j.\] 
\end{corr}

\paragraph{Structure of the section.}
In Section~\ref{subsec_same_perimeter}, we will address the issue of finding two large neighbourhoods of the two roots with exactly the same perimeters. In Section~\ref{subsec_two_holes}, we use this to prove Proposition~\ref{prop_two_holes_argument} and Corollary~\ref{corr_main_minus_monotonicity}. Finally, Section~\ref{subsec_last_step} is devoted to the end of the proof of Theorem~\ref{univ_main_thm}, and consists mostly of showing that $d(\q)$ and $\left( a_j(\q) \right)_{j \geq 1}$ are sufficient to characterize $\q$.

\subsection{Finding two pieces with the same perimeter}
\label{subsec_same_perimeter}

As explained above, given the uniform bi-rooted map $\left( M_n, e^1_n, e^2_n \right)$, we want to find two neighbourhoods of $e^1_n$ and $e^2_n$ with the same large perimeter $2p$. For this, we will perform a peeling exploration around the two roots and stop it when the perimeter of the explored region is exactly $2p$. However, since the perimeter process has a positive drift, it can make large positive jumps, we cannot guarantee that both perimeters will hit the value $p$ with high probability. We will therefore show a weaker result: roughly speaking, the probability that the perimeter processes around $e_n^1$ and $e_n^2$ both hit $p$ is bounded from below, even if we condition on $\QQ^1$ and $\QQ^2$. 

More precisely, we fix a deterministic peeling algorithm $\mathcal{A}$, and let $p,v_0 \geq 1$. We recall from the end of Section~\ref{subsec_lazy_peeling} that we can make sense of a filled-in peeling exploration on the finite map $M_n$ around $e^1_n$ or $e_n^2$. We perform the following exploration:
\begin{itemize}
\item
we explore the map $M_n$ around $e^1_n$ according to the algorithm $\mathcal{A}$ until the number of edges in the explored region is larger than $v_0$, or the perimeter of the explored region is exactly $2p$, and denote by $\tau^1_n$ the time at which we stop;
\item
we do the same thing around $e^2_n$ and denote by $\tau^2_n$ the stopping time.
\end{itemize}
We write $\mathcal{S}_{n,p,v_0}$ for the event where both $\tau^1_n$ and $\tau^2_n$ occur because the perimeter hits $2p$, and where the two regions explored around $e_n^1$ and $e_n^2$ are face-disjoint (the dependence of $\mathcal{S}$ in $\mathcal{A}$ will stay implicit). We note right now that $(M_n, e_n^1)$ has a planar, one-ended local limit. Hence, with probability $1-o(1)$ as $n \to +\infty$, the exploration is not stopped before $\tau_n^1$ or $\tau_n^2$ for the reason stated in the end of Section~\ref{subsec_lazy_peeling}.

The goal of this subsection is to prove the next result. We recall that the functions $r_j(\q)$ for $j \in \N^* \cup \{ \infty \}$ and $\q \in \cQ_h$ are defined in Proposition~\ref{prop_q_as_limit}.

\begin{prop}\label{prop_finding_regions_to_cut}
Let $(M_n, e^1_n, e_2^n)$ and $\QQ^1, \QQ^2$ be as in Section~\ref{sec_arg_deux_trous}.
We fix $j \in \N^* \cup \{ \infty \}$, and $\eps>0$. Then there is $\delta>0$ with the following property. For every $p \geq 1$ large enough, there is $v_0$ such that, for $n$ large enough:
\[ \mbox{if } \P \left( | r_j(\QQ^1)-r_j(\QQ^2) |>\eps \right) \geq \eps,\]
\[ \mbox{then }\P \left( | r_j(\QQ^1)-r_j(\QQ^2)|>\frac{\eps}{2} \mbox{ and } (\M_n, e_n^1, e_n^2) \in \mathcal{S}_{n,p,v_0} \right) \geq \delta. \]
\end{prop}

We recall that we have used the Skorokhod theorem to couple the finite and infinite maps together, so the last event makes sense.

Here is why Proposition~\ref{prop_finding_regions_to_cut} seems reasonable: we know that conditionally on $(\QQ^1, \QQ^2)$, the perimeters of the explored region along a peeling exploration of $\MM^1_{\QQ^1}$ and $\MM^2_{\QQ^2}$ are random walks conditioned to stay positive. Moreover, since $\QQ^1, \QQ^2 \in \mathcal{Q}_f$, these random walks do not have a too heavy tail, so each of them have a reasonable chance of hitting exactly $p$. However, there is no reason a priori why $\MM^1_{\QQ^1}$ and $\MM^2_{\QQ^2}$ should be independent conditionally on $\QQ^1$ and $\QQ^2$, so it might be unlikely that \emph{both} processes hit $p$. Therefore, the sketch of the proof will be the following:
\begin{itemize}
\item
we fix a large constant $C>0$ ($C$ will be much smaller than $p$),
\item
we prove that both walks have a large probability to hit the interval $[p,p+C]$ before the explored volume exceeds $v_0(p)$ (Lemma~\ref{lem_RW_hits_interval}),
\item
once both perimeter processes around $e_n^1$ and $e_n^2$ in $M_n$ have hit $[p,p+C]$, we use the Bounded ratio Lemma (Lemma~\ref{lem_BRL}, item 2) to show that, with probability bounded from below by roughly $e^{-C}$, both perimeters fall to exactly $p$ in at most $C$ steps. This will prove the proposition with $\delta \approx \eps e^{-C}$ and $v_0=v_0(p)$.
\end{itemize}
The point of replacing $p$ by $[p,p+C]$ is to deal with events of large probability, so that we don't need any independence to make sure that two events simultaneously happen.

For this, consider the peeling exploration of $\MM^1_{\QQ^1}$ according to $\mathcal{A}$. We denote by $\sigma^{1,\infty}_{[p,p+C]}$ the first time at which the half-perimeter is in $[p,p+C]$ (this stopping time might be infinite). We define $\sigma^{1,n}_{[p,p+C]}$ (resp. $\sigma^{2, \infty}_{[p,p+C]}$, $\sigma^{2,n}_{[p,p+C]}$) as the analogue quantity for the exploration in $\M_n$ around $e_n^1$ (resp. in $\MM^2_{\QQ^2}$, in $\left( M_n, e_n^2 \right)$).

\begin{lem}\label{lem_RW_hits_interval}
We have
\[ \lim_{C \to +\infty} \liminf_{p \to +\infty} \P \left( \sigma_{[p,p+C]}^{1,\infty} <+\infty \right) =1.\]
\end{lem}

\begin{proof}
We know that $\QQ^1 \in \cQ_f$ a.s.. Hence, it is enough to prove that, for any $\q \in \cQ_f$, we have
\begin{equation}\label{eqn_RW_hits_interval_infinite}
\lim_{C \to +\infty} \liminf_{p \to +\infty} \P \left( \sigma^{1,\infty}_{[p,p+C]}<+\infty \big| \mathbf{\QQ}=\mathbf{q} \right) =1.
\end{equation}
The lemma then follows by taking the expectation and using Fatou's lemma. Conditionally on $\QQ^1=\mathbf{q}$, the law of $\MM^1_{\QQ^1}$ is the law of $\MM_\q$. In particular, the process $P$ describing the half-perimeter of the explored region has the law of $X$ conditioned to stay positive, where $X$ is a random walk with step distribution $\widetilde{\nu}_{\q}$.

To prove \eqref{eqn_RW_hits_interval_infinite}, we distinguish two cases: the case where $\mathbf{q}$ is critical, and the case where it is not. We start with the second one. Then by the results of Section~\ref{subsec_IBPM}, the walk $X$ satisfies $\E \left[ |X_1| \right]<+\infty$ and $\E \left[ X_1 \right]>0$, so the conditioning to stay positive is non degenerate. Therefore, it is enough to prove
\begin{equation}\label{eqn_RW_hits_interval_conditioned}
\lim_{C \to +\infty} \liminf_{p \to +\infty} \P \left( \mbox{$X$ hits $[p,p+C]$} \right) =1.
\end{equation}
This follows from standard renewal arguments: if we denote by $(H_i)_{i \geq 0}$ the ascending ladder heights of $P$, then $(H_i)$ is a renewal set with density $\E[H_1]= \frac{1}{\E[X_1]}>0$. Let $I_p$ be such that $H_{I_p} < p \leq H_{I_{p+1}}$. Then the law of $H_{I_{p+1}}-H_{I_p}$ converges as $p \to +\infty$ to the law of $H_1$ biased by its size, so
\begin{align*}
\P \left( \mbox{$P$ does not hit $[p,p+C]$} \right) &\leq \P \left( H_{I_{p+1}} \notin [p,p+C] \right)\\& \leq \P \left( H_{I_{p+1}}-H_{I_p} > C \right) \xrightarrow[p \to +\infty]{} \frac{\E \left[ H_1 \mathbbm{1}_{H_1>C} \right]}{\E[H_1]},
\end{align*} 
and this last quantity goes to $0$ as $C \to +\infty$.

We now tackle the case where $\mathbf{q}$ is critical, which by the results in the end of~\ref{subsec_IBPM} implies $\sum_{i \geq 1} i^{3/2} \widetilde{\nu}_{\q}(i)<+\infty$. This case is more complicated since renewal arguments are not available anymore, and the conditioning is now degenerate, so absolute continuity arguments between $P$ and $X$ become more elaborate. On the other hand, the growth is now slower and the nonconditionned walk $X$ with step distribution $\widetilde{\nu}_{\q}$ is now recurrent, so it seems more difficult to jump over a large interval. And indeed, we will prove
\[ \lim_{p \to +\infty} \P \left( \mbox{$P$ hits $p$} \right) =1,\]
which is a much stronger version of \eqref{eqn_RW_hits_interval_conditioned}.

For this, our strategy will be the following: let $\chi_p$ be the first time at which $P$ is at least $p$.
\begin{itemize}
\item
The scaling limit of $P$ is a process with no positive jump, so ${P_{\chi_p}=p+o(p)}$ in probability as $p \to +\infty$.
\item
Between time $\chi_p$ and $\tau_p+o(p^{2/3})$, the process $P$ looks a lot like a nonconditioned random walk $X$ started from $P_{\tau_p}$.
\item
If $X$ is started from $p+o(p)$, the time it takes to first hit $P$ is $o(p^{2/3})$. This is a stronger version of the recurrence of $X$, and will follow from a local limit theorem for random walks.
\end{itemize}
Let us now be more precise. By Theorem 3 of \cite{Bud15} (see also \cite[Chapter 10]{C-StFlour}), we have the convergence
\[ \left( \frac{P_{nt}}{n^{2/3}} \right)_{t \geq 0} \xrightarrow[n \to +\infty]{(d)} \left( b_{\mathbf{q}} S^+_t \right)_{t \geq 0} \]
for the Skorokhod topology, where $S^+$ is a $3/2$-stable Lévy process with no positive jump conditioned to stay positive, and $b_{\mathbf{q}}>0$ (the precise value will not matter here). Since this limiting process has no positive jump, we have $P_{\tau_p}-p=o(p)$ in probability. Hence, there is a deterministic function $f(p)$ with $\frac{f(p)}{p} \to 0$ when $p \to +\infty$ such that, for any $\eps>0$,
\[ \P \left( P_{\tau_p}-p \geq \eps f(p) \right) \xrightarrow[p \to +\infty]{} 0. \]
We now fix $\eps>0$, and condition on $P_{\tau_p}=p'$ for some $p \leq p' \leq p+\eps f(p)$. We claim that then $\left( P_{\tau_p+i}-p' \right)_{0 \leq i \leq f(p)^{3/2}}$ can be coupled with $(X_i)_{0 \leq i \leq f(p)^{3/2}}$ in such a way that both processes are the same with probability $1-o(1)$. For this, recall from~\eqref{peeling_transitions} that $P$ can be described as a Doob $h$-transform of $X$, where $h$ is given by \eqref{eqn_defn_homega}. Hence, the Radon--Nikodym derivative of the first process with respect to the second is
\begin{equation}\label{critical_Radon_Nikodym}
\frac{h_{p'+X_{f(p)^{3/2}}}(1)}{h_{p'}(1)}.
\end{equation}
Since $\frac{X_{f(p)^{3/2}}}{f(p)}$ converges in distribution, we have $\frac{X_{f(p)^{3/2}}}{p} \to 0$ in probability. By using the fact that $\frac{p'}{p} \to 0$ uniformly in $p'$ and that $h_1(x)\sim c\sqrt{x}$ for some $c>0$ (see Section~\ref{subsec_IBPM}), we conclude that \eqref{critical_Radon_Nikodym} goes to $1$ as $p \to +\infty$, uniformly in $p' \in \left[ p,p+\eps f(p) \right]$. This proves our coupling claim. Note that under this coupling, the time where $P$ hits exactly $p$ is $\tau_p$ plus the time where $X$ hits $p-p'$.

We will now show that, if $p$ is large enough, for any $k \in [-\eps f(p),0]$, we have
\begin{equation}\label{eqn_critical_quantitative_recurrence}
\P \left( \mbox{$X$ hits $k$ before time $f(p)^{3/2}$} \right) \geq 1-\delta(\eps),
\end{equation}
where $\delta(\eps) \to 0$ as $\eps \to 0$. Together with our coupling result, this will imply that the probability for $P$ to hit $p$ before time $\tau_p+f(p)^{3/2}$ is at least $1-\delta (\eps)-o(1)$ as $p \to +\infty$. Since this is true for any $\eps>0$, this will conclude the proof of Lemma~\ref{lem_RW_hits_interval} in the critical case.

The proof of \eqref{eqn_critical_quantitative_recurrence} relies on the Local Limit Theorem (this is e.g. Theorem 4.2.1 of \cite{IL71}). This theorem (in the case $\alpha=3/2$) states that
\[ \sup_{k \in \Z} \left| n^{2/3} \P (X_n=k)-g \left( \frac{k}{n^{2/3}} \right) \right| \xrightarrow[n \to +\infty]{} 0, \]
where $g$ is a continuous function (the density of a $3/2$-stable variable).
On the other hand, let us denote $t=f(p)^{3/2}$. By the strong Markov property, for all $k \in \Z$, we have
\begin{align*}
\E_0 \left[ \sum_{i=0}^{t} \mathbbm{1}_{X_i=k} \right] &\leq \P_0 \left( \mbox{$X$ hits $k$ before time $t$} \right) \E_k \left[  \sum_{i=0}^{t} \mathbbm{1}_{X_i=k} \right]\\
&= \P_0 \left( \mbox{$X$ hits $k$ before time $t$} \right) \E_0 \left[  \sum_{i=0}^{t} \mathbbm{1}_{X_i=0} \right].
\end{align*}
Therefore, using the local limit theorem, we can write, for $-\eps f(p) \leq k \leq 0$ and $p$ large (the $o$ terms are all uniform in $k$):
\begin{align*}
\P_0 \left( \mbox{$X$ hits $k$ before $t$} \right) &\geq \frac{\sum_{i=0}^{t} \P_0 \left( X_i=k \right)}{\sum_{i=0}^{t} \P_0 \left( X_i=0 \right)}\\
&= \frac{\sum_{i=1}^{t} \left( \frac{1}{i^{2/3}} g \left( \frac{k}{i^{2/3}} \right) + o \left( \frac{1}{i^{2/3}} \right) \right)}{1+\sum_{i=1}^{t} \left( \frac{1}{i^{2/3}} g (0) + o \left( \frac{1}{i^{2/3}} \right) \right)}\\
&\geq \frac{-\eps t^{1/3} + \sum_{i=\eps t}^{t} \left( \frac{1}{i^{2/3}} \min_{[-\eps^{1/3},0]} g +o \left( \frac{1}{i^{2/3}} \right)\right)}{3t^{1/3} g(0) + \eps t^{1/3}}\\
&\geq \frac{-2\eps t^{1/3}+ \left( 3t^{1/3}-3\eps^{1/3} t^{1/3}\right) \min_{[-\eps^{1/3},0]} g}{\left( 3g(0)+\eps \right)t^{1/3}}\\
&= \frac{-2\eps +3(1-\eps^{1/3}) \min_{[-\eps^{1/3},0]} g}{3g(0)+\eps},
\end{align*}
where the third line uses that, for any index $i \geq \eps t$, we have
\[0 \geq \frac{k}{i^{2/3}} \geq -\frac{\eps f(p)}{(\eps t)^{2/3}}=-\eps^{1/3}. \]
We obtain a lower bound that goes to $1$ as $\eps \to 0$, so this proves \eqref{eqn_critical_quantitative_recurrence}, and Lemma~\ref{lem_RW_hits_interval}.
\end{proof}

\begin{proof}[Proof of Proposition~\ref{prop_finding_regions_to_cut}]
The subtlety in the proof is that we would like to say something about the finite maps $M_n$ conditionally on the values of $\QQ^1$ and $\QQ^2$, but $\QQ^1$ and $\QQ^2$ are defined in terms of the infinite limits. However, we can condition on the maps explored at the time when the perimeters of the explored parts hit $[p,p+C]$ for the first time. Then Proposition~\ref{prop_q_as_limit} guarantees that from these explored parts, we can get good approximations of $\QQ^1$ and $\QQ^2$ if $p$ is large enough.

In this proof, we will use a shortened notation for our peeling explorations. For $i \in \{1,2\}$ and $t \geq 0$, we will write $\expl^{n,i}_t=\expl_t^{\mathcal{A}} \left( M_n, e_n^i \right)$ and $\expl^{\infty,i}_t=\expl_t^{\mathcal{A}} \left( \MM^i_{\QQ^i} \right)$.

We fix $\eps>0$. By Lemma~\ref{lem_RW_hits_interval}, let $C$ be a constant (depending only on $\eps$) such that
\[ \liminf_{p \to +\infty} \P \left( \sigma_{[p,p+C]}^{1,\infty} <+\infty \right) > 1-\frac{\eps}{20}.\]
For $p$ large enough (where "large enough" may depend on $\eps$), there is $v_0=v_0(p)$ such that
\begin{equation}\label{eqn_sigma_reasonable}
\P \left( \sigma^{\infty,1}_{[p,p+C]} \leq v_0 \mbox{ and } \left| \expl^{\infty,1}_{\sigma^{\infty,1}_{[p,p+C]}} \right| \leq v_0 \right)>1-\frac{\eps}{20},
\end{equation}
where $|m|$ is the number of edges of a map $m$.
On the other hand, let us fix $j \in \N^* \cup \{\infty\}$. Proposition~\ref{prop_q_as_limit} provides a function $\widetilde{r}_j$ on the set of finite maps with a hole such that $\widetilde{r}_j(\expl^{\infty, 1}_t) \to r_j(\QQ^1)$ almost surely as $t \to +\infty$. Let $\eta<1$ be a small constant, which will be fixed later and will only depend on $\eps$. For $p$ large enough, we have
\begin{equation}\label{eqn_fq_approximation}
\P \left( \sigma^{\infty,1}_{[p,p+C]} \leq v_0(p) \mbox{ but } \left| \widetilde{r}_j \left( \expl^{\infty, 1}_{\sigma^{\infty,1}_{[p,p+C]}} \right) - r_j(\QQ^1) \right| \geq \frac{\eps}{8} \right) < \eta \frac{\eps}{20}.
\end{equation}
From now on, we take $p$ large enough so that both \eqref{eqn_sigma_reasonable} and \eqref{eqn_fq_approximation} hold. By almost sure local convergence and \eqref{eqn_sigma_reasonable}, for $n$ large enough (where "large enough" may depend on $\eps$ and $p$), we have
\[ \P \left( \sigma^{n,1}_{[p,p+C]}, \sigma^{n,2}_{[p,p+C]} \leq v_0 \mbox{ and } \left| \expl^{n,1}_{\sigma^{n,1}_{[p,p+C]}} \right|, \left| \expl^{n,2}_{\sigma^{n,2}_{[p,p+C]}} \right| \leq v_0 \right) > 1-\frac{\eps}{10}.\]
By the assumption that $|r_j(\QQ^1)-r_j(\QQ^2)|>\eps$ with probability at least $\eps$ and by \eqref{eqn_fq_approximation}, we deduce that
\[ \P \left(\begin{array}{c} \sigma^{n,1}_{[p,p+C]}, \sigma^{n,2}_{[p,p+C]} \leq v_0 \mbox{ and } \left| \expl^{n,1}_{\sigma^{n,1}_{[p,p+C]}} \right|, \left| \expl^{n,2}_{\sigma^{n,2}_{[p,p+C]}} \right| \leq v_0\\ \mbox{and } \left| \widetilde{r}_j \left( \expl^{\infty, 1}_{\sigma^{\infty,1}_{[p,p+C]}} \right) - \widetilde{r}_j \left( \expl^{\infty, 2}_{\sigma^{\infty,2}_{[p,p+C]}} \right) \right| \geq \frac{3}{4}\eps\end{array} \right) > \frac{4}{5} \eps.\]

Note that if this last event occurs but the two regions $\expl^{n,1}_{\sigma^{n,1}_{[p,p+C]}}$ and $\expl^{n,2}_{\sigma^{n,2}_{[p,p+C]}}$ have a common face, then the dual graph distance between the two roots is bounded by $2v_0$. However, by Proposition~\ref{prop_tightness_dloc_univ}, the volume of the ball of radius $2v_0$ around $e^1_n$ is tight as $n \to +\infty$, so the probability that this happens goes to $0$ as $n \to +\infty$. Hence, for $n$ large enough:
\begin{equation}\label{eqn_two_distinct_regions}
\P \left( \begin{array}{c}\expl^{n,1}_{\sigma^{n,1}_{[p,p+C]}}, \expl^{n,2}_{\sigma^{n,2}_{[p,p+C]}} \mbox{ are well-defined, face-disjoint, have at}\\\mbox{most $v_0$ edges and } \left| \widetilde{r}_j \left( \expl^{\infty, 1}_{\sigma^{\infty,1}_{[p,p+C]}} \right) - \widetilde{r}_j \left( \expl^{\infty, 2}_{\sigma^{\infty,2}_{[p,p+C]}} \right) \right| \geq \frac{3}{4}\eps \end{array} \right) > \frac{4}{5} \eps.
\end{equation}

Now assume that this last event occurs and condition on the $\sigma$-algebra $\mathcal{F}_{\sigma}$ generated by the pair $\left( \expl^{n,1}_{\sigma^{n,1}_{[p,p+C]}}, \expl^{n,2}_{\sigma^{n,2}_{[p,p+C]}} \right)$ of explored regions. Then, let $I_1, I_2 \in [0,C]$ be such that the perimeters of the two explored regions are $2p+2I_1$ and $2p+2I_2$. Then the complementary map is a uniform map of the $\left( 2p+2I_1, 2p+2I_2 \right)$-gon with genus $g_n$ and face degrees given by $\mathbf{\widetilde{F}}^n$ as follows. If $F_j^n$ is the number of internal faces of degree $2j$ in $\expl^{n,1}_{\sigma^{n,1}_{[p,p+C]}} \cup \expl^{n,2}_{\sigma^{n,2}_{[p,p+C]}}$, then $\widetilde{F}_j^n=f_j^n-F_j^n$. We now perform $I_1$ peeling steps according to $\mathcal{A}$ around $\expl^{n,1}_{\sigma^{n,1}_{[p,p+C]}}$, followed by $I_2$ peeling steps according to $\mathcal{A}$ around $\expl^{n,2}_{\sigma^{n,2}_{[p,p+C]}}$. We call a peeling step \emph{nice} if it consists of gluing together two boundary edges, which decreases the perimeter by $2$. The number of possible values of the map $\M_n \backslash \left( \expl^{n,1}_{\sigma^{n,1}_{[p,p+C]}} \cup \expl^{n,2}_{\sigma^{n,2}_{[p,p+C]}} \right)$ is
\[\beta_{g}^{(p+I_1,p+I_2)}(\mathbf{\widetilde{F}}^n).\]
On the other hand, if the $I_1+I_2$ additional peeling steps are all good and the regions around $e_n^1$ and $e_n^2$ are still disjoint after these steps, the number of possible complementary maps is
\[\beta_{g}^{(p,p)} (\mathbf{\widetilde{F}}^n). \]
It follows that
\[ \P \left( \mbox{the $I_1+I_2$ peeling steps are all nice} | \mathcal{F}_{\sigma} \right) = \frac{\beta_{g}^{(p,p)} (\mathbf{\widetilde{F}}^n)}{\beta_{g}^{(p+I_1,p+I_2)}(\mathbf{\widetilde{F}}^n)}. \]
Since $|\mathbf{F}^n|$ is bounded by $v_0(p)$, for $n$ large enough (where "large enough" may depend on $p$), the Bounded ratio Lemma applies to $\mathbf{\widetilde{F}}^n$. Therefore, by the Bounded ratio Lemma (more precisely, by Corollary~\ref{lem_BRL_boundaries}, item 2), the last ratio is always larger than a constant $\eta$ depending on $\eps$. More precisely $\eta$ may depend on $I_1$ and $I_2$, but $0 \leq I_1, I_2 \leq C(\eps)$, so $(I_1, I_2)$ can take finitely many values given $\eps$, so $\eta$ only depends on $\eps$ (and not on $p$). This is the value of $\eta$ that we choose for \eqref{eqn_fq_approximation}. For $i \in \{1,2\}$, we write $\tau_p^{n,i}=\sigma_{[p,p+C]}^{n,i}+I_i$. If the last $I_1+I_2$ peeling steps are nice, then after they are performed, both explored regions have perimeter $2p$. Therefore, it follows from the last computation and from \eqref{eqn_two_distinct_regions} that, for $n$ large enough, we have
\[ \P \left( \begin{array}{c}\expl^{n,1}_{\tau_p^{n,1}} \mbox{ and } \expl^{n,2}_{\tau_p^{n,2}} \mbox{ are both face-disjoint, have perimeter $2p$}\\\mbox{and volume $\leq v_0$, and } \left| \widetilde{r}_j \left( \expl^{\infty, 1}_{\sigma^{\infty,1}_{[p,p+C]}} \right) - \widetilde{r}_j \left( \expl^{\infty, 2}_{\sigma^{\infty,2}_{[p,p+C]}} \right) \right| \geq \frac{3}{4}\eps\end{array} \right) \geq \frac{4}{5} \eps \eta.\]
Finally, we can use \eqref{eqn_fq_approximation} to replace back the approximations $\widetilde{r}_j \left( \expl^{\infty, i}_{\sigma^{\infty,i}_{[p,p+C]}} \right)$ by $r_j(\QQ^i)$. We obtain
\[ \P \left( \begin{array}{c}
\expl^{1,n}_{\tau_p^{n,i}} \mbox{ and } \expl^{2,n}_{\tau_p^{n,2}} \mbox{ are both face-disjoint, have perimeter $2p$ }\\ \mbox{and $\leq v_0$ edges, and } \left| r_j(\QQ^1)-r_j(\QQ^2) \right| \geq \frac{1}{2}\eps
\end{array} \right) \geq \frac{3}{5} \eps \eta.\]
On this event, we have $(\M_n, e_n^1, e_n^2) \in \mathscr{S}_{n,p,v_0}$. Therefore, this concludes the proof of the proposition, with $\delta=\frac{3}{5} \eta \eps$.
\end{proof}

\subsection{The two holes argument: proof of Proposition~\ref{prop_two_holes_argument} and Corollary~\ref{corr_main_minus_monotonicity}}
\label{subsec_two_holes}

Now that we have Proposition~\ref{prop_finding_regions_to_cut}, the proof of Proposition~\ref{prop_two_holes_argument} is basically the same as two holes argument in~\cite{BL19} (i.e. the proof of Proposition 18). Therefore, we will not write the argument in full details, but only sketch it. We first stress two differences:
\begin{itemize}
\item
The first one is that the involution obtained by (possibly) swapping the two explored parts is now non-identity on a relatively small set of maps (but still on a positive proportion). The only consequence is that in the end, instead of contradicting the almost sure convergence of Proposition~\ref{prop_q_as_limit} on an event of probability $\eps$, we will contradict it on an event of probability $\delta<\eps$, where $\delta$ is given by Proposition~\ref{prop_finding_regions_to_cut}.
\item
The other difference is that in \cite{BL19}, the only observable we were using to approximate the Boltzmann weights was the ratio between perimeter and volume, which corresponds to our function $r_{\infty}$. Here we also need to deal with the functions $r_j$ for $j \in \N^*$. For this, we simply need the observation that, if $q$ is much larger than $p$, the proportion of peeling steps before $\tau_q$ where we discover a new face of perimeter $2j$ depends almost only on the part of the exploration between $\tau_p$ and $\tau_q$.
\end{itemize}

\begin{proof}[Sketch of proof of Proposition~\ref{prop_two_holes_argument}]
Fix $j \in \N^* \cup \{ \infty\}$. Let $\eps>0$, and assume
\begin{equation}\label{eqn_q1_ne_q2}
\P \left( |r_j(\QQ^1)-r_j(\QQ^2)|>\eps \right)>\eps.
\end{equation}
Let $\delta>0$ be given by Proposition~\ref{prop_finding_regions_to_cut}. Consider $p$ large (depending on $\eps$) and let $v_0$ be given by Proposition~\ref{prop_finding_regions_to_cut}. We assume that the peeling algorithm $\mathcal{A}$ that we work with has the property that the edge peeled at time $t$ for $t \geq \tau_p$ only depends on $\expl_t^{\mathcal{A}}(m) \backslash \expl_{\tau_p}^{\mathcal{A}}(m)$ (see~\cite{BL19} for a more careful description of this property). This is not a problem since Proposition~\ref{prop_finding_regions_to_cut} is independent of the choice of $\mathcal{A}$.

We define an involution $\Phi_n$ on the set of bi-rooted maps with genus $g_n$ and face degrees $\ff^n$ as follows: if $m \in \mathscr{S}_{n,p,v_0}$, then $\Phi_n(m)$ is obtained from $m$ by swapping the two regions $\expl^{1,n}_{\tau_p^{1,n}}$ and $\expl^{2,n}_{\tau_p^{2,n}}$. If $m \notin \mathscr{S}_{n,p,v_0}$, then $\Phi_n(m)=m$. Note that $\Phi_n(\M_n,e^1_n,e^2_n)$ is still uniform on bi-rooted maps with prescribed genus and face degrees. This map rooted at $e_1^n$ converges to a map $\widehat{\M}$, with either
\[\widehat{M}=\M^1_{\QQ^1}\]
or
\[\expl^{\mathcal{A}}_{\tau^1_p}(\widehat{M}) = \expl^1_{\tau^1_p} \quad \mbox{ and } \quad \widehat{M} \backslash \expl^{\mathcal{A}}_{\tau^1_p}(\widehat{M}) = M^2_{\QQ^2} \backslash \expl^2_{\tau^2_p}. \]

Moreover, by Proposition~\ref{prop_finding_regions_to_cut}, if $p$ has been chosen large enough, then with probability at least $\delta$, we are in the second case and furthermore $|r_j(\QQ^1)-r_j(\QQ^2)|>\frac{\eps}{2}$. Now assume that this last event occurs and that $q \gg p \gg 1$. Then we have
\begin{equation}\label{eqn_rtilde_and_r_badcase}
\widetilde{r}_j \left( \expl^{\mathcal{A}}_{\tau^1_p}(\widehat{M}) \right) \approx r_j(\QQ^1) \quad \mbox{ and } \quad \widetilde{r}_j \left( \expl^{\mathcal{A}}_{\widehat{\tau}_q}(\widehat{M}) \right) \approx r_j(\QQ^2),
\end{equation}
where $\widehat{\tau}_q$ is the first step where the perimeter of the explored part of $\widehat{\M}$ is at least $q$. The approximations of \eqref{eqn_rtilde_and_r_badcase} can be made arbitrarily precise if $p$ and $q$ were chosen large enough, so for $p$ large enough and $q$ large enough (depending on $p$), we have
\begin{equation}\label{eqn_rj_different}
\P \left( \left| \widetilde{r}_j \left( \expl^{\mathcal{A}}_{\tau^1_p}(\widehat{M})  \right) - \widetilde{r}_j \left( \expl^{\mathcal{A}}_{\widehat{\tau}_q}(\widehat{M}) \right) \right| > \frac{\eps}{4} \right) \geq \delta.
\end{equation}
On the other hand $\widehat{M}$ is a local limit of finite uniform maps, so by Theorem~\ref{thm_main_more_general} it has to be a mixture of Boltzmann infinite planar maps. But then~\eqref{eqn_rj_different} contradicts the almost sure convergence of Proposition~\ref{prop_q_as_limit}, so~\eqref{eqn_q1_ne_q2} cannot be true. Therefore, we must have $r_j(\QQ^1)=r_j(\QQ^2)$ a.s.. Since this is true for all $j \in \N^* \cup \{\infty\}$, by the last point of Proposition~\ref{prop_q_as_limit}, we have $\QQ^1=\QQ^2$, which concludes the proof.
\end{proof}

The passage from Proposition~\ref{prop_two_holes_argument} to Corollary~\ref{corr_main_minus_monotonicity} does also not require any new idea compared to~\cite{BL19}, so we do not write it down completely.

\begin{proof}[Sketch of the proof of Corollary~\ref{corr_main_minus_monotonicity}]
The proof is basically the same as the end of the proof of the main theorem in \cite{BL19}. The only difference is that we could not prove directly that $d(\q)$ and $\left( a_j(\q) \right)_{j \geq 1}$ are sufficient to characterize the weight sequence $\q$, so the result that we obtain is only Corollary~\ref{corr_main_minus_monotonicity} and not Theorem~\ref{univ_main_thm}.

More precisely, by the Euler formula, any map with genus $g_n$ and face degrees $\ff^n$ has exactly $|\ff^n|$ edges and $|\ff^n|-\sum_{j \geq 1} f_j^n +2-2g_n$ vertices, so, by invariance of $M_n$ under uniform rerooting, we have
\[ \E \left[ \frac{1}{\deg_{M_n}(\rho)} \right] = \frac{|\ff^n|-\sum_{j \geq 1} f_j^n +2-2g_n}{2|\ff^n|} \xrightarrow[n \to +\infty]{} \frac{1}{2} \left( 1-2\theta - \sum_j \alpha_j \right). \]
By the exact same argument as in \cite{BL19}, we deduce from Proposition~\ref{prop_two_holes_argument} that if $M_n \to \MM_{\QQ}$, then $d(\QQ)=\frac{1}{2} \left( 1-2\theta - \sum_j \alpha_j \right)$ a.s.. Similarly, by invariance under rerooting, for all $j \geq 1$, we have
\[ \P \left( \mbox{the root face of $M_n$ has degree $2j$} \right) = \frac{2j f_j^n}{2n} \xrightarrow[n \to +\infty]{} j \alpha_j. \]
By the same argument as for the mean vertex degree, we obtain $a_j(\mathbf{Q})=\alpha_j$ a.s..
\end{proof}

\subsection{Monotonicity of the mean inverse degree}
\label{subsec_last_step}

To conclude the proof of the main theorem, given Corollary~\ref{corr_main_minus_monotonicity}, it is enough to show that if $\sum_{j \geq 1} j^2 a_j(\q) <+\infty$, the weight sequence $\q$ is completely determined by $\left( a_j(\q) \right)_{j \geq 1}$ and $d(\q)$. For all this subsection, we fix a sequence $(\alpha_j)_{j \geq 1}$ such that $\sum_j j \alpha_j =1$ and $\sum_j j^2 \alpha_j <+\infty$ and $\alpha_1<1$. We recall from Proposition~\ref{prop_third_parametrization} that the weight sequences $\q$ such that $a_j(\q)=\alpha_j$ for all $j \geq 1$ form a one-parameter family $\left( \q^{(\omega)} \right)_{\omega \geq 1}$ given by
\[ q_j^{(\omega)}=\frac{j \alpha_j}{\omega^{j-1}h_j(\omega)} c_{\q^{(\omega)}}^{-(j-1)}, \quad \mbox{where} \quad c_{\q^{(\omega)}} = \frac{4}{1-\sum_{i \geq 1} \frac{1}{4^{i-1}} \binom{2i-1}{i-1} \frac{i \alpha_i}{\omega^{i-1} h_i(\omega)}}. \]
To prove Theorem~\ref{univ_main_thm}, it is sufficient to prove the following.

\begin{prop}\label{prop_monotonicity_deg}
Under the assumption $\sum_{j \geq 1} j^2 \alpha_j<+\infty$, the function $\omega \to d(\q^{(\omega)})$ is strictly decreasing.
\end{prop}

Since we were not able to establish this result by a direct argument, we will prove it using Corollary~\ref{corr_main_minus_monotonicity}.

\subsubsection{Basic properties of the mean inverse degree function}

Before moving on to the core of the argument, we start with some basic properties of the function $\omega \to d(\q^{(\omega)})$.

\begin{lem}\label{lem_degreefunction_basic}
\begin{itemize}
\item[$\bullet$]
The function $\omega \to d(\q^{(\omega)})$ is continuous on $[1,+\infty)$ and analytic on $(1,+\infty)$.
\item[$\bullet$]
We have $d(\q^{(\omega)})>0$ for all $\omega$ and $\lim_{\omega \to +\infty} d(\q^{(\omega)}) = 0$.
\item[$\bullet$]
We have $d(\q^{(\omega)}) \leq 1-\sum_{j \geq 1} \alpha_j$ for all $\omega \geq 1$, with equality if and only if $\omega=1$.
\end{itemize}
\end{lem}

\begin{proof}
The proof of the analyticity and continuity on $(1,+\infty)$ is a bit long and delayed to Appendix~\ref{subsec_analyticity}. The third item will follow from results of Angel, Hutchcroft, Nachmias and Ray~\cite{AHNR16}. The other properties are quite easy.

We start with the continuity statement in the first item. The analyticity proved in Appendix~\ref{subsec_analyticity} implies continuity on $(1,+\infty)$ so it is sufficient to prove the continuity at $\omega=1$. By the monotone convergence theorem, the function $\omega \to c_{\q^{(\omega)}}$ is continuous at $\omega=1$, so $q_j^{(\omega)}$ is continuous at $\omega=1$ for all $j$. Therefore, for every finite map $m$ with one hole, we have
\[ \P \left( m \subset \MM_{\q^{(\omega)}} \right) \xrightarrow[\omega \to 1]{} \P \left( m \subset \MM_{\q^{(1)}} \right), \]
so $\MM_{\q^{(\omega)}} \to \MM_{\q^{(1)}}$ in distribution for the local topology. Since the inverse degree of the root vertex is bounded and continuous for the local topology, the function $\omega \to d(\q^{(\omega)})$ is continuous at $1$.

We now prove the second item: $d(\q^{(\omega)})>0$ is immediate by finiteness of vertex degrees and $d(\q^{(\omega)}) \to 0$ is equivalent to proving $\deg_{\MM_{\q^{(\omega)}}}(\rho) \to +\infty$ in probability when $\omega \to +\infty$. For this, we notice (see~\eqref{eqn_omega_infinite} above) that when $\omega \to +\infty$, we have $h_{i}(\omega) \to 1$ for all $i \geq 1$ and
\[\widetilde{\nu}_{\q^{(\omega)}}(i) \xrightarrow[\omega \to +\infty]{} \begin{cases}
0 \mbox{ if $i \leq -1$,}\\
(i+1) \alpha_{i+1} \mbox{ if $i \geq 0$}.
\end{cases}
\]
In other words, the probability of any peeling step swallowing at least one boundary vertex goes to $0$ when $\omega \to +\infty$. Therefore, if we perform a peeling exploration where we peel the edge on the right of $\rho$ whenever it is possible, the probability to complete the exploration of the root in less than $k$ steps goes to $0$ for all $k$. It follows that the root degree goes to $+\infty$ in probability.

Finally, we move on to the third item. Since $\MM_{\q}$ is stationary, if we denote by $\MM_{\q}^{\mathrm{uni}}$ a map with the law of $\MM_{\q}$ biased by the inverse of the root vertex degree, then $\MM_{\q}$ is unimodular. A simple computation shows that $d(\q) \leq 1-\sum_{j \geq 1} \alpha_j$ is equivalent to $\E \left[ \kappa_{\MM_{\q}^{\mathrm{uni}}}(\rho) \right] \geq 0$ where, if $v$ is a vertex of a map $m$:
\[ \kappa_m(v)=2\pi-\sum_{f} \frac{\deg(f)-2}{\deg(f)}\pi,\]
and the sum is over all faces that are incident to $v$ in $m$, counted with multiplicity.
Moreover, we have equality if and only if $\E[\kappa_{\MM_{\q}^{\mathrm{uni}}}(\rho)] = 0$. The fact that $\E[\kappa_{\MM_{\q}^{\mathrm{uni}}}(\rho)] \geq 0$ is then a consequence of \cite[Theorem 1]{AHNR16}. Moreover, \cite{AHNR16} shows the equivalence between 17 different definitions of hyperbolicity. In particular, we have $\E[\kappa_{\MM_{\q}^{\mathrm{uni}}}(\rho)] > 0$ (Definition 1 in \cite{AHNR16}) if and only if $p_c<p_u$ with positive probability for bond percolation on $\MM_{\q}^{\mathrm{uni}}$. This is equivalent to $\P \left( p_c<p_u \right)>0$ for bond percolation on $\MM_{\q}$, which is equivalent to $\q$ being critical (i.e. $\omega=1$) by \cite[Theorem 12.9]{C-StFlour}\footnote{More precisely \cite[Theorem 12.9]{C-StFlour} is about half-plane supercritical maps. Here is a way to extend it to full-plane maps: there is a percolation regime on the half-plane version of $\MM_{\q}$ such that with positive probability, there are infinitely many infinite clusters. For topological reasons, at most two of them intersect the boundary infinitely many times. Hence, by changing the colour of finitely many edges, with positive probability there are two infinite clusters that do not touch the boundary. Since there is a coupling in which the half-plane version of $\MM_{\q}$ is included in the full-plane version, we have with positive probability two disjoint infinite clusters in $\MM_{\q}$ in a certain bond percolation regime, so $p_c<p_u$ with positive probability.}.
\end{proof}

\subsubsection{Proof of Proposition~\ref{prop_monotonicity_deg}}\label{subsubsec_final_argument}

\paragraph{Sketch of the argument.}
Roughly speaking, the idea behind the proof of Proposition~\ref{prop_monotonicity_deg} is the following observation. We fix $j \geq 2$ and recall that $M_{\mathbf{f}, g}$ stands for a uniform map in $\B_g(\mathbf{f})$. If $\mathbf{f}$ is a face degree sequence with $f_j \geq 1$, we write $\mathbf{f}^-=\mathbf{f}-\mathbf{1}_j$. We can describe the law of $M_{\mathbf{f}^-, g}$ in terms of the law of $M_{\mathbf{f}, g}$. Indeed, let $m^0_j$ be the map of Figure~\ref{fig_m_0_j} with a hole of perimeter $2$ and only one internal face which has degree $2j$. Then $M_{\mathbf{f}^-, g}$ has the law of $M_{\mathbf{f}, g} \backslash m^0_j$ conditioned on $m^0_j \subset M_{\mathbf{f}, g}$, where the two boundary edges have been glued together.

Therefore, if we know for some suitable face degree sequences $\mathbf{f}^k$ that $M_{\mathbf{f}^k, g_k}$ converges to an infinite map of the form $\MM_{\QQ}$, we can deduce that $M_{\mathbf{f}^{k,-}, g_k}$ converges to a map $\MM_{\QQ^-}$ and express $\QQ^-$ in terms of $\QQ$. A simple computation (Equation~\ref{eqn_radon_nikodym_omega} below) shows that $\QQ^-$ has the law of $\QQ$ biased by $Q_j$. By Lemma~\ref{lem_qj_is_monotone}, this means that if $\QQ$ is not deterministic, then $\QQ^-$ is "strictly less hyperbolic" than $\QQ$ (in the sense that $\omega_{\QQ^-}<\omega_{\QQ}$). On the other hand, the maps $M_{\mathbf{f}^k, g_k}$ and $M_{\mathbf{f}^{k,-}, g_k}$ have the same genus but the second one is smaller, so it would be natural to expect it to be "more hyperbolic" than the first. We will derive a contradiction from this paradox by considering, for a fixed genus $g$, the smallest value $k_0(g)$ of $k$ such that the expected $\omega$ corresponding to $M_{\mathbf{f}^k, g_k}$ is smaller than a certain threshold.

\begin{figure}
\begin{center}
\includegraphics[scale=1]{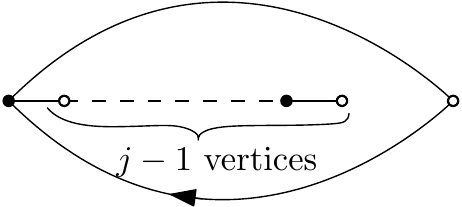}
\end{center}
\caption{The map $m_j^0$.}\label{fig_m_0_j}
\end{figure}

To make this sketch precise, we need to define precisely several objects. We will first build the face degree sequences $\mathbf{f}^k$, then explain how to approximate $\omega_{\QQ}$ on finite maps, and finally give the precise definition of $k_0(g)$ (which will actually depend on an additional parameter $t$).

\paragraph{The face degree sequences $\mathbf{F}^k$.}
Since we want our face degree sequences to respect the proportions $(\alpha_j)$, it makes sense to build the face degree sequences $\mathbf{f}^k$ randomly, by adding at each step a face of degree $2j$ with probability proportional to $\alpha_j$. This roughly means that we take $j$ random in the sketch above.
More precisely, we build our family $(\FF^k)_{k \geq 0}$ of face degree sequences as follows: let $(J_k)_{k \geq 1}$ be i.i.d. random variables on $\N^*$ with
\[ \P \left( J_k=j \right)=\frac{\alpha_j}{\sum_{i \geq 1} \alpha_i}. \] 
We write $F^k_j=\sum_{i=1}^k \mathbbm{1}_{J_i=j}$, so that $\FF^k$ is a random face degree sequence with $k$ faces in total.

\paragraph{Approximating $\omega$ on finite maps.}
In order to define $k_0(g)$, we would like to build variables $\Omega_{k,g}$ such that if $M_{\FF^k,g}$ is close to $\MM_{\q^{(\Omega)}}$ for the local topology, then $\Omega_{k,g}$ is close to $\Omega$. For this, we will rely on Lemma~\ref{lem_unif_volume}. We fix an arbitrary peeling algorithm $\mathcal{A}$ throughout this section. We recall that Lemma~\ref{lem_unif_volume} gives the convergence
\[ \frac{V_t^{(\omega)} - 2 P_t^{(\omega)}}{t} \xrightarrow[t \to +\infty]{(P)} r(\omega),\]
where $V_t^{(\omega)}=\left| \expl^{\mathcal{A}}_t \left( \MM_{\q^{(\omega)}} \right) \right|$ and $P_t^{(\omega)}=\left| \partial \expl^{\mathcal{A}}_t \left( \MM_{\q^{(\omega)}} \right) \right|$ are respectively the half-perimeter and the total number of edges of $\expl^{\mathcal{A}}_t \left( \MM_{\q^{(\omega)}} \right)$, and $r(\omega)=\frac{ \left(\sqrt{\omega}-\sqrt{\omega-1} \right)^2}{2 \sqrt{\omega(\omega-1)}}$ is a homeomorphism from $(1,+\infty]$ to $[0,+\infty)$. Therefore, for $k,g \geq 0$ and $t \geq 1$, we define
\[ \Omega_{k,g}^t = r^{-1} \left( \frac{ \left| \expl_t^{\mathcal{A}} \left( M_{\FF^k,g} \right) \right| - 2 \left| \partial \expl_t^{\mathcal{A}} \left( M_{\FF^k,g} \right) \right|}{t} \right),\]
with the convention $\Omega_{k,g}^t=+\infty$ if $M_{\FF^k,g}$ does not exist or if the peeling exploration of $M_{\FF^k,g}$ using $\mathcal{A}$ is stopped before time $t$. Note that with this definition we always have $\Omega^t_{k,g} \in [1,+\infty]$.

\paragraph{Defining $k_0^t(g)$.}
We now fix $1<\omega_0<\omega_1<+\infty$. As explained above, for $g \geq 0$ we want to consider the smallest $k$ for which the average value of $\Omega^t_{k,g}$ is smaller than $\omega_0$. Because we will need the uniform convergence of Lemma~\ref{lem_unif_volume} which holds only for $\omega$ bounded away from $1$ and $+\infty$, we first exclude artificially the values of $k$ that are too small or too large. More precisely, let $\eps>0$ be small enough to satisfy
\begin{equation}\label{eqn_choice_eps}
\eps < \min_{\omega \geq \min(\omega_0, \omega_1)} \left( d(\q^{(1)}) - d(\q^{(\omega)}) \right) \quad \mbox{and} \quad \eps < \min_{1 \leq \omega \leq \max(\omega_0, \omega_1)} d(\q^{(\omega)}).
\end{equation}
Note that the existence of such an $\eps$ is guaranteed by the second and third items of Lemma~\ref{lem_degreefunction_basic}. For $g \geq 0$, we write
\[ k_{\min}(g) = \frac{2 \sum_{j \geq 1} \alpha_j}{1-\eps-\sum_{j \geq 1} \alpha_j} g \quad \mbox{and} \quad k_{\max}(g)=\frac{1}{\eps} \left( \sum_{j \geq 1} \alpha_j \right) g.\]
Since $\eps$ will be fixed until the end, we omit the dependence in $\eps$ in the notation.
These values were chosen so that a map with face degrees $\FF^{k_{\min}(g)}$ and genus $g$ has average degree of order $\frac{1}{\eps}$ (in particular such a map exists), whereas in a map with face degrees $\FF^{k_{\max}(g)}$ and genus $g$ the genus is about $\eps$ times the size.

We now set, for $t \geq 1$ and $g \geq 0$:
\begin{align}\label{eqn_defn_k0tg}
k_0^t(g) &= \min \left\{ k \in \left[ k_{\min}(g), k_{\max}(g) \right] \big| \E \left[ \left( \Omega^t_{k,g} \right)^{-1} \right] \geq \omega_0^{-1} \right\},\\
k_1^t(g) &= \min \left\{ k \in \left[ k_{\min}(g), k_{\max}(g) \right] \big| \E \left[ \left( \Omega^t_{k,g} \right)^{-1} \right] \geq \omega_1^{-1} \right\}.\nonumber
\end{align}
The only reason why we use $\omega_0^{-1}$ instead of $\omega_0$ in the definition is to have a bounded quantity in the expectation, and pass easily from convergences in distribution to convergences of the expectation later. We first prove that this definition indeed makes sense and that the map $M_{\FF^{k_0^t(g)}, g}$ is well-defined with high probability.

\begin{lem}\label{lem_k0_tg_reasonable}
There is $t_0$ depending only on $\omega_0, \omega_1$ such that for all $t \geq t_0$:
\begin{enumerate}
\item
for $g$ large enough the number $k_0^t(g)$ is well defined;
\item
for $g$ large enough we have $k_0^t(g)>k_{\min}(g)$;
\item
with probability $1-o(1)$ as $g \to +\infty$, we have $\Ver \left( \FF^{k_0^t(g)}, g \right) \geq \frac{\eps}{2} \left| \FF^{k_0^t(g)}\right|$, and the same is true with $k_0^t(g)-1$ instead of $k_0^t(g)$.
\end{enumerate}
Moreover, the same is true if we replace $k_0^t(g)$ with $k_1^t(g)$.
\end{lem}

\begin{proof}
We first recall that
\[\Ver \left( \FF^k, g \right) = 2-2g+\sum_{j \geq 1} (j-1) F^k_j.\]
Moreover, by the law of large numbers and the definition of $\FF^k$, we have
\[\frac{1}{k} \sum_{j \geq 1} (j-1) F^k_j \xrightarrow[k \to +\infty]{a.s.} \frac{1}{\sum_{i \geq 1} \alpha_i} \sum_{j \geq 1} (j-1) \alpha_j \quad \mbox{and} \quad \frac{1}{k} \left| \FF^k \right| \xrightarrow[k \to +\infty]{a.s.} \frac{1}{\sum_{i \geq 1} \alpha_i}.  \]
From here, by the choice of $k_{\min}(g)$, it follows easily that almost surely, for $g$ large enough, we have $\Ver \left( \FF^k, g \right) \geq \frac{\eps}{2} \left| \FF^k \right|$ for all $k \geq k_{\min}(g)$. In particular, the third item of the Lemma will follow from the first and the second.

We now prove the first item. We need to prove that the minimum in the definition of $k_0^t(g)$ is over a nonempty set, so it is enough to prove that, if $t$ is larger than some $t_0$, we have
\begin{equation}\label{eqn_omega_kmax}
\E \left[ \left( \Omega^t_{k_{\max}(g), g} \right)^{-1} \right]> \omega_0^{-1}
\end{equation}
for $g$ large enough. We know that $\Ver \left( \FF^k, g \right) \geq \frac{\eps}{2} \left| \FF^k \right|$ for $g$ large enough, so by Theorem~\ref{thm_main_more_general} and Corollary~\ref{corr_main_minus_monotonicity}, up to extracting a subsequence in $g$, we have the local convergence
\begin{equation}\label{eqn_subseqlimit_kmax}
M_{\FF^{k_{\max}(g)}, g} \xrightarrow[g \to +\infty]{(d)} \MM_{q^{(\Omega_{\max})}},
\end{equation}
where $\Omega_{\max}$ is a random variable on $(1,+\infty)$. On the other hand, by the law of large numbers and the definition of $k_{\max}(g)$, we have
\[ \frac{1}{g} \Ver \left( \FF^{k_{\max}(g)}, g \right) \xrightarrow[g \to +\infty]{a.s.} -2+ \frac{1}{\eps} \sum_{j \geq 1} (j-1)\alpha_j \quad \mbox{and} \quad \frac{1}{g}\left| \FF^{k_{\max}(g)}\right| \xrightarrow[g \to +\infty]{a.s.} \frac{1}{\eps}.\]
Therefore, the average degree in $M_{\FF^{k_{\max}(g)}, g}$ tends to $-\eps+\frac{1}{2} \left( 1-\sum_{j} \alpha_j \right) = d \left( \q^{(1)}\right)-\eps$. By Corollary~\ref{corr_main_minus_monotonicity}, we must have
\begin{equation}\label{eqn_degree_omegamax}
d \left( \q^{(\Omega_{\max})}\right) = d \left( \q^{(1)}\right)-\epsilon
\end{equation}
almost surely. By the first inequality in our choice~\eqref{eqn_choice_eps} of $\eps$, this implies $\Omega_{\max}<\omega_0$ a.s., so $\E \left[ \Omega_{\max}^{-1} \right] > \omega_0^{-1}$.

On the other hand, let $t \geq 1$. Since the explored map at time $t$ of a peeling exploration is a local function, the convergence~\eqref{eqn_subseqlimit_kmax} implies
\begin{equation}\label{eqn_convergence_omegamax_t}
\Omega^t_{k_{\max}(g), g} \xrightarrow[g \to +\infty]{(d)} r^{-1} \left( \frac{V_t^{(\Omega_{\max})} - 2 P_t^{(\Omega_{\max})}}{t} \right).
\end{equation}
Moreover, by the third item and the continuity in Lemma~\ref{lem_degreefunction_basic}, we also get that $\Omega_{\max}$ is supported on a compact subset of $(1,+\infty)$ depending only on $\eps$ (and not on the subsequence in $g$ that we are working with). Therefore, by the uniform convergence result of Lemma~\ref{lem_unif_volume}, the right-hand side of~\eqref{eqn_convergence_omegamax_t} converges in probability to $\Omega_{\max}$ as $t \to +\infty$, at a speed independent from the subsequence. Hence, remembering $\E \left[ \Omega_{\max}^{-1} \right] > \omega_0^{-1}$, there is $t_0$ depending only on $\eps$ such that for $t \geq t_0$, we have
\[ \E \left[ r^{-1} \left( \frac{V_t^{(\Omega_{\max})} - 2 P_t^{(\Omega_{\max})}}{t} \right)^{-1} \right] > \omega_0^{-1}. \]
Combined with~\eqref{eqn_convergence_omegamax_t}, this implies $\E \left[ \left( \Omega^t_{k_{\max}(g), g} \right)^{-1} \right] > \omega_0^{-1}$ for $g$ large enough. We have proved that for $t \geq t_0$, every subsequence in $g$ contains a subsubsequence along which~\eqref{eqn_omega_kmax} holds for $g$ large enough, so~\eqref{eqn_omega_kmax} holds for $g$ large enough. This proves the first item of the lemma.

To prove the second item, we need to show that there is $t'_0$ such that if $t \geq t'_0$, then
\[ \E \left[ \left( \Omega^t_{k_{\min}(g), g} \right)^{-1} \right] < \omega_0^{-1} \]
for $g$ large enough.
The proof is very similar to the proof of~\eqref{eqn_omega_kmax}, so we do not write it in full details. The only difference is the computation of the average degree: this time, if $\Omega_{\min}$ plays the role of $\Omega_{\max}$ for the first item, using the definition of $k_{\min}(g)$ we get $d \left( q^{(\Omega_{\min})} \right)=\frac{\eps}{2}$ a.s.. Hence, by the second inequality of~\eqref{eqn_choice_eps}, this implies $\Omega_{\min} > \omega_0$, and the end of the proof is the same.
\end{proof}

\begin{proof}[Proof of Proposition~\ref{prop_monotonicity_deg}]
Let $t$ be larger than the $t_0$ from Lemma~\ref{lem_k0_tg_reasonable}. By the third item of Lemma~\ref{lem_k0_tg_reasonable} and Theorem~\ref{thm_main_more_general}, we know that $\left( M_{\FF^{k_0^t(g)},g} \right)_{g \geq 0}$ is tight and any subsequential limit is of the form $\MM_{\QQ}$. Moreover by the law of large numbers, for all $j \geq 1$ we have $\frac{F^k_j}{\left| \mathbf{F}^k_j \right|} \to \alpha_j$ a.s. when $k \to +\infty$. Therefore, by  Corollary~\ref{corr_main_minus_monotonicity}, any subsequential limit $\MM_{\QQ}$ must satisfy $a_j(\q)=\alpha_j$ for all $j \geq 1$, so $\QQ$ is of the form $\q^{(\Omega)}$ for some random $\Omega$. Moreover, the same holds if $k_0^t(g)$ is replaced by $k_0^t(g)-1$, or $k_1^t(g)$, or $k_1^t(g)-1$. Therefore, for all $t \geq t_0$, we can fix a subsequence $S^t$ (depending on $t$) such that when $g \to +\infty$ along $S^t$ the following convergences hold jointly:
\begin{align*}
M_{\FF^{k_0^t(g)},g} & \xrightarrow[g \to +\infty]{(d)} \MM_{\q^{(\Omega_0^t)}}, & M_{\FF^{k_1^t(g)},g} & \xrightarrow[g \to +\infty]{(d)} \MM_{\q^{(\Omega_1^t)}},\\
M_{\FF^{k_0^t(g)-1},g} & \xrightarrow[g \to +\infty]{(d)} \MM_{\q^{(\Omega_0^{t,-})}}, & M_{\FF^{k_1^t(g)-1},g} & \xrightarrow[g \to +\infty]{(d)} \MM_{\q^{(\Omega_1^{t,-})}},\\
\Omega^t_{k_0^t(g),g} & \xrightarrow[g \to +\infty]{(d)} \widetilde{\Omega}^t_0, & \Omega^t_{k_1^t(g),g} & \xrightarrow[g \to +\infty]{(d)} \widetilde{\Omega}^t_1,\\
\Omega^t_{k_0^t(g)-1} & \xrightarrow[g \to +\infty]{(d)} \widetilde{\Omega}^{t,-}_0, & \Omega^t_{k_1^t(g)-1} & \xrightarrow[g \to +\infty]{(d)} \widetilde{\Omega}^{t,-}_1,\\
\frac{g}{\left| \FF^{k_0^t(g)} \right|} & \xrightarrow[g \to +\infty]{(d)} \theta_0^t, & \frac{g}{\left| \FF^{k_1^t(g)} \right|} & \xrightarrow[g \to +\infty]{(d)} \theta_1^t,
\end{align*}
where the first four convergences are for the local topology. We highlight that for our purpose it is sufficient to consider \emph{one} subsequence $S^t$, and that we will not try to understand \emph{all} subsequential limits. In what follows, we will always consider $g$ in the subsequence $S^t$ and omit to precise it. Note that we have $\Omega_0^t, \Omega_0^{t,-} \in [1,+\infty)$ and $\widetilde{\Omega}_0^t, \widetilde{\Omega}_0^{t,-} \in [1,+\infty]$, since we have no bound a priori on $\Omega^t_{k_0^t(g),g}$.
We also have $\theta_0^t \in \left[ \eps, \frac{1}{2} \left( 1-\eps-\sum_j \alpha_j \right) \right]$ by the inequality $k_{\min}(g) \leq k_0^t(g) \leq k_{\max}(g)$ and the law of large numbers to estimate $\FF^{k_{\min}(g)}$ and $\FF^{k_{\max}(g)}$. By Corollary~\ref{corr_main_minus_monotonicity}, we also know that almost surely
\begin{equation}\label{eqn_omega0_in_compact}
d \left( \q^{(\Omega_0^t)} \right) = d \left( \q^{(\Omega_0^{t,-})} \right) = \frac{1}{2} \left( 1-2\theta_0^t-\sum_{j \geq 1} \alpha_j \right).
\end{equation}
In particular $d \left( \q^{(\Omega_0^t)} \right)$ and $d \left( \q^{(\Omega_0^{t,-})} \right)$ are bounded away from $0$ and $d(\q^{(1)})$. By Lemma~\ref{lem_degreefunction_basic}, this implies that $\Omega_0^t$ and $\Omega_0^{t,-}$ take their values in a compact subset of $(1,+\infty)$ that depends only on $\eps$. Therefore, there is a subsequence $S$ such that, when $t \to +\infty$ along $S$ the following convergences hold jointly: 
\begin{align*}
\Omega_0^t & \xrightarrow[t \to +\infty]{(d)}  \Omega_0, &  \Omega_1^t & \xrightarrow[t \to +\infty]{(d)}  \Omega_1,\\
\Omega_0^{t,-} & \xrightarrow[t \to +\infty]{(d)}  \Omega_0^{-} , &  \Omega_1^{t,-} & \xrightarrow[t \to +\infty]{(d)}  \Omega_1^{-},\\
\widetilde{\Omega}_0^t & \xrightarrow[t \to +\infty]{(d)}  \widetilde{\Omega}_0, &  \widetilde{\Omega}_1^t & \xrightarrow[t \to +\infty]{(d)}  \widetilde{\Omega}_1,\\
\widetilde{\Omega}_0^{t,-} & \xrightarrow[t \to +\infty]{(d)}  \widetilde{\Omega}_0^{-} , &  \widetilde{\Omega}_1^{t,-} & \xrightarrow[t \to +\infty]{(d)}  \widetilde{\Omega}_1^{-},\\
\theta_0^t & \xrightarrow[t \to +\infty]{} \theta_0, & \theta_1^t & \xrightarrow[t \to +\infty]{} \theta_1,
\end{align*}
where $\Omega_0, \Omega_0^- \in (1,+\infty)$ and $\widetilde{\Omega}_0, \widetilde{\Omega}_0^- \in [1,+\infty]$ and $\theta_0 \in \left[ \eps, \frac{1}{2} \left( 1-\eps-\sum_j \alpha_j \right) \right]$. From now on, we will always consider $t \geq t_0$ with $t$ in this subsequence $S$, and we will omit to precise it. 

The rough sketch of the argument is to show that
\begin{equation}\label{eqn_cyclic_ineq_omega}
\E \left[ \left( \widetilde{\Omega}^{-}_0 \right)^{-1} \right] \leq \omega_0^{-1} \leq \E \left[ \left( \widetilde{\Omega}_0 \right)^{-1} \right] = \E \left[ \left( \Omega_0 \right)^{-1} \right] \leq \E \left[ \left( \Omega_0^{-} \right)^{-1} \right] = \E \left[ \left( \widetilde{\Omega}^{-}_0 \right)^{-1} \right],
\end{equation}
and deduce from the equality in the third inequality that $\Omega_0$ is deterministic and therefore equal to $\omega_0$.

More precisely, by definition of $k_0^t(g)$ and since $k_0^t(g)>k_{\min}(g)$, we have
\[ \E \left[ \left( \Omega^t_{k_0^t(g)-1} \right)^{-1} \right] < \omega_0^{-1} \leq  \E \left[ \left( \Omega^t_{k_0^t(g)} \right)^{-1} \right].\]
By letting $g \to +\infty$ (along $S^t$) and then $t \to +\infty$ (along $S$), this implies
\begin{equation}\label{eqn_ineq_omegatilde}
\E \left[ \left( \widetilde{\Omega}^{-}_0 \right)^{-1} \right] \leq \omega_0^{-1} \leq  \E \left[ \left( \widetilde{\Omega}_0 \right)^{-1} \right],
\end{equation}
which are the first and second inequalities in~\eqref{eqn_cyclic_ineq_omega}.

On the other hand, the third inequality will be obtained by the argument sketched in the beginning of this Subsection~\ref{subsubsec_final_argument}. Let us first look at the relation between $M_{\FF^{k}, g}$ and $M_{\FF^{k-1}, g}$.
We recall that $\FF^{k}=\FF^{k-1}+\mathbf{1}_{J_k}$, where $\P (J_k=j)=\frac{\alpha_j}{\sum_{i \geq 1} \alpha_i}$ for all $j \geq 1$. If we condition on $\FF^{k-1}$ and $\FF^{k}$, then the law of $M_{\FF^{k-1}, g}$ is the law of $M_{\FF^{k}, g} \backslash m_{J_k}^0$, conditioned on $m_{J_k}^0 \subset M_{\FF^{k}, g}$, where $m^0_{J_k}$ is the map with perimeter $2$ of Figure~\ref{fig_m_0_j}. Therefore, for any map $m$ with one hole, we have
\[ \P \left( m \subset M_{\FF^{k-1}, g} | J_k=j \right) = \P \left( m+m_j^0 \subset M_{\FF^{k}, g} | J_k=j, \, m_j^0 \subset  M_{\FF^{k}, g} \right), \]
where $m+m_{j}^0$ is the map obtained from $m$ by replacing the root edge of $m$ by a copy of $m_j^0$. By summing over $j$, we obtain
\[
\P \left( m \subset M_{\FF^{k-1}, g} \right) = \frac{1}{\sum_{i \geq 1} \alpha_i} \sum_{j \geq 1} \alpha_j \frac{\P \left( m+m_j^0 \subset M_{\FF^{k}, g} \right)}{\P \left( m_j^0 \subset M_{\FF^{k}, g} \right)}.
\]
We now take $k=k_0^t(g)$ and let $g \to +\infty$ (along $S^t$) to replace $M_{\FF^{k-1}, g}$ and $M_{\FF^{k}, g}$ by respectively $\MM_{\q^{(\Omega_0^{t,-})}}$ and $\MM_{\q^{(\Omega_0^{t})}}$. We note that $m_j^0+m$ has the same perimeter as $m$ but one more internal face of degree $2j$. We obtain, for every finite map $m$ with one hole,
\[ \E \left[ C_{|\partial m|} \left( \q^{(\Omega_0^{t,-})} \right) \prod_{f \in m} q^{(\Omega_0^{t,-})}_{\deg(f)/2} \right] = \frac{1}{\sum_{i \geq 1} \alpha_i} \sum_{j \geq 1} \alpha_j \frac{\E \left[ C_{\partial m} \left( \q^{(\Omega_0^{t})} \right) \times \prod_{f \in m} q^{(\Omega_0^{t})}_{\deg(f)/2} \times q^{(\Omega_0^{t})}_j \right]}{ \E \left[ C_1 \left( \q^{(\Omega_0^{t})} \right) q_j^{(\Omega_0^{t})} \right]}. \]
This can be interpreted as a Radon--Nikodym derivative, i.e. the map $\MM_{\q^{(\Omega_0^{t,-})}}$ has the law of $\MM_{\q^{(\Omega_0^{t})}}$ biased by
\begin{equation}\label{eqn_radon_nikodym_omega}
\frac{1}{\sum_{i \geq 1} \alpha_i} \sum_{j \geq 1} \alpha_j \frac{q_j^{(\Omega_0^{t})}}{\E \left[ q_j^{(\Omega_0^{t})} \right]},
\end{equation}
using the fact that $C_1(\q)=1$. Since $\Omega$ is a measurable function of the map $\MM_{\q^{(\Omega)}}$ by Proposition~\ref{prop_q_as_limit}, it follows that $\Omega_0^{t,-}$ has the law of $\Omega_0^t$ biased by \eqref{eqn_radon_nikodym_omega}. In particular, we have
\[
\E \left[ \left( \Omega_0^{t,-} \right)^{-1} \right] = \frac{1}{\sum_{i \geq 1} \alpha_i} \sum_{j \geq 1} \alpha_j \frac{\E \left[ q_j^{(\Omega_0^t)} \, \left( \Omega_0^t \right)^{-1} \right]}{\E \left[ q_j^{(\Omega_0^t)}\right]}.
\]
We can now let $t \to +\infty$ (along $S$) to get
\begin{equation}\label{eqn_expectation_omega_inverse}
\E \left[ \left( \Omega_0^{-} \right)^{-1} \right] = \frac{1}{\sum_{i \geq 1} \alpha_i} \sum_{j \geq 1} \alpha_j \frac{\E \left[ q_j^{(\Omega_0)} \, \left( \Omega_0 \right)^{-1} \right]}{\E \left[ q_j^{(\Omega_0)}\right]}.
\end{equation}
But by Lemma~\ref{lem_qj_is_monotone}, we have $\E \left[ q_j^{(\Omega_0)} \, \left( \Omega_0 \right)^{-1} \right] \geq \E \left[ q_j^{(\Omega_0)} \right] \E \left[ \left( \Omega_0 \right)^{-1} \right]$ for all $j \geq 1$, so the last display implies $\E \left[ \left( \Omega_0^{-} \right)^{-1} \right] \geq \E \left[ \left( \Omega_0 \right)^{-1} \right]$, which is the third inequality of the sketch~\eqref{eqn_cyclic_ineq_omega}.

We now move on to the two equalities of~\eqref{eqn_cyclic_ineq_omega}. For this, we need to argue that $\widetilde{\Omega}_0^t$ is a good approximation of $\Omega_0^t$ for $t$ large. By definition of $\widetilde{\Omega}_0^t$ and by local convergence, we have (the limits in $g$ are along $S^t$)
\begin{align}\label{eqn_omegatilde_t_as_limit}
\E \left[ \left( \widetilde{\Omega}_0^t \right)^{-1} \right] &= \lim_{g \to +\infty} \E \left[ \left( \Omega_{k_0^t(g),g}^t \right)^{-1} \right] \nonumber \\
&= \lim_{g \to +\infty} \E \left[ \left( r^{-1} \left( \frac{ \left| \expl_t^{\mathcal{A}} \left( M_{\FF^{k_0^t(g)},g} \right) \right| - 2 \left| \partial \expl_t^{\mathcal{A}} \left( M_{\FF^{k_0^t(g)},g} \right) \right|}{t} \right) \right)^{-1} \right] \nonumber
\\
&= \E \left[ \left( r^{-1} \left( \frac{V_t^{(\Omega_0^t)}-2 P_t^{(\Omega_0^t)}}{t} \right) \right)^{-1} \right].
\end{align}
When $t \to +\infty$ (along $S$), the left-hand side of~\eqref{eqn_omegatilde_t_as_limit} goes to $\E \left[ \left( \widetilde{\Omega}_0 \right)^{-1} \right]$. On the other hand, we recall that~\eqref{eqn_omega0_in_compact} implies that $\Omega_0^t$ lies in a compact subset of $(1,+\infty)$ depending only on $\eps$. Since $\Omega_0^t \to \Omega_0$ along $S$, by Lemma~\ref{lem_unif_volume} we have the convergence (along $S$)
\[ \frac{V_t^{(\Omega_0^t)} -2P_t^{(\Omega_0^t)}}{t} \xrightarrow[t \to +\infty]{(P)} r(\Omega_0). \]
Therefore, when $t \to +\infty$ (along $S$), the right-hand side of~\eqref{eqn_omegatilde_t_as_limit} goes to $\E \left[ \left( \Omega_0 \right)^{-1} \right]$. This proves the first equality of~\eqref{eqn_cyclic_ineq_omega}. The second one is proved in the exact same way, using $\Omega^{t,-}_0$, $\widetilde{\Omega}^{t,-}_0$ instead of $\Omega^{t}_0$, $\widetilde{\Omega}^{t}_0$.

We have therefore proved all of~\eqref{eqn_cyclic_ineq_omega}, so all the inequalities must be equalities. In particular, \eqref{eqn_expectation_omega_inverse} becomes
\[ \E \left[ \left( \Omega_0 \right)^{-1} \right] = \frac{1}{\sum_{i \geq 1} \alpha_i} \sum_{j \geq 1} \alpha_j \frac{\E \left[ q_j^{(\Omega_0)} \, \left( \Omega_0 \right)^{-1} \right]}{\E \left[ q_j^{(\Omega_0)}\right]}. \]
However, we also know by Lemma~\ref{lem_qj_is_monotone} that $\E \left[ q_j^{(\Omega_0)} \, \left( \Omega_0 \right)^{-1} \right] \geq \E \left[ q_j^{(\Omega_0)} \right] \E \left[ \left( \Omega_0 \right)^{-1} \right]$ for all $j$, so for all $j \geq 1$ we must have the equality
\[ \alpha_j \E \left[ q_j^{(\Omega_0)} \left( \Omega_0 \right)^{-1} \right] = \alpha_j \E \left[ q_j^{(\Omega_0)} \right] \E \left[ \left( \Omega_0 \right)^{-1} \right].\]
In particular, we fix $j \geq 2$ such that $\alpha_j>0$ (such a $j$ exists because $\alpha_1<1$). Then $\E \left[ q_j^{(\Omega_0)} \left( \Omega_0 \right)^{-1} \right] = \E \left[ q_j^{(\Omega_0)} \right] \E \left[ \left( \Omega_0 \right)^{-1} \right]$. Since $\omega \to \omega^{-1}$ and $\omega \to q_j^{(\omega)}$ are two decreasing functions (by Lemma~\ref{lem_qj_is_monotone}), this is only possible if $\Omega_0$ is deterministic. But then~\eqref{eqn_cyclic_ineq_omega} yields $\E \left[ \left( \Omega_0 \right)^{-1} \right] = \omega_0^{-1}$, so $\Omega_0=\omega_0$ a.s..

We can now finish the proof. By the exact same argument as for $\Omega_0$, we also have $\Omega_1=\omega_1$ a.s.. We recall that $\omega_0<\omega_1$. By letting $t \to +\infty$ (along $S$) in~\eqref{eqn_omega0_in_compact} and using the continuity result of Lemma~\ref{lem_degreefunction_basic}, we get
\begin{equation}\label{eqn_fin_omega0_theta0}
d \left( \q^{(\omega_0)} \right) = \frac{1}{2} \left( 1-2\theta_0-\sum_{j \geq 1} \alpha_j \right)
\end{equation}
and similarly
\begin{equation}\label{eqn_fin_omega1_theta1}
d \left( \q^{(\omega_1)} \right) = \frac{1}{2} \left( 1-2\theta_1-\sum_{j \geq 1} \alpha_j \right).
\end{equation}
On the other hand, by the definition~\eqref{eqn_defn_k0tg} of $k_0^t(g)$ and $k_1^t(g)$, since $\omega_0<\omega_1$, we have
\[k_0^t(g) \geq k_1^t(g)\]
for all $t$ and $g$. Therefore, we have $\left| \FF^{k_0^t(g)} \right| \geq \left| \FF^{k_1^t(g)} \right|$. By letting $g \to +\infty$ (along $S^t$) we deduce $\theta_0^t \leq \theta_1^t$. Letting $t \to +\infty$ (along $S$) we get $\theta_0 \leq \theta_1$. Combining this with~\eqref{eqn_fin_omega0_theta0} and~\eqref{eqn_fin_omega1_theta1}, this proves $d \left( \q^{(\omega_0)} \right) \geq d \left( \q^{(\omega_1)} \right)$, so the function $\omega \to d \left( \q^{(\omega)} \right)$ is nonincreasing on $(1,+\infty)$. Since it is nonconstant (for example by the second item of Lemma~\ref{lem_degreefunction_basic}) and analytic (first item of Lemma~\ref{lem_degreefunction_basic}), it is decreasing on $(1,+\infty)$. Finally, we extend the result to $[1,+\infty)$ by continuity at $1$ (first item of Lemma~\ref{lem_degreefunction_basic}).
\end{proof}

\section{Asymptotic enumeration: convergence of the ratio}\label{sec_univ_asymp}

\begin{proof}[Proof of Corollary \ref{prop_cv_ratio}]
Fix $j \geq 1$, and let $m_j^0$ be the map of Figure~\ref{fig_m_0_j} with a hole of perimeter $2$. On the one hand, we have
\[\P \left( m_j^0 \subset \mathbb{M}_{\q} \right)=C_2(\q)q_j.\]
On the other hand, 
\[\P \left( m_j^0 \subset M_{\mathbf{f}^{n},g_n} \right)=\frac{\beta_{g_n}(\mathbf{f}^{n}-\mathbf{1}_j)}{\beta_{g_n}(\mathbf{f}^{n})}.\]
The last equality is proved by contracting $m_j^0$ in $\M_{\mathbf{f}^{n},g_n}$ into the root edge of a map with face degrees given by $\mathbf{f}^{n}-\mathbf{1}_j$. The corollary follows by letting $n \to +\infty$.
\end{proof}

\appendix

\section{Appendices: proofs of technical lemmas}\label{sec_univ_annex}

\subsection{Monotonicity of $q_j^{(\omega)}$ in $\omega$}
\label{appendix_qj_monotone}

The goal of this appendix is to prove Lemma~\ref{lem_qj_is_monotone}. Let us recall the context: we fix a sequence $\left( \alpha_j \right)_{j \geq 1}$ such that $\sum_{j \geq 1} j \alpha_j = 1$. Following~\eqref{eqn_qjomega_univ}, for $\omega \geq 1$, we have
\begin{equation*}
q_j^{(\omega)}=\frac{j \alpha_j}{\omega^{j-1}h_j(\omega)} \left( \frac{1-\sum_{i \geq 1} \frac{1}{4^{i-1}} \binom{2i-1}{i-1} \frac{i \alpha_i}{\omega^{i-1} h_i(\omega)} }{4} \right)^{j-1},
\end{equation*}
where $h_i(\omega)=\sum_{s=0}^{i-1} \frac{1}{(4\omega)^s} \binom{2s}{s}$. We want to prove that $q_j^{(\omega)}$ is nonincreasing in $\omega$, and even decreasing for $j \geq 2$ provided $\alpha_j>0$.

\begin{proof}[Proof of Lemma~\ref{lem_qj_is_monotone}]
We first note that $q_1^{(\omega)}=\alpha_1$ is constant, so the result is immediate for $j=1$. We now assume $j \geq 2$.
Let us write $g_i(\omega)=\frac{\binom{2i-1}{i-1}}{(4 \omega)^{i-1} \sum_{p=0}^{i-1} \frac{1}{(4\omega)^p} \binom{2p}{p}}$. This is a decreasing function of $\omega \geq 1$, since the denominator is a polynomial with nonnegative coefficients. Moreover, we have
\begin{equation}\label{eqn_derivative_qj}
\frac{\mathrm{d} q_j^{(\omega)}}{\mathrm{d} \omega} = \frac{4^{j-2}}{\binom{2j-1}{j-1}} j \alpha_j \left( \frac{1-\sum_{i \geq 1} i \alpha_i g_i(\omega)}{4} \right)^{j-2} \sum_{i \geq 1} i \alpha_i \left( g'_j(\omega) - g'_j(\omega) g_i(\omega) - (j-1) g_j(\omega) g'_i(\omega)\right).
\end{equation}
Therefore, we want to prove that for all $i,j \geq 1$ and $\omega \geq 1$:
\begin{equation}\label{eqn_reduced_derivative}
g'_j(\omega)-g'_j(\omega)g_i(\omega )-(j-1)g_j(\omega) g'_i(\omega) \leq 0.
\end{equation}

For $i \geq 1$, we denote by $P_i(\omega)$ the polynomial in the denominator in the definition of $g_i(\omega)$, so that $g_i(\omega)=\frac{\binom{2i-1}{i-1}}{P_i(\omega)}$. We also set the convention $P_0=0$. We then have the recursion
\[P_{i+1}(\omega)=4\omega P_i (\omega)+\binom{2i}{i}\]
for all $i \geq 0$. We can then rewrite~\eqref{eqn_reduced_derivative} as $F_{i,j}(\omega) \geq 0$, where for $i,j \geq 0$, we have defined
\[F_{i,j}(\omega):=P_i^2 (\omega) P_j'(\omega)- \binom{2i-1}{i-1}\left(P_j'(\omega) P_i(\omega) +(j-1)P_i'(\omega) P_j(\omega) \right).\]
We only need $F_{i,j}(\omega) \geq 0$ for $i,j \geq 1$, but we will prove it for $i \geq 1$ and $j \geq 0$.
For this, we introduce
\begin{align*}
\Delta F_{i,j}(\omega) &:=F_{i,j+1}(\omega)-4\omega F_{i,j}(\omega)\\&=4P_j(\omega) P_i^2(\omega)- \binom{2i-1}{i-1}\left(4P_j(\omega) P_i(\omega) +j\binom{2j}{j}P_i'(\omega) +4\omega P_i'(\omega) P_j(\omega) \right)
\end{align*}
by the recursion. We also set
\begin{align}
\Delta^2 F_{i,j}(\omega) &:=\Delta F_{i,j+1}(\omega)-4\omega \Delta F_{i,j}(\omega) \\&=4\binom{2j}{j}P_i^2(\omega)-\binom{2i-1}{i-1} \times 4\binom{2j}{j}P_i(\omega) \nonumber \\
\\&\hspace{0.5cm}-\binom{2i-1}{i-1}\left( (j+1)\binom{2(j+1)}{j+1}-4\omega (j-1)\binom{2j}{j} \right) P_i'(\omega) \nonumber \\
&\geq \binom{2j}{j}\left(4P_i^2(\omega)-\binom{2i-1}{i-1} \left(4P_i(\omega)+6P_i'(\omega)\right)\right), \label{eqn_defn_delta2F}
\end{align}
where the last inequality just follows from $\omega \geq 1$.

Now let $G_i(\omega):=P_i^2(\omega)-\binom{2i-1}{i-1} \left(P_i(\omega)+\frac	{3}{2}P_i'(\omega)\right)$. We will prove that for $\omega\geq 1$ and $i \geq 1$, we have $G_i(\omega)\geq 0$. On the one hand, since $\deg(P_i)=i-1$ and $P_i$ has positive coefficients, it is clear that for all $k\geq i$:
\[[\omega^k]G_i\geq 0.\]
On the other hand, let $k\leq i-2$. Then
\begin{align*}
[\omega^k]G_i&\leq [\omega^k] P_i^2-\binom{2i-1}{i-1} [\omega^k] P_i'\\&=\sum_{a+b=k} 4^k \binom{2(i-1-a)}{i-1-a} \binom{2(i-1-b)}{i-1-b}\\&\hspace{0.5cm}-4^{k+1}(k+1)\binom{2i-1}{i-1}\binom{2(i-2-k)}{i-2-k}.
\end{align*}
It is easily checked that the quantity $\binom{2(i-1-a)}{i-1-a} \binom{2(i-1-b)}{i-1-b}$ is maximal when $a=k$ or $b=k$.
Finally, since we have
\[4\binom{2i-1}{i-1}\binom{2(i-2-k)}{i-2-k}\geq \binom{2(i-1)}{i-1} \binom{2(i-1-k)}{i-1-k}, \]
we can conclude that
\[[\omega^k]G_i\leq 0.\]
Therefore, we have proved that, starting from the highest order coefficients, the coefficients of the polynomial $G_i$ are all positive, then all negative. For $i \geq 2$, since $G_i$ is nonconstant, this implies that $G_i$ has a unique nonnegative real root $\omega^*$, and $G_i(\omega)$ is positive for all $\omega>\omega^*$. We now claim that $\omega^*=1$. Indeed, it is easily checked that 
\[P_i(1)=\frac{i}{2}\binom{2i}{i} \quad \mbox{and} \quad P'_i(1)=\frac{i(i-1)}{3}\binom{2i}{i},\]
so $G_i(1)=0$ and $\omega^*=1$.
Therefore, if $i \geq 2$, we have $G_i(\omega) \geq 0$ for $\omega \geq 1$. This is also true for $i=1$, since then $G_i$ is constant and equal to $1$.

Therefore, by~\eqref{eqn_defn_delta2F}, for all $i,j \geq 1$ and $\omega \geq 1$, we have 
\[\Delta^2 F_{i,j}(\omega) \geq 0.\]
But since $\Delta F_{i,0}=0$, we have
$\Delta F_{i,j}(\omega) \geq 0$,
and since $F_{i,0}=0$, we have
\[F_{i,j}(\omega) \geq 0.\]
This proves~\eqref{eqn_reduced_derivative}, and the nonincreasing statement in the lemma.

It remains to check that $q_j^{(\omega)}$ is decreasing if $j \geq 2$ and $\alpha_j>0$. For this, note that if $i \geq 2$, then $G_i$ is nonconstant so $G_i(\omega)>0$ for $\omega>1$. From here, we deduce $\Delta^2 F_{i,j}(\omega)>0$ for $i \geq 2$, $j \geq 0$ and $\omega>1$, so $\Delta F_{i,j}(\omega) >0$ for $i \geq 2$, $j \geq 1$ and $\omega>1$, so $F_{i,j}(\omega)>0$ for $i,j \geq 2$ and $\omega>1$. Therefore~\eqref{eqn_reduced_derivative} is strict for $i,j \geq 2$ and $\omega>1$. By~\eqref{eqn_derivative_qj} and the fact that $\alpha_j>0$, this implies $\frac{\mathrm{d} q_j^{(\omega)}}{\mathrm{d} \omega}<0$ for $j \geq 2$ and $\omega>1$, so $q_j^{(\omega)}$ is decreasing in $\omega$ for $j \geq 2$.
\end{proof}

\subsection{Estimation lemmas for planarity and one-endedness}\label{sec_univ_estim_lemmas}

Our goal here is to prove Lemmas~\ref{lem_calcul_planar} and~\ref{lem_calcul_OE}. We first restate the recursion \eqref{rec_biparti_genre_univ} of~\cite{Lo19}, which will be our main tool:

\begin{equation}\label{rec_biparti_genre_univ_bis}
\binom{|\mathbf{f}|+1}{2}\beta_g(\mathbf{f})=\hspace{-0.5cm}\sum_{\substack{\mathbf{h}^{(1)}+\mathbf{h}^{(2)}=\mathbf{f}\\g^{(1)}+g^{(2)}+g^*=g}}\hspace{-0.5cm}(1+|\mathbf{h}^{(1)}|)\binom{v \left( \mathbf{h}^{(2)},g^{(2)} \right) }{2g^*+2}\beta_{g^{(1)}}(\mathbf{h}^{(1)}) \beta_{g^{(2)}}(\mathbf{h}^{(2)})+\sum_{g^*\geq 0}\binom{v \left( \mathbf{f},g \right) +2g^*}{2g^*+2}\beta_{g-g*}(\mathbf{f}).
\end{equation}

We start with a few easy but useful consequences of \eqref{rec_biparti_genre_univ_bis}.

\begin{lem}\label{lem_couper_en_deux}
Fix $\kappa>0$ and assume that $v(\ff,g) \geq \kappa |\ff|$.
\begin{enumerate}
\item
We have
\[\sum_{\substack{\mathbf{h}^{(1)}+\mathbf{h}^{(2)}=\mathbf{f}\\g^{(1)}+g^{(2)}=g}} | \mathbf{h}^{(1)} | \beta_{g^{(1)}}(\mathbf{h}^{(1)})  \times | \mathbf{h}^{(2)}| \beta_{g^{(2)}}(\mathbf{h}^{(2)}) \leq \frac{16}{\kappa^2} |\ff| \beta_g(\mathbf{f}). \]
\item
We have
\[\sum_{\substack{\mathbf{h}^{(1)}+\mathbf{h}^{(2)}=\mathbf{f}\\ g^{(1)}+g^{(2)}=g}} \beta_{g^{(1)}}(\mathbf{h}^{(1)}) \beta_{g^{(2)}}(\mathbf{h}^{(2)}) \leq \frac{32}{\kappa^2} \beta_g(\mathbf{f}).\]
\item
For every $g^* \geq 0$, there is a constant $C>0$ such that, if $|\ff|$ is large enough:
\[\beta_{g-g^*}(\mathbf{f}) \leq C |\mathbf{f}|^{-2g^*} \beta_g(\mathbf{f}).\]
\end{enumerate}
\end{lem}

\begin{proof}
We start with the first item. 
By only keeping the terms with $g^*=0$ and $v \left( \mathbf{h}^{(2)},g^{(2)} \right) \geq v \left( \mathbf{h}^{(1)},g^{(1)} \right)$ in the first sum of \eqref{rec_biparti_genre_univ_bis}, we get:
\begin{equation}\label{ineq_partielle}
\binom{|\mathbf{f}|+1}{2}\beta_g(\mathbf{f}) \geq \hspace{-0.5cm}\sum_{\substack{\mathbf{h}^{(1)}+\mathbf{h}^{(2)}=\mathbf{f} \\ g^{(1)}+g^{(2)}=g \\ v \left( \mathbf{h}^{(2)},g^{(2)} \right) \geq v \left( \mathbf{h}^{(1)},g^{(1)} \right)}}\hspace{-0.5cm}(1+|\mathbf{h}^{(1)}|)\binom{v \left( \mathbf{h}^{(2)},g^{(2)} \right) }{2}\beta_{g^{(1)}}(\mathbf{h}^{(1)}) \beta_{g^{(2)}}(\mathbf{h}^{(2)}).
\end{equation}
Since $v \left( \mathbf{h}^{(2)},g^{(2)} \right) + v \left( \mathbf{h}^{(1)},g^{(1)} \right)=v \left( \ff, g\right) +2$, in every term of \eqref{ineq_partielle}, we have*
\[ v \left( \mathbf{h}^{(2)},g^{(2)} \right)-1 \geq \frac{1}{2} v \left( \mathbf{\ff},g \right)\geq \frac{\kappa}{2} |\ff| \geq \frac{\kappa}{2} |\mathbf{h}^{(2)}|.\]
This implies $\binom{v \left( \mathbf{h}^{(2)},g^{(2)} \right)}{2}\geq \frac{\kappa^2}{8} |\ff| \times |\mathbf{h}^{(2)}|$. Using the crude bounds $\binom{|\ff|+1}{2} \leq |\ff|^2$ and $1+|\mathbf{h}^{(1)}| \geq |\mathbf{h}^{(1)}|$, we obtain
\begin{equation}\label{eqn_lem_couper_en_deux}
|\ff|^2 \beta_g(\mathbf{f}) \geq \frac{\kappa^2}{8} \hspace{-0.5cm}\sum_{\substack{\mathbf{h}^{(1)}+\mathbf{h}^{(2)}=\mathbf{f} \\ g^{(1)}+g^{(2)}=g \\ v \left( \mathbf{h}^{(2)},g^{(2)} \right) \geq v \left( \mathbf{h}^{(1)},g^{(1)} \right)}} \hspace{-0.5cm} |\ff| \times |\mathbf{h}^{(1)}| \times |\mathbf{h}^{(2)}| \times \beta_{g^{(1)}}(\mathbf{h}^{(1)}) \beta_{g^{(2)}}(\mathbf{h}^{(2)}).
\end{equation}
The equation is now symmetric, so up to losing an extra factor two, we can remove the assumption $v \left( \mathbf{h}^{(2)},g^{(2)} \right) \geq v \left( \mathbf{h}^{(1)},g^{(1)} \right)$ and we get the first item. The second item follows straight from the first one and the inequality $|\mathbf{h}^{(1)}| \times |\mathbf{h}^{(2)}| \geq \frac{|\ff|}{2}$, since $|\mathbf{h}^{(1)}|+|\mathbf{h}^{(2)}|=\frac{|\ff|}{2}$ (except if $\mathbf{h}^{(1)}=\mathbf{0}$ or $\mathbf{h}^{(2)}=\mathbf{0}$, but then the terms with $g_1,g_2>0$ do not contribute and the terms $g_1=0$ or $g_2=0$ contribute exactly $\beta_g(\mathbf{f})$ each).

Finally, for the third item, we simply need to consider the term corresponding to $g^*$ in the second sum of \eqref{rec_biparti_genre_univ_bis} and use $v(\ff,g) \geq \kappa |\ff|$.
\end{proof}

\begin{proof}[Proof of Lemma~\ref{lem_calcul_planar}]
Let $i^*>1$ be such that $\alpha_{i^*}>0$. Considering a map $m$ of $\mathcal{B}_{g^{(j)}}^{(p^j_1,p^j_2,\ldots,p^j_{\ell_j})}$, we can tesselate each boundary face of $m$ into faces of degree $2i^*$ as in the proof of Lemma~\ref{lem_petites_faces} to obtain a map $m'$. We need at most $p_i^j$ faces for a boundary of length $2p_i^j$. We then root $m'$ on the first boundary of $m$, and add a marked edge on each of the other boundaries. This operation is injective, so we obtain 
\begin{equation}\label{eqn_boundary_tesselation}
\beta_{g^{(j)}}^{(p^j_1,p^j_2,\ldots,p^j_{\ell_j})}(\mathbf{h}^{(j)})\leq |\mathbf{h}^{(j)}|^{\ell_j-1}\beta_{g^{(j)}}(\mathbf{\widehat{h}}^{(j)}),
\end{equation}
where $\mathbf{\widehat{h}}^{(j)}=\mathbf{h}^{(j)}+(\sum_{i=1}^{\ell_j} p_i^j)\mathbf{1}_{i^*}$. Using the crude bound $|\mathbf{h}^{(j)}| \leq |\ff^n|$, we can bound each term of the left-hand side of Lemma~\ref{lem_calcul_planar} as follows:
\begin{equation}\label{eqn_calcul_planar_1}
\prod_{j=1}^k \beta_{g^{(j)}}^{(p^j_1,p^j_2,\ldots,p^j_{\ell_j})}(\mathbf{h}^{(j)})\leq |\ff^n|^{\sum_{i=1}^k (\ell_j-1)} \prod_{j=1}^k \beta_{g^{(j)}}(\mathbf{\widehat{h}}^{(j)}).
\end{equation}
On the other hand, by the second item of Lemma~\ref{lem_couper_en_deux} and an easy induction on $k$, we have
\begin{equation}\label{eqn_calcul_planar_2}
\sum_{\substack{\mathbf{h}^{(1)}+\mathbf{h}^{(2)}+\ldots+\mathbf{h}^{(k)}=\ff^n-\mathbf{h}^{(0)} \\ g^{(1)}+g^{(2)}+\ldots+g^{(k)}=g_n-1-\sum_j(\ell_j-1)}}\prod_{j=1}^k \beta_{g^{(j)}}(\mathbf{\widehat{h}}^{(j)})\leq C' \beta_{g_n-g^*}(\mathbf{\widehat{f}}^n),
\end{equation}
where $\mathbf{\widehat{f}}^n =\mathbf{\widehat{h}}^{(1)}+\mathbf{\widehat{h}}^{(2)}+\ldots+\mathbf{\widehat{h}}^{(k)}$, the constant $C'$ does not depend on $n$, and $g^*=1+\sum_j(\ell_j-1)$. Note that $\mathbf{\widehat{f}}^n=\ff^n-\mathbf{h}^{(0)}+p\mathbf{1}_{i^*}$, with $p=\sum_{j=1}^k\sum_{i=1}^{\ell_j} p_i^j$. Therefore, we can write
\begin{equation}\label{eqn_calcul_planar_3}
\beta_{g_n-g^*}(\mathbf{\widehat{f}}^n) \leq \beta_{g_n-g^*}(\mathbf{f}^n+p\mathbf{1}_{i^*}) = O \left( \beta_{g_n-g^*}(\mathbf{f}^n) \right),
\end{equation}
where the last inequality comes from the Bounded ratio Lemma (Lemma~\ref{lem_BRL}), which applies since by our choice of $i^*$ we have $i^* \ff^n_{i^*} > \frac{\alpha_{i^*}}{2} |\ff^n|$ for $n$ large enough.

By combining \eqref{eqn_calcul_planar_1}, \eqref{eqn_calcul_planar_2} and \eqref{eqn_calcul_planar_3}, we have bounded the left-hand side of Lemma~\ref{lem_calcul_planar} by $O \left( |\ff^n|^{g^*-1} \beta_{g_n-g^*}(\mathbf{f}^n) \right)$. By the third item of Lemma~\ref{lem_couper_en_deux}, this is $O \left( |\ff^n|^{-g^*-1} \beta_{g_n}(\mathbf{f}^n) \right)$. Since $g^* \geq 0$, this concludes the proof.
\end{proof}

\begin{proof}[Proof of Lemma~\ref{lem_calcul_OE}]
The proof of the first part works in the exact same way as the proof of Lemma~\ref{lem_calcul_planar}, with the exception that now $g^*=\sum_j(\ell_j-1)$ (the $+1$ term was coming from $m$ having genus $1$ and is not there anymore), so the last bound becomes $O \left( |\ff^n|^{-g^*} \beta_{g_n}(\mathbf{f}^n) \right)$. Nevertheless, since the $\ell_j$'s are not all equal to $1$, we have $g^* \geq 1$ and the conclusion remains the same.

Now we prove the second point. As in the proof of Lemma~\ref{lem_calcul_planar}, let $i^*>1$ be such that $\alpha_{i^*}>1$. First, for the same reason as in \eqref{eqn_boundary_tesselation}, we can get rid of the boundaries by writing
\begin{equation}\label{eqn_boundary_tesselation_bis}
\beta_{g^{(j)}}^{p_j}(\mathbf{h}^{(j)})\leq \beta_{g^{(j)}}(\mathbf{\widehat{h}}^{(j)}),
\end{equation}
where $\mathbf{\widehat{h}}^{(j)}=\mathbf{h}^{(j)}+p_j\mathbf{1}_{i^*}$. 
Combining this with the first item of Lemma~\ref{lem_couper_en_deux} and an easy induction on $k$, we get
\[\sum_{\substack{\mathbf{h}^{(1)}+\mathbf{h}^{(2)}+\ldots+\mathbf{h}^{(k)}=\ff^n-\mathbf{h}^{(0)} \\ g^{(1)}+g^{(2)}+\ldots+g^{(k)}=g_n}}\, \prod_{j=1}^k |\mathbf{\widehat{h}}^{(j)}| \beta^{p_j}_{g^{(j)}}(\mathbf{\widehat{h}}^{(j)}) = O \left( |\mathbf{\widehat{f}}^n| \beta_{g_n}(\mathbf{\widehat{f}}^n) \right), \]
where $\mathbf{\widehat{f}}^n =\mathbf{\widehat{h}}^{(1)}+\mathbf{\widehat{h}}^{(2)}+\ldots+\mathbf{\widehat{h}}^{(k)}$. Now consider a term of the last summand such that both $|\mathbf{h}^{(1)}|$ and $|\mathbf{h}^{(2)}|$ are larger than $a$. Then the product $\prod_{j=1}^k |\mathbf{\widehat{h}}^{(j)}|$ has at least one factor larger than $\frac{|\ff^n|-|\mathbf{h}^{(0)}|}{k}$, and another one larger than $a$. Hence, it can be bounded from below by $\frac{\left( |\ff^n|-|\mathbf{h}^{(0)}| \right) a}{k}>\frac{|\ff^n| a}{2k}$ for $n$ large enough. Therefore, from the last display we get
\begin{equation}\label{eqn_calcul_OE_fin}
\sum_{\substack{\mathbf{h}^{(1)}+\mathbf{h}^{(2)}+\ldots+\mathbf{h}^{(k)}=\ff^n-\mathbf{h}^{(0)} \\ g^{(1)}+g^{(2)}+\ldots+g^{(k)}=g_n\\ |\mathbf{h}^{(1)}|, |\mathbf{h}^{(2)}| >a}} \, \prod_{j=1}^k \beta_{g}(\mathbf{\widehat{h}}^{(j)}) = O \left( \frac{1}{a} \beta_{g_n}(\mathbf{\widehat{f}}^n) \right), 
\end{equation}
where the $O$ is uniform in $a$. Finally, we have $\mathbf{\widehat{f}}^n=\mathbf{f}^n-\mathbf{h}^{(0)}+p\mathbf{1}_{i^*}$, where $p=\sum_{j=1}^k p_j$. Hence, using the Bounded ratio Lemma just like in the end of the proof of Lemma~\ref{lem_calcul_planar}, we have
\[\beta_{g_n}(\mathbf{\widehat{f}}^n)\leq \beta_{g_n}(\mathbf{f}^n+p\mathbf{1}_{i^*}) = O \left( \beta_{g_n} (\mathbf{f}^n) \right). \]
Combined with \eqref{eqn_calcul_OE_fin}, this proves the lemma.
\end{proof}

\subsection{Proof of the analyticity in Lemma~\ref{lem_degreefunction_basic}}
\label{subsec_analyticity}

We fix $(\alpha_j)_{j \geq 1}$ such that $\sum_{j \geq 1} j \alpha_j=1$ and recall that, for all $\omega \in [1,+\infty)$, the weight sequence $\q^{(\omega)}$ is given by \eqref{eqn_qjomega_univ}. More precisely, we have
\begin{equation}\label{eqn_def_q_with_g}
q_j^{(\omega)}=\frac{j \alpha_j}{\omega^{j-1} c(\omega)^{j-1} h_{j}(\omega)},
\end{equation}
where $h_{j}(\omega)=\sum_{s=0}^{j-1} \binom{2s}{s} (4 \omega)^{-s}$ and
\begin{equation}\label{eqn_defn_comega_complex}
c(\omega)=\frac{4}{1-\sum_{i \geq 1} \frac{1}{4^{i-1}} \binom{2i-1}{i-1} \frac{i \alpha_i}{\omega^{i-1} h_{i}(\omega)}}.
\end{equation}

If we denote by $\rho$ the root vertex of an infinite map $\MM_{\q}$, our goal is to prove that the function
\[ \omega \to \E \left[ \frac{1}{\deg_{\MM_{\q^{(\omega)}}}(\rho)} \right] \]
is analytic (and therefore continuous) on $(1,+\infty)$. We will show this by expressing the expectation as a sum over possible neighbourhoods of the root. The first step will be to prove that the probability to observe a fixed neighbourhood around the root is analytic. Our analyticity arguments will rely on extending certain functions to complex values of $\omega$, which we start doing in the next lemma. We note that the definition of $c(\omega)$ from \eqref{eqn_defn_comega_complex} still makes sense for $\omega$ complex provided the series in the denominator converges.

\begin{lem}\label{lem_comega_complex_neighbourhood}
There is a complex neighbourhood of $(1,+\infty)$ on which the function $\omega \to c(\omega)$ is well defined, analytic and on which $c(\omega) \ne 0$.
\end{lem}

\begin{proof}
To show that $c(\omega)$ is well defined and analytic, we will prove that for each $\omega_0 \in (1,+\infty)$, the infinite sum in the denominator of \eqref{eqn_defn_comega_complex} converges uniformly on a complex neighbourhood of $\omega_0$. Since each term of the sum is analytic (as an inverse of a polynomial), this will be enough to conclude.

We fix $\omega_0>1$ and write $\delta=\frac{1}{2}(\omega_0-1)$. We also fix an integer $s_0$ which will be specified in a few lines. For $\omega$ complex with $|\omega-\omega_0|<\delta$, we have $|\omega| \geq 1+\frac{\delta}{2}$. Hence, for $i \geq s_0$, we can write
\[ \left| \sum_{s=s_0+1}^{i-1} \binom{2s}{s} (4\omega)^{-s} \right| \leq \sum_{s>s_0} |\omega|^{-s} \leq \frac{\left( 1+\frac{\delta}{2} \right)^{-s_0}}{\delta/2}.\]
Therefore, we have, for $i \geq s_0$, we have:
\begin{equation}\label{eqn_minoration_homega}
\left| h_{i}(\omega) \right| \geq \left| \sum_{s=0}^{s_0} \binom{2s}{s} (4\omega)^{-s} \right| - \frac{2}{\delta} \left( 1+\frac{\delta}{2} \right)^{-s_0}.
\end{equation} 
We now fix the value of $s_0$: we choose $s_0>1$ large enough to have \[\frac{2}{\delta} \left( 1+\frac{\delta}{2} \right)^{-s_0} < \frac{1}{8\omega_0}.\] We now know that the function $\omega \to \sum_{s=0}^{\min(i,s_0)} \binom{2s}{s} (4\omega)^{-s}$ is continuous and is at least $1+\frac{1}{2\omega_0}$ for $\omega=\omega_0$ (it is a sum of nonnegative terms where the term $s=0$ is $1$ and the term $s=1$ is $\frac{1}{2\omega_0}$). Therefore, there is $0<\eps<\delta$ such that, for $\omega$ complex with $|\omega-\omega_0|<\eps$, we have
\[ \left| \sum_{s=0}^{s_0} \binom{2s}{s} (4\omega)^{-s} \right| > 1+\frac{1}{4 \omega_0}.\]
From now on, we consider $\omega$ in this complex ball. By \eqref{eqn_minoration_homega}, we have
\[ \left| h_{i}(\omega) \right| \geq 1+\frac{1}{8 \omega_0}, \]
for all $i>s_0$, so
\[\sup_{\omega \in B(\omega_0, \eps)} \left| \frac{1}{4^{i-1}} \binom{2i-1}{i-1} \frac{i \alpha_i}{\omega^{i-1} h_{i}(\omega)} \right| \leq \frac{1}{4^{i-1}} \binom{2i-1}{i-1} \frac{i \alpha_i}{\left( 1+\frac{\delta}{2} \right)^{i-1}\left( 1+\frac{1}{8 \omega_0} \right)}, \]
which is summable over $i$. Therefore, the infinite sum in the denominator of \eqref{eqn_defn_comega_complex} converges uniformly on $B(\omega_0, \eps)$, so $c(\omega)$ is well defined and analytic on a complex neighbourhood of $\omega_0$.

Moreover, we know that $c(\omega_0)>0$, so $c(\omega) \ne 0$ on a complex neighbourhood of $\omega_0$. Since this is true for all $\omega_0 \in (1,+\infty)$, this proves the lemma.
\end{proof}

Note that in the complex neighbourhood given by Lemma~\ref{lem_comega_complex_neighbourhood}, the "weights" $q_j^{(\omega)}$ given by~\eqref{eqn_def_q_with_g} also make sense for all $j \geq 1$ and are analytic in $\omega$.
To control the terms of the sum giving the expected inverse degree, we will need to make sure that the $q_j^{(\omega)}$ are not too large, i.e. that they are "admissible" even for $\omega$ complex.

\begin{lem}
Let $\omega_0>\omega_1>1$ be real. There is a complex neighbourhood of $\omega_0$ on which for all $j \geq 1$ we have $\left| q_j^{(\omega)} \right| \leq q_j^{(\omega_1)}$. In particular, on a complex neighbourhood of the line $\{ \omega>1 \}$, the sequence $\left( \left| q_j^{(\omega)} \right| \right)_{j \geq 1}$ is admissible and subcritical.
\end{lem}

\begin{proof}
First, the last part of the lemma is an immediate consequence of the first one, since a weight sequence dominated term by term by an admissible and subcritical weight sequence is also admissible and subcritical. We now prove the first part.

For $j=1$, we have $q_1^{(\omega)}=\alpha_1$ for all $\omega$, so $q_1^{(\omega)}=q_1^{(\omega_1)}$ in particular. We now find a complex neighbourhood of $\omega_0$ on which $\left| q_j^{(\omega)} \right| \leq q_j^{(\omega_1)}$ for all $j \geq 2$. Let $\delta>0$ (to be specified later). By the same argument as in the proof of Lemma~\ref{lem_comega_complex_neighbourhood}, there is a complex neighbourhood of $\omega_0$ on which
\[ \left| h_{j}(\omega) \right| \geq (1-\delta) \left| h_{j}(\omega_0) \right|\]
for all $j \geq 1$. Moreover, we have $\omega_0 c(\omega_0)>0$, so by continuity there is a complex neighbourhood of $\omega_0$ on which
\[ \left| \omega c(\omega) \right| \geq (1-\delta) \left| \omega_0 c(\omega_0) \right|.\]
By combining the last two equations and \eqref{eqn_def_q_with_g}, we get, for $\omega$ in a complex neighbourhood of $\omega_0$,
\begin{equation}\label{eqn_bounding_q_complex}
\left| q_j^{(\omega)} \right| \leq \frac{1}{(1-\delta)^j} q_j^{(\omega_0)} = \frac{1}{1-\delta} \left( \frac{\omega_1 c(\omega_1)}{(1-\delta) \omega_0 c(\omega_0)} \right)^{j-1} \frac{h_{j}(\omega_1)}{h_{j}(\omega_0)} \times q_j^{(\omega_1)}
\end{equation}
for all $j \geq 2$.

We now claim that $\omega_0 c(\omega_0)>\omega_1 c(\omega_1)$. Indeed, let $j \geq 2$ be such that $\alpha_j>0$. We know from Lemma \ref{lem_qj_is_monotone} that $q_j^{(\omega_0)}<q_j^{(\omega_1)}$. On the other hand $h_{j}(\omega_0)<h_{j}(\omega_1)$ since $h_j(\omega)$ is a polynomial with positive coefficients in $\frac{1}{\omega}$. Hence, the claim that $\omega_0 c(\omega_0)>\omega_1 c(\omega_1)$ follows from~\eqref{eqn_def_q_with_g}.

Therefore, if $\delta$ is chosen small enough, we have $\frac{\omega_1 c(\omega_1)}{(1-\delta) \omega_0 c(\omega_0)}<1$. Since $\frac{h_{j}(\omega_1)}{h_{j}(\omega_0)} \leq \frac{h_j(1)}{1} \leq C \sqrt{j}$ for some absolute constant $C$, we conclude by \eqref{eqn_bounding_q_complex} that, on a complex neighbourhood of $\omega_0$, we have
\begin{equation}\label{eqn_bounding_q_complex_2}
\left| q_j^{(\omega)} \right| < q_j^{(\omega_1)}
\end{equation}
for all $j \geq j_0$, where $j_0$ may depend on $\omega_0$ and $\omega_1$ but not on $\omega$. Moreover, for $2 \leq j \leq j_0$, we know by Lemma \ref{lem_qj_is_monotone} that $q_j^{(\omega_0)}<q_j^{(\omega_1)}$ and that $q_j^{(\omega)}$ is continuous in $\omega$, so up to shrinking our neighbourhood of $\omega_0$, \eqref{eqn_bounding_q_complex_2} holds for all $j \geq 2$. This proves the lemma.
\end{proof}
 
We now move on to the proof of analyticity. For this, we will write the expected inverse degree of the root as an infinite sum over all possible values of the "ball" of radius $1$. We first precise the notion of ball that we will use\footnote{This notion is actually closer to the \emph{hull} of a map. Since it is not useful here, we will not compare it to other notions of hull introduced in the literature.}. We consider a peeling algorithm $\mathcal{A}_{1}$ such that, if the root vertex $\rho$ is on $\partial m$, then $\mathcal{A}_{1}(m)$ is the edge on $\partial m$ on the right of $\rho$. If $M$ is a map, we perform a filled-in peeling exploration of $M$ using the algorithm $\mathcal{A}_{1}$, and denote it by $\left( \expl_t^{\mathcal{A}_1}(M) \right)_{t \geq 0}$. We stop the exploration at the first time $\tau$ where $\rho \notin \partial \expl_t^{\mathcal{A}_1}(M)$, which is finite almost surely if $\rho$ has finite degree. We denote by $B_1^{\bullet}(M)$ the explored map $\expl_{\tau}^{\mathcal{A}_1}(M)$. Finally, we denote by $\mathcal{H}$ the set of possible values of $B_1^{\bullet}(M)$, where $M$ is an infinite, one-ended planar map with finite vertex degrees. Note that $\mathcal{H}$ is an infinite set of finite planar maps with one hole.

We can now write, for $\omega>1$:
\begin{equation}\label{eqn_meandegree_to_sum}
\E  \left[ \frac{1}{\deg_{\MM_{\q^{(\omega)}}}(\rho)} \right] = \sum_{m \in \mathcal{H}} \frac{1}{\deg_{m}(\rho)} p_m^{(\omega)},
\end{equation}
where $p_m^{(\omega)} = \P \left( B_1^{\bullet} \left( \MM_{\q^{(\omega)}} \right)=m \right)$. We have
\[ p_m^{(\omega)} = \P \left( m \subset \MM_{\q^{(\omega)}} \right) = \left( \omega c(\omega) \right)^{|\partial m|-1} h_{|\partial m|}(\omega) \prod_{f \in m} q^{(\omega)}_{|f|},\]
where $|\partial m|$ is the half-perimeter of the hole of $m$ and $|f|=\frac{\deg (f)}{2}$, and the product is over internal faces $f$ of $m$. Since $q_j^{(\omega)}$ is well-defined and analytic in $\omega$ on a complex neighbourhood of $(1,+\infty)$, the right-hand side still makes sense for $\omega$ complex, so there is a complex neighbourhood of $(1,+\infty)$ on which $p_m^{(\omega)}$ makes sense for all $m$.

We now fix $\omega_0>1$. To prove analyticity in a neighbourhood of $\omega_0$, it is enough to prove that the sum in \eqref{eqn_meandegree_to_sum} converges uniformly in $\omega$ on a neighbourhood of $\omega_0$. Therefore, it is sufficient to find a complex neighbourhood $\mathcal{N}$ of $\omega_0$ on which
\begin{equation}\label{eqn_normal_convergence}
\sum_{m \in \mathcal{H}} \sup_{\omega \in \mathcal{N}} \left| p_m^{(\omega)} \right| <+\infty.
\end{equation}

For this, in the computation of $p_m^{(\omega)}$, we first replace $q_j^{(\omega)}$ using \eqref{eqn_def_q_with_g} to obtain
\[  \left| p_m^{(\omega)} \right| = \left| \omega c(\omega) \right|^{|\partial m|-1-\sum_{f \in m} (|f|-1)} \times \left| h_{|\partial m|}(\omega) \right| \times \prod_{f \in m} \frac{|f| \alpha_{|f|}}{\left| h_{|f|}(\omega) \right|}.\]
The Euler formula shows that $|\partial m|-1-\sum_{f \in m} (|f|-1)=-\mathrm{Inn}(m)$, where $\mathrm{Inn}(m) \geq 0$ is the number of internal vertices of $m$. We now fix $\omega_1, \omega_2$ real with $1<\omega_1<\omega_0<\omega_2$. By using on the one hand the fact that $\omega \to \omega c(\omega)$ is increasing on $(1,+\infty)$ (see the discussion right after \eqref{eqn_bounding_q_complex}), and on the other hand the same argument as in the proof of Lemma~\ref{lem_comega_complex_neighbourhood}, there is a complex neighbourhood $\mathcal{N}_0$ of $\omega_0$ on which
\[ \left| \omega c(\omega) \right| \geq \omega_1 c(\omega_1) \quad \hspace{-0.2cm}\mbox{ and } \hspace{-0.2cm}\quad \forall j \geq 1, \, h_j(\omega_2) \leq \left| h_j(\omega) \right| \leq h_j(\omega_1) \leq \frac{1}{\sqrt{1-\omega_1^{-1}}}\]
where the rightmost inequality comes from the Taylor expansion of $\frac{1}{\sqrt{1-\omega_1^{-1}}}$.
For $\omega \in \mathcal{N}_0$, we have
\[ \left| p_m^{(\omega)} \right| \leq \frac{1}{\sqrt{1-\omega_1^{-1}}} \left( \omega_1 c(\omega_1) \right)^{-\mathrm{Inn}(m)} \prod_{f \in m} \frac{|f| \alpha_{|f|}}{h_{|f|}(\omega_2)}. \]
On the other hand, for $m \in \mathcal{H}$, we have
\[ \P \left( B_1^{\bullet} \left( \MM_{\q^{(\omega_2)}} \right)=m \right) = p_m^{(\omega_2)} = h_{|\partial m|}(\omega_2) \left( \omega_2 c(\omega_2) \right)^{-\mathrm{Inn}(m)} \prod_{f \in m} \frac{|f| \alpha_{|f|}}{h_{|f|}(\omega_2)},  \]
with $h_{|\partial m|}(\omega_2) \geq 1$. It follows that
\[ \sup_{\omega \in \mathcal{N}_0} |p_m^{(\omega)}| \leq \frac{1}{\sqrt{1-\omega_1^{-1}}} \left( \frac{\omega_2 c(\omega_2)}{\omega_1 c(\omega_1)} \right)^{\mathrm{Inn}(m)} \P \left( m \subset \MM_{\q^{(\omega_2)}} \right), \]
so
\begin{equation}\label{eqn_bound_normal_convergence}
\sum_{m \in \mathcal{H}} \sup_{\omega \in \mathcal{N}_0} |p_m^{(\omega)}| \leq \frac{1}{\sqrt{1-\omega_1^{-1}}}  \E \left[ \left( \frac{\omega_2 c(\omega_2)}{\omega_1 c(\omega_1)} \right)^{\mathrm{Inn}\left( B_1^{\bullet} \left( \MM_{\q^{(\omega_2)}} \right) \right)}\right].
\end{equation}
Therefore, to prove \eqref{eqn_normal_convergence} and therefore analyticity, it is sufficient to prove that for any $\omega_0>1$, we can find $1<\omega_1<\omega_0<\omega_2<+\infty$ such that the last expectation is finite. This will follow from the next lemma.

\begin{lem}\label{lem_unif_exp_tail}
Let $\omega_0>1$. We can find $1<\omega_3<\omega_0<\omega_4<+\infty$ and a constant $y_0>1$ such that, for all $\omega \in [\omega_3, \omega_4]$, we have
\[ \E \left[ y_0^{\mathrm{Inn} \left( B_1^{\bullet} \left( \MM_{\q^{(\omega)}}\right) \right)} \right]<+\infty. \]
\end{lem}

First, let us explain how Lemma \ref{lem_unif_exp_tail} proves the finiteness of the right-hand side of \eqref{eqn_bound_normal_convergence}. We fix $\omega_3, \omega_4$ as given by Lemma~\ref{lem_unif_exp_tail}. Since $\omega \to \omega g(\omega)$ is continuous at $\omega_0$, there are $\omega_1 \in (\omega_3, \omega_0)$, $\omega_2 \in (\omega_0, \omega_4)$ and $y_0>1$ such that
\[ \frac{\omega_2 g(\omega_2)}{\omega_1 g(\omega_1)} < y_0. \]
Such $\omega_1$ and $\omega_2$ suit our needs, so all we have left to do is to prove Lemma~\ref{lem_unif_exp_tail}.

\begin{proof}[Proof of Lemma \ref{lem_unif_exp_tail}]
We want to track down the number of inner vertices created at each step of a peeling exploration. The sketch of the proof is the following:
\begin{itemize}
\item
we first argue that by an absolute continuity argument, the number of internal vertices created by a peeling exploration can be replaced by a random walk with known step distribution;
\item
we then bound the number of internal vertices created by one step of this random walk;
\item
we finally estimate the number of peeling steps needed to complete the exploration of $B_1^{\bullet}$.
\end{itemize}

From now on, the values of $\omega$ that we will consider will be real and larger than $1$.
For $t \geq 0$, we denote respectively by $P_t^{(\omega)}$ and $I_t^{(\omega)}$ the half-perimeter and the number of inner vertices of $\expl_t^{\mathcal{A}_1} \left( \MM_{\q^{(\omega)}}\right)$. We also write $\tau^{(\omega)}$ for the first peeling step at which the root vertex of $\MM_{\q^{(\omega)}}$ disappears from the boundary, and note that $\mathrm{Inn} \left( B_1^{\bullet} \left( \MM_{\q^{(\omega)}} \right) \right) = I^{(\omega)}_{\tau^{(\omega)}}$.

We recall that $P^{(\omega)}$ has the law of a random walk $\widetilde{P}^{(\omega)}$ with step distribution $\widetilde{\nu}_{\q^{(\omega)}}$, started from $1$ and conditioned on the positive probability event $\{ \forall t \geq 0, \widetilde{P}^{(\omega)}_t \geq 1\}$. Moreover, conditionally on $P^{(\omega)}$, the law of $I^{(\omega)}$ can be described as follows:
\begin{itemize}
\item
the increments $I^{(\omega)}_{t+1}-I_t^{(\omega)}$ are independent;
\item
if $P_{t+1}^{(\omega)}-P_t^{(\omega)} \geq 0$, then $I^{(\omega)}_{t+1}-I_t^{(\omega)}=0$;
\item
if $P_{t+1}^{(\omega)}-P_t^{(\omega)} \leq -1$, then $I^{(\omega)}_{t+1}-I_t^{(\omega)}=0$ has the law of the \emph{total} number of vertices of a $\q^{(\omega)}$-Boltzmann map of the $2 \left( -1-P_{t+1}^{(\omega)}+P_t^{(\omega)} \right)$-gon (with the convention that the unique map of the $0$-gon has a single vertex and no internal face).
\end{itemize}
Let $\widetilde{I}^{(\omega)}$ be the process obtained from $\widetilde{P}^{(\omega)}$ in the exact same way as $I^{(\omega)}$ is obtained from $P^{(\omega)}$. Then $\left( P^{(\omega)}, I^{(\omega)} \right)$ has the law of $\left( \widetilde{P}^{(\omega)}, \widetilde{I}^{(\omega)} \right)$ conditioned on $\{ \forall t \geq 0, \widetilde{P}^{(\omega)}_t \geq 1\}$. Moreover, by the choice of our peeling algorithm $\mathcal{A}_1$, we know that $\tau^{(\omega)}$ is the first time at which the perimeter decreases and the swallowed part is on the left of the peeled edge. Therefore, it is the $G$-th time where the perimeter decreases, where $G$ is a geometric variable (starting at $1$) with parameter $\frac{1}{2}$, independent from $P^{(\omega)}$ and $I^{(\omega)}$. Similarly, let $\widetilde{\tau}^{(\omega)}$ be the $G$-the time where $\widetilde{P}^{(\omega)}$ decreases. Then $\left( P^{(\omega)}, I^{(\omega)}, \tau^{\omega} \right)$ has the law of $\left( \widetilde{P}^{(\omega)},  \widetilde{I}^{(\omega)},  \widetilde{\tau}^{\omega} \right)$ conditioned\footnote{The triplet $\left( \widetilde{P}^{(\omega)},  \widetilde{I}^{(\omega)},  \widetilde{\tau}^{\omega} \right)$ can be interpreted in terms of a peeling exploration of the half-plane analog of $\MM_{\q^{(\omega)}}$, but this will not be necessary for us.} on $\{ \forall t \geq 0, \widetilde{P}^{(\omega)}_t \geq 1\}$. Therefore, for all $y\geq 0$, we have
\[ \E \left[ y^{\mathrm{Inn} \left( B_1^{\bullet} \left( \MM_{\q^{(\omega)}}\right) \right)} \right] \leq \frac{1}{\P \left( \forall t \geq 0, \widetilde{P}^{(\omega)}_t \geq 1 \right)} \E \left[ y^{\widetilde{I}^{(\omega)}_{\widetilde{\tau}^{(\omega)}}} \right].\]
In particular, since the conditioning is nondegenerate, it is enough to find $y_0>1$ such that, for $\omega$ in a real neighbourhood of $\omega_0$, we have 
\begin{equation}\label{eqn_inner_exp_tail_goal}
\E \left[ y_0^{\widetilde{I}^{(\omega)}_{\widetilde{\tau}^{(\omega)}}} \right]<+\infty.
\end{equation}

Let $\widehat{I}^{(\omega)}$ have the distribution of $\widetilde{I}^{(\omega)}_1$, conditionned on $\{ \widetilde{I}^{(\omega)}_1 \geq 1 \}$ (i.e. $\widehat{I}$ is the law of the number of internal vertices created by a peeling step, conditioned to be nonzero). We recall that $\widetilde{\tau}{(\omega)}$ is the $G$-th time where the walk $\widetilde{I}^{(\omega)}$ increases, where $G$ is independent from $\widetilde{I}^{(\omega)}$ and $\P \left( G=i \right)=\frac{1}{2^{i+1}}$ for all $i \geq 1$. It follows that $\widetilde{I}^{(\omega)}_{\widetilde{\tau}^{(\omega)}}$ has the law of the sum of $G$ i.i.d. copies of $\widehat{I}^{(\omega)}$. Therefore, we need to get bounds on the variable $\widehat{I}^{(\omega)}$. For that, we will first get bounds on the law of $\widetilde{I}^{(\omega)}_1$.

More precisely, by summing all possible peeling cases, for all $y \geq 1$ and $\omega \in (1,+\infty)$,  we have
\begin{equation}\label{eqn_x_power_v}
\E \left[ y^{\widetilde{I}_1^{(\omega)}}-1 \right] = 2 \sum_{i \geq 0} \left( \frac{1}{\omega c(\omega)} \right)^{i+1} \sum_{|\partial m|=i} \left( y^{\#\mathrm{Vertices}(m)}-1 \right) \prod_{f \in m} q^{(\omega)}_{|f|},
\end{equation}
where the sum is over all finite maps of the $2i$-gon, the product over inner faces and $|f|=\deg(f)/2$. We have computed the expectation of $y^{\widetilde{I}_1^{(\omega)}}-1$ instead of $y^{\widetilde{I}_1^{(\omega)}}$ because $y^{\widetilde{I}_1^{(\omega)}}-1$ vanishes in the "peeling cases" where $\widetilde{P}_1^{(\omega)} \geq 0$ (i.e. when a peeling case discovers a new face, no internal vertex is created). We first fix $\omega' \in (1,\omega_0)$ and prove that there is $y$ (that may depend on $\omega'$) such that $\E \left[ y^{\widetilde{I}^{(\omega')}} \right]<+\infty$. Indeed, using the Euler formula, if $m$ is a finite map of the $2i$-gon, the total number of vertices of $m$ is equal to $i+1+\sum_f \left( |f|-1 \right)$, where the sum is over internal faces. We recall from Section~\ref{subsec_prelim_combi} that $W_i(\q)$ is the partition function of $\q$-Boltzmann maps of the $2i$-gon and $W_i^{\bullet}(\q)$ is the corresponding pointed partition function (i.e. biased by the number of vertices). Bounding $y^{\#\mathrm{Vertices}(m)}-1$ by $y^{\#\mathrm{Vertices}(m)}$, we can rewrite \eqref{eqn_x_power_v} as
\begin{equation}\label{eqn_q_omega_x}
\begin{split}
\E \left[ y^{\widetilde{I}_1^{(\omega')}} \right] &\leq 1+2 \sum_{i \geq 0}  \left( \frac{y}{\omega' c(\omega')} \right)^{i+1} W_i \left( \q^{(\omega_5, y)} \right)\\& \leq 1+2 \sum_{i \geq 0}  \left( \frac{y}{\omega' c(\omega')} \right)^{i+1} W_i^{\bullet} \left( \q^{(\omega_5, y)} \right),
\end{split}
\end{equation}
where $\q^{(\omega',y)}$ is the weight sequence given by $q_j^{(\omega',y)}=q_j^{(\omega')} y^{j-1}$. We first claim that if $y>1$ is small enough, then $\q^{(\omega',y)}$ is still admissible. Indeed, using the notation of Section~\ref{subsec_prelim_combi}, we have $f_{\q^{(\omega', y)}}(x)=f_{\q^{(\omega')}}(yx)$, so $\q^{(\omega',y)}$ is admissible if and only if the equation
\begin{equation}\label{eqn_admissibility_tilted}
f_{\q^{(\omega')}}(yx)=1-\frac{1}{x}
\end{equation}
has a solution. We already know that $\q^{(\omega')}$ is admissible, and the smallest solution to~\eqref{eqn_admissibility_tilted} is $\frac{1}{4}c_{\q^{(\omega')}}=\frac{1}{4}c(\omega')$ for $y=1$. Moreover, by definition of the walk $\widetilde{\nu}_{\q^{(\omega')}}$, we have
\[ \sum_{j \geq 1} q_j^{(\omega')} \left( \omega' c(\omega') \right)^{j-1} = \sum_{i \geq 0} \widetilde{\nu}_{\q^{(\omega')}}(i) \leq 1 <+\infty,\]
which proves that the radius of convergence of $\q^{(\omega')}$ is at least $\omega' c(\omega')$, so the radius of convergence of $f_{\q^{(\omega')}}$ is at least $\frac{1}{4} \omega' c(\omega')>\frac{1}{4} c(\omega')$. Now let $x \in \left( \frac{1}{4} c(\omega') , \frac{1}{4} \omega' c(\omega') \right)$. By convexity of $f_{\q^{(\omega')}}$ and concavity of $x \to 1-\frac{1}{x}$, we have
\[ 1-\frac{1}{x} < f_{\q^{(\omega')}}(x) <+\infty. \]
It follows that for $y>1$ small enough
\[ 1-\frac{1}{x/y} < f_{\q^{(\omega')}} \left( y \times (x/y) \right) <+\infty. \]
On the other hand, the inequality is reversed if we replace $x/y$ by $1$. Hence, by the intermediate value theorem~\eqref{eqn_admissibility_tilted} has a solution, so $\q^{(\omega', y)}$ is admissible for $y>1$ small enough. Moreover $y \to c_{\q^{(\omega', y)}}$ is continuous in a neighbourhood of $y=1$.
On the other hand, using the formula~\eqref{eqn_exact_pointed_partition_function}, the right-hand side of~\eqref{eqn_q_omega_x} is finite if and only if $\frac{y}{\omega' c(\omega')} \times c_{\q^{(\omega', y)}}<1$.
This is true for $y=1$, so by continuity in $y$ this is also true for some $y>1$. We have therefore proved that for some fixed $\omega' \in (1,\omega_0)$, there is $y'>1$ such that $\E \left[ \left( y' \right)^{\widetilde{I}_1^{(\omega')}} \right]<+\infty$.

Let us now come back to~\eqref{eqn_x_power_v}. By Lemma~\ref{lem_qj_is_monotone} and the fact that $\omega \to \omega c(\omega)$ is increasing (see the discussion right after~\eqref{eqn_bounding_q_complex}), each term in the right-hand side is nonincreasing in $\omega$ and nondecreasing in $y$. By the dominated convergence theorem, it follows that~\eqref{eqn_x_power_v} is continuous in $(\omega, y)$ for $\omega \in [\omega', +\infty)$ and $y \in [1,y']$. Since the left-hand side is equal to $0$ for $y=1$, there is $y_0 \in (1,y')$ and a real neighbourhood $\mathcal{N}_1$ of $\omega_0$ such that, for $\omega \in \mathcal{N}_1$, we have
\begin{equation}\label{eqn_itilde_exp_tail}
\E \left[ y_0^{\widetilde{I}_1^{(\omega)}}-1 \right] \leq \frac{1}{2 \omega_0 c(\omega_0)}.
\end{equation}
On the other hand, we recall that $\widehat{I}^{(\omega)}$ has the law of $\widetilde{I}_1^{(\omega)}$ conditionned on the event $\left\{ \widetilde{P}^{(\omega)}_1 \leq -1 \right\}$, which has probability $\widetilde{\nu}_{\q^{(\omega)}} \left( (-\infty,-1] \right)$. Therefore, for all $\omega \in [1,+\infty)$, we can write
\[ \E \left[ y_0^{\widehat{I}^{(\omega)}}-1 \right] \leq \frac{\E \left[ y_0^{\widetilde{I}_1^{(\omega)}}-1 \right]}{\widetilde{\nu}_{\q^{(\omega)}} \left( (-\infty, -1] \right)} \leq \frac{\E \left[ y_0^{\widetilde{I}_1^{(\omega)}}-1 \right]}{\widetilde{\nu}_{\q^{(\omega)}} (-1)} = \frac{\omega c(\omega)}{2} \E \left[ y_0^{\widetilde{I}_1^{(\omega)}}-1 \right],  \]
using the definition of $\widetilde{\nu}_{\q}$. Therefore, using~\eqref{eqn_itilde_exp_tail}, for $\omega$ in the neighbourhood $\mathcal{N}_1$ of $\omega_0$, we have
\[ \E \left[ y_0^{\widehat{I}^{(\omega)}} \right] \leq 1+\frac{1}{4} \frac{\omega c(\omega)}{\omega_0 c(\omega_0)}.\] In particular, by continuity of $c(\omega)$, there is a neighbourhood $\mathcal{N}_2$ of $\omega_0$ such that, for $\omega \in \mathcal{N}_2$, we have $\E \left[ y_0^{\widehat{I}^{(\omega)}} \right] \leq \frac{3}{2}$.

Finally, let $\left( \widehat{I}^{(\omega)}_k \right)_{k \geq 1}$ be i.i.d. copies of $\widehat{I}^{(\omega)}$. We recall that the quantity that we are interested in has the law of
\[ \widehat{I}^{(\omega)}_1+\widehat{I}^{(\omega)}_2+\dots+\widehat{I}^{(\omega)}_G,\]
where $G$ is geometric with parameter $\frac{1}{2}$ and independent from the $\widehat{I}^{(\omega)}_k$. For $\omega \in \mathcal{N}_2$, we have
\[ \E \left[ y_0^{\widehat{I}^{(\omega)}_1+\dots+\widehat{I}^{(\omega)}_G} \right] = \sum_{k \geq 1} \frac{1}{2^{k+1}} \E \left[ y_0^{\widehat{I}^{(\omega)}} \right]^k \leq \sum_{k \geq 1} \frac{1}{2^{k+1}} \left( \frac{3}{2} \right)^k <+\infty.\]
This proves~\eqref{eqn_inner_exp_tail_goal}, and therefore the lemma.
\end{proof}

\subsection{Uniform convergence of the volume process}\label{subsec_unif_volume}

The goal of this appendix is to prove Lemma~\ref{lem_unif_volume}, which states that the approximation of $\omega$ given by Proposition~\ref{prop_q_as_limit} is good uniformly in $\omega$. If $(X_t^{(\omega)})_{t \geq 1, \omega \in \Omega}$ and $(X_{\infty}^{(\omega)})_{\omega \in \Omega}$ are random variables, we say that \emph{$X_t^{(\omega)}$ converges in probability to $X_{\infty}^{(\omega)}$ uniformly in $\omega$ over $\Omega$} if for all $\eps>0$, there is $t_0>0$ such that, for all $t \geq t_0$ and $\omega \in \Omega$:
\[ \P \left( \left| X_t^{(\omega)}-X_{\infty}^{(\omega)} \right|>\eps \right)<\eps. \]

To prove Lemma~\ref{lem_unif_volume}, we will adapt the proof of~\eqref{limit_volume_growth} from~\cite[Prop 10.12]{C-StFlour}. For this, we will need the uniform weak law of large numbers below. This result is probably not new but we could not locate it in the literature. We give a proof for the sake of completeness.

\begin{lem}\label{lem_unif_WLLN}
Let $\left( X^{(\omega)} \right)_{\omega \in \Omega}$ be a uniformly integrable family of random variables on $\R^+$. For each $\omega \in \Omega$, let $\left( X_t^{(\omega)} \right)_{t \geq 1}$ be i.i.d. copies of $X^{(\omega)}$. Then the convergence in probability
\[ \overline{X}^{(\omega)}_t := \frac{1}{t} \sum_{i=1}^t X_i^{(\omega)} \xrightarrow[t \to +\infty]{} \E \left[ X^{(\omega)} \right]\]
is uniform in $\omega \in \Omega$.
\end{lem}

\begin{proof}
We adapt one of the classical proofs of the weak law of large numbers: we truncate the variables at some cutoff $A$ and use the second moment method on the truncated variables.

More precisely, for $A>0$, we write $X^{(\omega, \leq A)}_t=X^{(\omega)}_t \mathbbm{1}_{X^{(\omega)}_t \leq A}$ and $X^{(\omega, >A)}_t=X^{(\omega)}_t \mathbbm{1}_{X^{(\omega)}_t > A}$. We also write
\[ \overline{X}^{(\omega, \leq A)}_t := \frac{1}{t} \sum_{i=1}^t X_i^{(\omega), \leq A} \quad \mbox{ and } \quad \overline{X}^{(\omega, > A)}_t := \frac{1}{t} \sum_{i=1}^t X_i^{(\omega), > A}.\]

Fix $\eps>0$. By uniform integrability, let $A$ be such that $\E \left[ X^{(\omega, >A)} \right] \leq \frac{\eps^2}{8}$ for all $\omega \in \Omega$. By a variance computation, we have
\[ \P \left( \left| \overline{X}^{(\omega, \leq A)}_t - \E \left[ X^{(\omega, \leq A)} \right] \right| \geq \frac{\eps}{4} \right) \leq \frac{4A^2}{\eps t}. \]
We also have
\[ \left| \E \left[ X^{(\omega, \leq A)} \right] - \E \left[ X^{(\omega)}\right] \right| =\E \left[ X^{(\omega, >A)} \right] \leq \frac{\eps^2}{8}<\frac{\eps}{2}.\]
Finally, by the Markov inequality
\[ \P \left( \left| \overline{X}^{(\omega, \leq A)}_t-\overline{X}^{(\omega)}_t \right| \geq \frac{\eps}{4} \right) \leq \frac{4}{\eps} \E \left[ \overline{X}_t^{(\omega, >A)} \right] = \frac{4}{\eps} \E \left[ X^{(\omega, >A)} \right] \leq \frac{\eps}{2}.\]
Combining the last three displays, we get $\P \left( \left| \overline{X}^{(\omega)}_t - \E \left[ X^{(\omega)}\right] \right| \geq \eps \right) \leq \frac{\eps}{2} + \frac{4A^2}{\eps t}$ for all $t \geq 1$ and $\omega \in \Omega$, and the lemma follows by taking $t \geq \frac{8A^2}{\eps^2}$.
\end{proof}

We can now prove Lemma~\ref{lem_unif_volume}. We recall that we have fixed $(\alpha_j)_{j \geq 1}$ with $\sum_{j \geq 1} j \alpha_j=1$ and $\alpha_1<1$ and that, for all $\omega \geq 1$, we denote by $\q^{(\omega)}$ the unique weight sequence such that $\omega_{\q}=\omega$ and $a_j(\q)=\alpha_j$ for all $j \geq 1$. We finally recall that $\left( P_t^{(\omega)} \right)_{t \geq 0}$ and $\left( V_t^{(\omega)}  \right)_{t \geq 0}$ are respectively the perimeter and volume processes associated to a peeling exploration of $\MM_{\q^{(\omega)}}$.

\begin{proof}[Proof of Lemma \ref{lem_unif_volume}]
We know that $\left( P_t^{(\omega)}, V_t^{(\omega)} \right)_{t \geq 0}$ has the law of a random walk $\left( \widetilde{P}_t^{(\omega)}, \widetilde{V}_t^{(\omega)} \right)_{t \geq 0}$ on $\Z^2$ started from $1$ and conditionned on the event that $\widetilde{P}_t^{(\omega)}\geq 1$ for all $t \geq 0$. Moreover, the step distribution of $\widetilde{P}^{(\omega)}$ is $\widetilde{\nu}_{\q^{(\omega)}}$. Since the function $\left( \mathbbm{1}_{p \geq 1} h_p(\omega) \right)_{p \geq 1}$ is harmonic for $\widetilde{P}^{(\omega)}$ on $\{1,2,\dots\}$, by a simple martingale argument we have
\[ \P \left( \forall t \geq 0, \widetilde{P}^{(\omega)}_t \geq 1 | \widetilde{P}^{(\omega)}_0=1 \right) = \frac{h_1(\omega)}{\lim_{p \to +\infty} h_p(\omega)} = \sqrt{\frac{\omega-1}{\omega}} \]
by \eqref{eqn_defn_homega}. Since $\omega$ lies in a compact subset of $(1,+\infty)$, this probability is bounded away from $0$. Therefore, it is enough to prove
\[ \frac{\widetilde{V}^{(\omega)}_t-2\widetilde{P}^{(\omega)}_t}{t} \xrightarrow[t \to +\infty]{} \frac{ \left(\sqrt{\omega}-\sqrt{\omega-1} \right)^2}{2 \sqrt{\omega(\omega-1)}} \]
in probability, uniformly in $\omega$. Note that the computation $\E \left[ \widetilde{V}^{(\omega)}_1-2\widetilde{P}^{(\omega)}_1 \right]=\frac{ \left(\sqrt{\omega}-\sqrt{\omega-1} \right)^2}{2 \sqrt{\omega(\omega-1)}}$ is given by Proposition 10.12 in~\cite{C-StFlour}, so the convergence holds for $\omega$ fixed.

To prove that it is uniform, by the uniform weak law of large number (Lemma~\ref{lem_unif_WLLN}), it is sufficient to prove that $\widetilde{V}^{(\omega)}_1-2\widetilde{P}^{(\omega)}_1$ is uniformly integrable in $\omega$ for $\omega$ in a compact subset of $(1,+\infty)$. This can be deduced from continuity both in distribution and for the expectation. More precisely, assume that this is not the case. Then we can find a sequence $(\omega_i)_{i \geq 1}$ such that $\left( \widetilde{V}^{(\omega_i)}_1-2\widetilde{P}^{(\omega_i)}_1\right)_{i \geq 1}$ is not uniformly integrable. By compactness, up to extraction we can assume $\omega_i \to \omega_{\infty} \in (1,+\infty)$. Then we claim that
\begin{equation}\label{eqn_v1_continuous_in_law}
\widetilde{V}^{(\omega_i)}_1-2\widetilde{P}^{(\omega_i)}_1 \xrightarrow[i \to +\infty]{(d)} \widetilde{V}^{(\omega_{\infty})}_1-2\widetilde{P}^{(\omega_{\infty})}_1.
\end{equation}
Indeed, $\widetilde{V}^{(\omega)}_1-2\widetilde{P}^{(\omega)}_1$ can be interpreted as the number of inner edges created by one peeling step in the half-plane version of $\MM_{\q^{(\omega)}}$. As proved in the beginning of Appendix~\ref{subsec_analyticity}, each $q_j^{(\omega)}$ is a continuous function of $\omega$, so the probability of each peeling case is a continuous function of $\omega$, so for all $k$ the probability $\P \left( \widetilde{V}^{(\omega)}_1-2\widetilde{P}^{(\omega)}_1 =k \right)$ is continuous in $\omega$, which proves~\eqref{eqn_v1_continuous_in_law}.

On the other hand, we know from the computation above that $\E \left[ \widetilde{V}^{(\omega)}_1-2\widetilde{P}^{(\omega)}_1 \right]$ is a continuous function of $\omega$, so $\E \left[ \widetilde{V}^{(\omega_i)}_1-2\widetilde{P}^{(\omega_i)}_1 \right] \to \E \left[ \widetilde{V}^{(\omega_{\infty})}_1-2\widetilde{P}^{(\omega_{\infty})}_1 \right]$ as $i \to +\infty$. By the Skorokhod embedding theorem and Scheffé's lemma, it follows that we can couple the variables $\widetilde{V}^{(\omega_i)}_1-2\widetilde{P}^{(\omega_i)}_1$ in such a way that they converge to $\widetilde{V}^{(\omega_{\infty})}_1-2\widetilde{P}^{(\omega_{\infty})}_1$ in $L^1$. In particular they are uniformly integrable, which concludes the proof.
\end{proof}

\bibliographystyle{abbrv}
\bibliography{bibli}

\end{document}